\documentclass[a4paper, 12pt]{amsart}
\usepackage{amstext, amssymb,tikz, appendix, mathrsfs, mathabx,stmaryrd,array}
\usetikzlibrary{arrows.meta}
\usepackage{csquotes}
\usepackage[colorlinks=true, breaklinks=true, linkcolor=black, citecolor=blue, urlcolor=red]{hyperref} 
\usepackage[english]{babel}

\numberwithin{equation}{section}
\newtheorem{letterthm}{Theorem}

\newtheorem{letterdef}[letterthm]{Definition}

\newtheorem{theorem}{Theorem}[section]

\newtheorem{lemma}[theorem]{Lemma}
\newtheorem{corollary}[theorem]{Corollary}
\newtheorem{proposition}[theorem]{Proposition}

\newtheorem{observation}[theorem]{Observation}

\theoremstyle{definition} 
\newtheorem{definition}[theorem]{Definition}
\newtheorem{notation}[theorem]{Notation}
\newtheorem{remark}[theorem]{Remark}
\newtheorem{example}[theorem]{Example}

\newcommand{\act}{\curvearrowright}
\DeclareMathOperator{\atom}{atom}
\newcommand{\cA}{\mathcal A}
\newcommand{\al}{\alpha}
\newcommand{\Arg}{{\text{Arg}}}
\newcommand{\cC}{\mathcal C}
\newcommand{\C}{\mathbf C}
\newcommand{\cD}{\mathcal D}
\DeclareMathOperator{\Diag}{Diag}
\DeclareMathOperator{\dimP}{dim_P}
\newcommand{\cF}{\mathcal F}

\DeclareMathOperator{\Frac}{Frac}
\DeclareMathOperator{\full}{full}
\newcommand{\fH}{\mathfrak H}
\newcommand{\scrH}{\mathscr H}
\DeclareMathOperator{\Hilb}{Hilb}
\DeclareMathOperator{\id}{id}
\DeclareMathOperator{\Ind}{Ind}
\DeclareMathOperator{\Irr}{Irr}
\newcommand{\into}{\hookrightarrow}
\newcommand{\fK}{\mathfrak K}
\newcommand{\scrK}{\mathscr K}
\DeclareMathOperator{\Leaf}{Leaf}

\DeclareMathOperator{\Mod}{Mod}
\newcommand{\N}{\mathbf{N}}
\newcommand{\NInd}{\text{Ind-mixing}}
\newcommand{\cO}{\mathcal O}
\newcommand{\ot}{\otimes}
\newcommand{\onto}{\twoheadrightarrow}
\newcommand{\norm}[1]{\lVert#1\rVert}
\newcommand{\cP}{\mathcal P}
\DeclareMathOperator{\PSU}{PSU}
\newcommand{\cQ}{\mathscr Q}
\newcommand{\R}{\mathbf{R}}
\DeclareMathOperator{\Rep}{Rep}
\newcommand{\reps}{F,T,V,\cO}
\DeclareMathOperator{\diff}{diff}
\DeclareMathOperator{\FD}{FD}

\DeclareMathOperator{\res}{res}
\newcommand{\bS}{\mathbf{S}}
\DeclareMathOperator{\smal}{small}
\newcommand{\sub}[1]{\langle#1\rangle}
\DeclareMathOperator{\supp}{supp}
\newcommand{\fT}{\mathfrak T}
\DeclareMathOperator{\TLJ}{TLJ}
\newcommand{\ti}{\tilde}
\DeclareMathOperator{\U}{U}

\DeclareMathOperator{\VN}{VN}

\DeclareMathOperator{\Ver}{Ver}

\newcommand{\fX}{\mathfrak X}
\newcommand{\scrX}{\mathscr X}

\newcommand{\scrY}{\mathscr Y}
\newcommand{\Z}{\mathbf{Z}}
\newcommand{\fZ}{\mathfrak Z}

\providecommand{\keywords}[1]{\textbf{\textit{Index terms---}} #1}

\begin{document}
	
	\title[Irreducible Pythagorean representations]{Irreducible Pythagorean representations of R.~Thompson's groups and of the Cuntz algebra}
	\thanks{
		AB is supported by the Australian Research Council Grant DP200100067.\\
		DW is supported by an Australian Government Research Training Program (RTP) Scholarship.}
	\author{Arnaud Brothier and Dilshan Wijesena}
	\address{Arnaud Brothier\\ School of Mathematics and Statistics, University of New South Wales, Sydney NSW 2052, Australia and
University of Trieste, Via Weiss, 2 34128 Trieste, Italia}
	\email{arnaud.brothier@gmail.com\endgraf
		\url{https://sites.google.com/site/arnaudbrothier/}}
\address{Dilshan Wijesena\\ School of Mathematics and Statistics, University of New South Wales, Sydney NSW 2052, Australia}
	\email{dilshan.wijesena@hotmail.com}
	\maketitle
	
\begin{abstract}
We introduce the Pythagorean dimension: a natural number (or infinity) for all representations of the Cuntz algebra and certain unitary representations of the Richard Thompson groups called Pythagorean.
For each finite Pythagorean dimension $d$ we completely classify (in a functorial manner) all such representations using finite dimensional linear algebra. Their irreducible classes form a nice moduli space: a real manifold of dimension $2d^2+1$.
Apart from a finite disjoint union of circles, each point of the manifold corresponds to an irreducible unitary representation of Thompson's group $F$ (which extends to the other Thompson groups and the Cuntz algebra) that is not monomial. The remaining circles provide monomial representations which we previously fully described and classified.
We translate in our language a large number of previous results in the literature. We explain how our techniques extend them.\end{abstract}
	

\section*{Introduction}

Richard Thompson defined three countable discrete groups $F\subset T\subset V$ in unpublished notes in the mid 1960s.
These groups are famous for witnessing for the first time rare behaviours in group theory, for remaining mysterious after extensive studies, and for naturally appearing in various parts of mathematics, see the expository article \cite{Cannon-Floyd-Parry96}.
For instance $T,V$ were the first examples of infinite simple finitely presented groups (they are in fact of type $F_\infty$) and $F$ was the first example of a torsion-free group of type $F_\infty$ with infinite geometric dimension \cite{Brown-Geoghegan84, Brown87}. 
The group $F$ is still unknown to be amenable (or weakly amenable, exact or sofic) even though $F$ is not elementary amenable and does not contain non-abelian free groups \cite{Brin-Squier85}.
The lack of understanding of these groups comes from the difficulty in constructing actions for them such as (unitary) representations.
Note that an irreducible representation of $F$ is either one dimensional or infinite dimensional, and thus difficult to construct.

In a different context, Dixmier considered in 1964 the universal C*-algebra $\cO:=\cO_2$ generated by $s_0,s_1$ satisfying $s_0s_0^* + s_1s_1^*=1$ and $s_0^*s_0 = s_1^*s_1 = 1$ \cite{Dixmier64}. 
This C*-algebra was deeply studied by Cuntz and is today named after him \cite{Cuntz77}.
The Cuntz algebra was the first example of a separable purely infinite simple C*-algebra. 
It plays a key role in the classification of simple infinite C*-algebras, see the recent survey \cite{Aiello-Conti-Rossi21} and reference therein for extensive details on $\cO$.
It is a simple non-type I C*-algebra. This implies that its spectrum (i.e.~the space of isomorphism classes of irreducible representations of $\cO$) is large and has no obvious topology except the coarse one, see for instance \cite{Glimm61}.
Moreover, the presentation of $\cO$ forces any of its representation to be infinite dimensional and thus not easily constructible. 
However, many families of representations of $\cO$ have been constructed, classified and studied providing powerful applications to (among others) harmonic analysis and wavelet theory, see for instance the introduction \cite{DHJ-atomic15} and the last section of this present article.

Independently, Birget and Nekrashevych found a tight connection between the Thompson group $V$ and the Cuntz algebra $\cO$ \cite{Birget04, Nekrashevych04}.
Indeed, the elements $s_ws_w^*$, where $w$ is any finite binary string, generate the so-called {\it standard Cartan subalgebra} $\cA\subset\cO$ and its normaliser (i.e.~the group of unitaries $g$ of $\cO$ satisfying $g\cA g^*=\cA$) modulo the unitary group of $\cA$ is isomorphic to Thompson's group $V$ via an explicit embedding built from the $s_w$'s.
The embedding $V\into\cO$ has been widely used for constructing unitary representations of $V$, see the last section of this present article.
We will surprisingly see that this process can often be inverted.

In this article we will consider a third object\,---\,the {\it Pythagorean algebra} $\cP$\,---\,which naturally appeared in {\it Jones' technology} (see below). 
Representations of $\cP$ can be promoted into representations of $F,T,V$ and $\cO$ obtaining {\it all} representations of $\cO$. 
The right notion of representation theory for $\cP$ (called P-module) will provide powerful classification tools and irreducibility characterisations for representations of $F,T,V,\cO$. 
In particular, we will see that a large part of the spectrum of $\cO$ decomposes as a disjoint union of smooth manifolds (while the general theory considers this spectrum as a highly pathological space).

\noindent
\underline{\bf Jones' technology.}
While studying subfactors and conformal field theories Vaughan Jones has surprisingly encountered the Thompson group $T$ \cite{Jones17}, see also the survey \cite{Brothier20-survey} and the introduction of \cite{Brothier22-FS}.
From that discovery he has developed an efficient technology for constructing actions of groups $\Frac(\cC)$ that arise as fraction groups of a category $\cC$ (or more precisely as an isotropy group of the groupoid of fractions of $\cC$) \cite{Jones16}.
Actions of $\cC$ (formally functors) provide actions of $\Frac(\cC)$ and in certain situations it is much easier to concretely construct actions of the former.
For instance, $F$ is the fraction group of the category of binary forests $\mathcal F$ which (as a monoidal category) is singly generated by the unique binary tree with 2 leaves.  
Hence, one morphism $Z\to Z\ot Z$ in any monoidal category defines an action of $F$ and if the category is symmetric we can extend this action to the larger groups $T$ and $V$.

By considering the category of Hilbert spaces $(\Hilb,\oplus)$ equipped with the direct sum for monoidal structure and isometries for morphisms it is then natural to introduce the {\it Pythagorean algebra} $\cP$ \cite{Brothier-Jones}.
This is the universal C*-algebra generated by $a,b$ satisfying the so-called {\it Pythagorean equality}: 
\begin{equation}\label{eq:P}\tag{P}
a^*a+b^*b=1.
\end{equation}
Any (unital *-algebra) representation $$\cP\to B(\fH),\ a\mapsto A,b\mapsto B$$ on a complex Hilbert space $\fH$ defines an isometry $\fH\to\fH\oplus\fH$, thus a monoidal functor $\mathcal F\to\Hilb$, and by Jones' technology yields a unitary representation $\sigma:F\act \scrH$ on a larger Hilbert space $\scrH$ that extends canonically to $\sigma^T:T\act \scrH$ and $\sigma^V:V\act \scrH$.
We call them {\it Pythagorean representations} (in short P-representations) of the Thompson groups.
Remarkably, $\sigma$ extends canonically to an action $\sigma^\cO:\cO\act \scrH$ of the Cuntz algebra whose extension is compatible with the Birget--Nekrashevych embedding $V\into \cO$. 
From there, we can deduce that {\it all} representation of $\cO$ come from a representation of $\cP$, see \cite[Proposition 7.1]{Brothier-Jones}. 
In particular, P-representations of the Thompson groups are the representations that extend to $\cO$.
We may now study all representations of $\cO$ (and P-representations of $F,T,V$) using the algebra $\cP$ that has a more basic presentation than $\cO$ and, in particular, admits {\it finite} dimensional representations.

{\bf This phenomena illustrates well the philosophy of Jones' technology: constructing actions of a complicated object like $F,T,V,\cO$ from actions of a more elementary one like $\cP$.}

All these constructions are {\it functorial} providing four functors $\Pi^X:\Rep(\cP)\to \Rep(X)$ with $X=F,T,V,\cO$ and where $\Rep(X)$ denotes the category of {\it unitary} representations of $F,T,V$ and the category of bounded $*$-representations on {\it complex Hilbert spaces} for $\cP$ and $\cO$.
A morphism (or intertwinner) between two representations $\sigma^X:X\act \scrH,\pi^X:X\act \mathscr K$ is a bounded linear map $\theta:\scrH\to\mathscr K$ satisfying that $$\pi^X(g)\circ \theta = \theta \circ \sigma^X(g) \text{ for all } g\in X$$
where $X=F,T,V,\cO,\cP$.
Our aim is to classify these representations of the Thompson groups and the Cuntz algebra up to {\it unitarily conjugacy} (that is, $\pi^X$ is equivalent to $\sigma^X$ when there exists a $\theta$ as above which is a unitary transformation).

\noindent
\underline{\bf Diffuse and atomic representations.}
We have previously introduced the notion of {\it diffuse} and {\it atomic} representations of $\cP$ (in \cite{BW22} and \cite{BW23}, respectively) and their associated representations of $F$ (see Subsection \ref{subsec:diff-atom-rep}).
Moreover, we proved that any P-representation $\sigma:F\act \scrH$ decomposes canonically into a diffuse part and an atomic part: $\sigma_{\atom}\oplus\sigma_{\diff}$. This result extends tautologically to $T,V,\cO$.

For $Y$ in $F,T,V$, we have proved that a P-representation of $Y$ is diffuse if and only if it is {\it Ind-mixing}: it does not contain any representation induced by a finite dimensional one (i.e.~$\Ind_H^Y\theta\not\subset \sigma_{\diff}$ if $H\subset Y$ is a subgroup and $\theta\in \Rep(H)$ is finite dimensional). 
To our knowledge this is the first class of representations of $Y$ that are known to be Ind-mixing and moreover are not {\it mixing} (for their matrix coefficients do not vanish at infinity by \cite[Section 5]{Brothier-Jones}).

In sharp contrast, any atomic representation of $Y$ constructed via a finite dimensional representation of $\cP$ (we say that the representation of $F$ has finite Pythagorean dimension or in short finite P-dimension, see below) decomposes as a finite direct sum of monomial representations $\Ind_{Y_p}^Y\chi$ (with $\chi:Y_p\to \bS_1$) arising from the so-called {\it parabolic subgroups} $Y_p:=\{g\in Y:\ g(p)=p\}$ of $Y$ (for $p$ in the Cantor space $\cC:=\{0,1\}^{\N}$ and for the usual action of $Y$ on it).

\noindent
\underline{\bf Aim of this paper.}
This paper focuses on {\it diffuse} representations of $F,T,V$ and $\cO$. 
These representations are less tractable and more complex than the atomic ones since they are far from being monomials.
Nevertheless, we provide powerful tools for classifying them and checking their irreducibility.
Additionally, we make a synthesis between novel results in the diffuse case and previous results in the atomic case obtaining general statements on representations of the Cuntz algebra and P-representations of the Thompson groups.

\noindent
{\bf The representation theory of $F, T, V$ and $\cO$ that was known before this
paper.}
Most of previous works in the literature has been focused on representations of $\cO$ (see the end of the introduction and the end of the paper for references and more precise statements).
More specifically, on 
\begin{itemize}
\item atomic representations of $\cO$: mostly of finite P-dimension (see below);
\item diffuse representations of $F,T,V$ and $\cO$ of P-dimension one;
\item specific families of diffuse representations of $F$ and $\cO$ of higher P-dimension.
\end{itemize}
Our previous paper \cite{BW23} provided a complete description and classification of all finite P-dimensional atomic representations for the Thompson groups and the Cuntz algebra.

 \noindent
{\bf The representation theory of $F, T, V$ and $\cO$ that is known because of
this paper.}
This paper introduced the Pythagorean dimension for all representations of $\cO$ and all Pythagorean representations of $F,T,V$. 
We determine and classify all the irreducible {\it diffuse} representations of {\it finite} P-dimension for $F,T,V$ and $\cO$. 
Additionally, we show that the representation theories of all {\it diffuse} representations of $F,T,V,\cO$ are identical. Indeed, we have isomorphisms of categories as stated in Theorem \ref{thm:rigidity-category}, item 2.

\noindent
{\bf The representation theory of $F, T, V$ and $\cO$ that is still unknown even
after this paper.}
The classification of diffuse representations of infinite P-dimension (for $F,T,V,\cO$) remains widely open. 
There are some partial results on those in \cite{olesen2016thompson,Bergmann-Conti03}.
In a work in progress we will present a rather extensive classification of {\it atomic} representations of {\it infinite} P-dimension for the Thompson groups $F,T$ and $V$. The $\cO$-case remains less understood.
Some subfamilies of these latter have been considered and classified previously in \cite{Bratteli-Jorgensen-19, DHJ-atomic15, Bergmann-Conti03, Kawamura-05}.
Finally, not much is known (regarding irreducibility and classification) on non-Pythagorean representations of $F,T,V$.

\noindent
{\bf Dimension function.}
We introduce a {\it dimension function} for all representations of $\cO$ and P-representations of $F,T,V$ (see Definition \ref{def:pythag-dim}.)
The definition is rather obvious but seems at a first glance difficult to compute. 
We will provide a practical way to determine it.
Recall that for $X=F,T,V,\cO$ we have the {\it Pythagorean functor} $\Pi^X:\Rep(\cP)\to\Rep(X)$.

\begin{letterdef}\label{ldef:dimension}
Let $\sigma^X$ be a P-representation of $X=F,T,V$ (resp.~\emph{any} representation of $X=\cO$).
The \emph{Pythagorean dimension} $\dim_P(\sigma^X)$ of $\sigma^X$ (in short the P-dimension) is the minimal dimension of $\pi$ for a representation of $\cP$ satisfying that $\sigma$ and $\Pi^X(\pi)$ are unitarily conjugate, in symbols:
$$\dim_P(\sigma^X):=\min(\dim(\pi): \pi\in \Rep(\cP), \Pi^X(\pi)\simeq \sigma^X).$$
We take the convention that $\dim_P(\sigma^X)=\infty$ if there are no finite dimensional representation $\pi$ of $\cP$ satisfying $\Pi^X(\pi)\simeq\sigma^X$.
\end{letterdef}

Note that $\dim_P$ does not depend on $X=F,T,V,\cO$, i.e.~if $\sigma^X\in \Rep(X)$ and $\sigma^\cO\restriction_X=\sigma^X$, then $\dim_P(\sigma^\cO)=\dim_P(\sigma^X)$.
The P-dimension is a new invariant for all the representations of $\cO$ and for the P-representations of $F,T,V$.
We will show that it is a meaningful one: for each value of $d\in\N$ there are continuously many pairwise non-unitarily equivalent {\it irreducible} representations of P-dimension $d$, see Theorems \ref{thm:manifold-diffuse} and \ref{thm:manifold-atomic}.
Hence, many representations of $\cO$ (that are necessarily infinite dimensional) have in fact ``finite dimensional roots''. 
Moreover, we will prove it is additive:
$$\dim_P(\sigma\oplus\pi)=\dim_P(\sigma)+\dim_P(\pi).$$
This will be first proven for the diffuse case in Proposition \ref{prop:p-dim-sum}. Then in Subsection \ref{subsec:atomic-case} it will be explained how the result extends to all P-representations.
The tensor product of P-representations of the Thompson groups is not in general a P-representation.  
Hence, it does not make much sense to ask if $\dim_P$ is multiplicative or not for the usual tensor product.
However, in a future article the authors will introduce a new ``tensor product'' $\boxtimes$ defined for a large class of P-representations and thus representations of $\cO$. This class of P-representations will be closed under $\boxtimes$ and moreover $\dim_P$ will be shown to be multiplicative for $\boxtimes$.

We mainly restrict the analysis of this article to representations with {\it finite} P-dimension where our techniques are the most powerful.
Additionally, we mainly focus on the {\it diffuse} case.
Being diffuse and of finite P-dimension are the two key conditions which will make all our machinery work.
Indeed, under these two assumptions we will concretely show that all intertwinners between P-representations of $\cO$ (and thus of $F,T,V$) arise at the level of $\cP$. 
{\bf This would permit to reduce difficult problems of irreducibility and unitarily conjugacy of infinite dimensional representations into more tractable problems in finite dimensional linear algebra.}

Before stating our first main theorem we introduce {\it Pythagorean modules} (in short P-module) and morphisms between them. 
They form a category $\Mod(\cP)$ with more subtle notions of decompositions and reductions.

\noindent
{\bf Pythagorean module.} 
A P-module is a triple $m:=(A,B,\fH)$ where $\fH$ is a Hilbert space and $A,B\in B(\fH)$ are bounded linear operators from $\fH$ to $\fH$ satisfying the Pythagorean equality \eqref{eq:P}.
So far this is the same information as a representation $\pi$ of $\cP$ where $A=\pi(a),B=\pi(b)$.
Now, a morphism $\theta:m\to m'$ between two P-modules is a bounded linear map $\theta:\fH\to\fH'$ satisfying that $\theta\circ A=A'\circ \theta$ and $\theta\circ B=B'\circ \theta$.
Hence, we do not require that $\theta$ intertwinnes the adjoints which is the key difference.
This provides more morphisms and makes it more difficult for a P-module to be irreducible.

A sub-module of $m$ is then a P-module $m'=(A',B',\fH')$ such that $\fH'\subset \fH$ and the restrictions of $A,B$ to $\fH'$ are $A',B'$, respectively.

\noindent
{\bf Complete, full and residual.}
Given a P-module $(A,B,\fH)$ we may have the decomposition $\fH=\fH_0\oplus\fZ$ where $\fH_0$ is a sub-module but $\fZ$ is only a vector subspace and does not contain any nontrivial sub-module.
In that case we say that $\fH_0$ is a {\it complete sub-module} and $\fZ$ is a {\it residual subspace}.
Moreover, $\fH$ is called {\it full} if $\fZ=\{0\}$ for all such decompositions.

\noindent
{\bf Reducible and decomposable P-modules.}
For unitary representations of groups and *-algebra representations of C*-algebras, irreducibility and indecomposability are the same notions. Although, they differ for P-modules.
A P-module $(A,B,\fH)$ is {\it decomposable} if $\fH=\fH_1\oplus\fH_2\oplus\fZ$ with $\fH_1,\fH_2$ nontrivial sub-modules and is called {\it reducible} if $\fH=\fH_0\oplus\fZ$ with $\fH_0$ a sub-module and $\fZ$ a nonzero vector subspace (possibly residual). It is otherwise indecomposable and irreducible, respectively.

Here is a key technical result which permits to easily compute the P-dimension of a representation. 
We will constantly use it for computing intertwinner spaces.

\begin{letterthm}\label{thm:smallest-module}
Let $(A,B,\fH)$ be a \emph{finite dimensional} P-module and let 
$$\sigma:\cO\act\scrH, s_i\mapsto S_i,i=1,0$$ be the associated representation of the Cuntz algebra $\cO$.
In particular, $(S_0^*,S_1^*,\scrH)$ is a P-module containing $\fH$. 
Define $\mathcal M$ to be the set of all \emph{complete} sub-modules of $\scrH$.

The following assertions are true.
\begin{enumerate}
\item The set $\mathcal M$ admits a smallest element $\fH_{\smal}$ for the inclusion, i.e.~the intersection $\fH_{\smal}$ of all elements of $\mathcal M$ is in $\mathcal M$. In particular, $\fH_{\smal}$ is contained in $\fH$.
\item The P-dimension $\dimP(\scrH)$ of $\scrH$ is equal to the usual dimension $\dim(\fH_{\smal})$ of the vector space $\fH_{\smal}$.
\item In particular, if $(A,B,\fH)$ is \emph{full} then $\fH=\fH_{\smal}$ and $\dim_P(\scrH) = \dim(\fH)$.
\end{enumerate}
\end{letterthm}

In the {\it diffuse} case the theorem is proven in two parts in the main body of the paper. First, the case when $\sigma$ is {\it irreducible} is proven as Theorem \ref{theo:min-invariant-subspace} in Subsection \ref{subsec:smallest-P-mod}. Then the general diffuse case is proven as Corollary \ref{cor:min-invariant-subspace} in Subsection \ref{subsec:ext-reducible-rep}. The reason for this choice of presentation is because the proof of the general case requires Theorem \ref{theo:pythag-irrep} to describe how to decompose $\sigma$ into irreducible components.
The {\it atomic} case is mainly deduced from our previous classification from \cite{BW23}. This is explained in Subsection \ref{subsec:atomic-case}.
Theorem \ref{thm:smallest-module} does {\it not} extend to {\it infinite} P-dimension, see Remark \ref{rem:THB}.
Note that several results in Section \ref{sec:classification} concerning diffuse P-representations are stated for $F$. However, they all readily extend to $T,V,\cO$ by Theorem \ref{theo:isom-diff-functor} (i.e.~Theorem \ref{thm:rigidity-category} item 4).

Hence, for any representation $\sigma:\cO\to B(\scrH)$ of the Cuntz algebra with finite P-dimension we have a canonical subspace $\fH_{\smal}$ given by the last theorem. We show that all intertwinners of the (necessarily infinite dimensional) space $\scrH$ restrict to intertwinners of the finite dimensional subspace $\fH_{\smal}$.
Using again the diffuse assumption we are able to deduce similar statements in the Thompson groups case.

\begin{letterthm}\label{thm:rigidity}
Consider two \emph{diffuse and finite dimensional} P-modules $(A_1,B_1,\fH_1)$ and $(A_2,B_2,\fH_2)$ with associated representations $\sigma^X_1:X\act \scrH_1$ and $\sigma^X_2:X\act \scrH_2$ where $X=F,T,V,\cO$.
Let $\fH_{i,\smal},i=1,2$ be the smallest complete sub-module of $\scrH_i$ from Theorem \ref{thm:smallest-module}. 
The following assertions are true.
\begin{enumerate}
\item The representation $\sigma_1:F\act\scrH_1$ is irreducible if and only if $\fH_{1}$ is indecomposable. 
\item If $\Theta:\scrH_1\to\scrH_2$ intertwinnes the representations of $F$, then $\Theta$ restricts to a P-module morphism $\theta:\fH_{1,\smal}\to\fH_{2,\smal}$. 
In particular, the representations $\sigma_1$ and $\sigma_2$ of $F$ are unitarily equivalent if only if the sub-modules $\fH_{1,\smal}, \fH_{2,\smal}$ are. 
\end{enumerate}
The above statements continue to hold if we replace $F$ by $T,V$ or $\cO$.
\end{letterthm}

Statement $(1)$ of above is proven as Theorem \ref{theo:pythag-irrep} in Subsection \ref{subsec:pyth-irrep-equiv} meanwhile statement $(2)$ is heavily reliant on Theorem \ref{thm:smallest-module} and thus is proven in two steps. The irreducible case is proven as Theorem \ref{theo:pythag-equiv} in Subsection \ref{subsec:pyth-irrep-equiv} and the result is completed by Corollary \ref{cor:intertwinner-rep-to-mod} in Subsection \ref{subsec:ext-reducible-rep}.

Theorem \ref{thm:rigidity} is very surprising and powerful.
Note that both statements are no longer true in the infinite P-dimensional case.
A similar rigidity statement holds in the \emph{atomic} finite P-dimensional case (see Theorem \ref{thm:rigidity-category}). 
Although, this latter result relies on a complete description and classification of atomic P-modules and representations, see Section \ref{subsec:atomic-case}.

Theorem \ref{thm:rigidity} permits to decide easily if representations are irreducible or if two of them are unitarily equivalent or not.
We can then construct huge families of irreducible representations of $F,T,V,\cO$ that are pairwise inequivalent by carefully choosing pairs of matrices and using elementary linear algebra, see Theorem \ref{thm:manifold-diffuse} and Theorem \ref{thm:manifold-atomic}.

\noindent
{\bf Categorical version of Theorem \ref{thm:rigidity}.}
All constructions using Jones' technology are in essence {\it functorial} permitting us to reformulate and strengthen our main results using the language of categories.
Below we combine several results (old and new) to obtain four categorical statements.
Define $\Mod(\cP)$ to be the category of P-modules with morphisms as previously described. We use the superscript $\FD$ and subscripts $\full,\diff$ to write the various full subcategories of $\Mod(\cP)$ or $\Rep(X)$ with $X=F,T,V,\cO$ that we may consider where $\FD$ stands for finite P-dimensional, $\diff$ for diffuse and $\full$ for full. 
Additionally, $\Mod^{FD+}(\cP)$ is the category of finite P-dimensional P-modules $m$ such that $m$ does not contain any atomic P-module of P-dimension $1$ (and similarly we define $\Rep^{\FD+}(F)$).

\begin{letterthm}\label{thm:rigidity-category}
The following assertions are true.
\begin{enumerate}
\item The functor $\Pi^Z:\Mod_{\full}^{\FD}(\cP)\to \Rep^{\FD}(Z)$ 
is an \emph{equivalence of categories} for $Z=T,V,\cO$ (i.e.~for any $\sigma\in \Rep^{\FD}(Z)$ there exists $m\in \Mod_{\full}^{\FD}(\cP)$ so that $\Pi^Z(m)\simeq \sigma$ and $\Pi^Z$ gives bijections on morphism spaces).
\item The functor $\Pi^F:\Mod_{\full}^{\FD+}(\cP)\to \Rep^{\FD+}(F)$ is an \emph{equivalence of categories}.
\item The categories $\Rep^{\FD}(Z)$ with $Z=T,V,\cO$ and $\Rep^{\FD+}(F)$ are all \emph{semisimple}: each representation decomposes into a finite direct sum of irreducible ones which belong to the category.
\item The four categories $\Rep_{\diff}(X)$ with $X=F,T,V,\cO$ are \emph{isomorphic} via the obvious forgetful functors (i.e.~equivalences of categories that are \emph{bijective} on objects).
\end{enumerate}
\end{letterthm}

The first and second statements combine Theorem \ref{thm:rigidity} for the diffuse case (translated in categorical language) and the classification of atomic P-modules and representations in Section \ref{subsec:atomic-case}. 
The third is a consequence of the first two since $\Mod_{\full}^{\FD}(\cP)$ and $\Mod_{\full}^{\FD+}(\cP)$ are semisimple.
The fourth statement comes from a separate argument which, in contrast to the other results, does not require finite P-dimensions, see Theorem \ref{theo:isom-diff-functor}.

One can observe that $\Mod_{\full}(\cP)$ is isomorphic to the category of representations of $\cP$ whose underlying P-modules are full. Although, we keep using the terminology {\it P-module} rather than {\it representation} to insist on the choice of morphisms considered between them.

\noindent
{\bf Classification for each finite P-dimension.}
We fix a (finite) dimension $d\geq 1$ and introduce coordinates by considering all P-modules $(A,B,\C^d)$. Hence, $A,B$ are now $d$ by $d$ complex matrices satisfying $A^*A+B^*B=I_d$. 
Equivalently, this is the set of all $2d$ by $d$ matrices $R$ so that $R^*R=I_d$ which forms a submanifold of $M_{2d,d}(\C)$ of (real) dimension $3d^2$ (to see this consider the smooth map $M_{2d,d}(\C)\to H_d, R\mapsto R^*R$ valued in the hermitian matrices and take the pre-image of $I_d$).
We consider the subset $\Irr(d)$ of all {\it irreducible} P-modules $(A,B,\C^d)$ which is {\it open} and thus forms a submanifold of same dimension.
The irreducibility of the P-modules of $\Irr(d)$ implies  that $\PSU(d)$ (the projective special unitary group of $\C^d$) acts by conjugation {\it freely} on it.

\noindent
{\bf The diffuse picture.}
The subset of all {\it diffuse} irreducible P-modules $(A,B,\C^d)$ is again open in $\Irr(d)$ and thus forms a submanifold $\Irr_{\diff}(d)$ of same dimension $3d^2$ that is stabilised by $\PSU(d)$.
Using Theorem \ref{thm:rigidity-category} (item 1) we conversely deduce that {\it all} irreducible diffuse representations of $X$ (with $X=F,T,V,\cO$) of P-dimension $d$ are constructed from P-modules in $\Irr_{\diff}(d)$.
Moreover, the $\PSU(d)$-orbits classify those representations by Theorem \ref{thm:rigidity}.
Since the action of $\PSU(d)$ is free, proper (by compactness of $\PSU(d)$) and is clearly smooth we can pull back a manifold structure on the orbit space. We deduce the following beautiful classification (see Theorem \ref{theo:diffuse-d}).

\begin{letterthm}\label{thm:manifold-diffuse}
Consider a natural number $d\geq 1$ and let $\Irr_{\diff}(d)$ be the set of irreducible diffuse P-modules $(A,B,\C^d)$.
Fix $X$ being either $F,T,V$ or $\cO$.
The following assertions are true.
\begin{enumerate}
\item The set $\Irr_{\diff}(d)$ forms a smooth submanifold of $M_{2d,d}(\C)$ of real dimension $3d^2$. 
\item The formula 
$$u\cdot (A,B):=(uAu^*,uBu^*)$$ 
defines a free proper smooth action of the Lie group $\PSU(d)$ on $\Irr_{\diff}(d)$.
\item The orbit space $\PSU(d)\backslash \Irr_{\diff}(d)$ has a natural structure of manifold of dimension $2d^2+1$. 
\item If $m\in \Irr_{\diff}(d)$, then $\Pi^X(m)$ is an irreducible diffuse representation of $X$ of P-dimension $d$.
If $m,\ti m\in \Irr_{\diff}(d)$, then $\Pi^X(m)\simeq \Pi^X(\ti m)$ if and only if $m$ and $\ti m$ are in the same $\PSU(d)$-orbit.
\end{enumerate}
\end{letterthm}

For each $d\geq 1$ we have a moduli space of dimension $2d^2+1$ that parametrises the diffuse irreducible classes of representations of $F,T,V,\cO$ of P-dimension $d$.
For $d=1$ we have that $\Irr_{\diff}(1)$ is the 3-sphere minus two subcircles (corresponding to the atomic representations of P-dimension $1$) and $\PSU(1)$ acts trivially on it so their associated representations are pairwise unitarily inequivalent.
For $d=2$ we have that $\Irr_{\diff}(2)$ is a manifold of dimension $12$ and the orbit space of dimension $9$. 

\noindent
{\bf The atomic picture.}
We now describe the atomic picture that is deduced from our previous article \cite{BW23} and is recalled in Section \ref{subsec:atomic-case}.
We include it here for completeness.
Fix $d\geq 2$ (we discard the case $d=1$ for this discussion) and define $W_d$ to be a set of representatives of all prime binary strings of length $d$ modulo cyclic permutations.
Given $w\in W_d$ and a complex number $z\in\bS_1$ of modulus one we define (using matrices) an atomic irreducible P-module $m_{w,z}$ of finite dimension $d$.
Conversely, all atomic irreducible P-module of finite P-dimension is isomorphic to a $m_{w,z}$.
Moreover, $\Pi^Y(m_{w,z})$, where $Y=F,T,V$, is a monomial representation built from the parabolic subgroup $Y_p\subset Y$ (where $p=w^\infty$ is the infinite word of period $w$) and the phase $z$.
This latter is irreducible. 
Moreover, one can show that $m_{w,z}\simeq m_{v,y}$ if and only if $(w,z)=(v,y)$ if and only if $\Pi^X(m_{w,z})\simeq \Pi^X(m_{v,y})$ for $X=F,T,V,\cO$.
We deduce the following.

\begin{letterthm}\label{thm:manifold-atomic}
Consider $d\geq 1$ and $X=F,T,V,\cO$. Write $\Irr_{\atom}(d)$ for the irreducible atomic P-modules of dimension $d$.
The following assertions are true.
\begin{enumerate}
\item The map $\PSU(d)\times W_d\times \bS_1\to \Irr_{\atom}(d),\ (u,w,z)\mapsto u\cdot m_{w,z}$ is an isomorphism of manifolds. Hence, $\Irr_{\atom}(d)$ is a compact manifold of dimension $d^2$.
\item The orbit space $\PSU(d)\backslash \Irr_{\atom}(d)$ is isomorphic to the manifold $W_d\times \bS_1$.
\item If $(X,d)\neq (F,1)$, then $\Pi^X$ restricts to a bijection from $\{m_{w,z}:\ (w,z)\in W_d\times \bS_1\}$ to the isomorphism classes of irreducible atomic P-representations of $X$ of P-dimension $d$.
\end{enumerate}
\end{letterthm}

From this geometric point of view there are much more diffuse representations than atomic ones.

\noindent
\underline{\bf Pythagorean representations among examples in the literature.}
We end the article by describing a number of important constructions of representations of the Thompson groups and of the Cuntz algebra that have been previously classified. 
We list below all these families mostly using their original names:
\begin{itemize}
\item $BJ$: the {\it permutation} representations of Bratteli and Jorgensen of $\cO$ \cite{Bratteli-Jorgensen-19};
\item $DHJ$: the {\it purely atomic} representations of Dutkay, Haussermann and Jorgensen of $\cO$ \cite{DHJ-atomic15};
\item $ABGP$: a specific subclass of permutation representations studied by Barata--Pinto, Ara\'ujo--Pinto and Guimar\~aes--Pinto that they restrict to $F,T,V$ \cite{barata2019representations,Araujo-Pinto22,Guimaraes-Pinto22};
\item $G$ and $O$: some deformations of the Koopman representations of $F,T,V$ from Garncarek and from Olesen, and a subfamily $O_B\subset O$ of {\it Bernoulli} representations \cite{Gar12,olesen2016thompson};
\item $MSW$: the {\it Hausdorff} representations of Mori, Suzuki and Watatani of $\cO$ \cite{MSW-07};
\item $ABC$: the {\it induced product} representations of Aita, Bergmann and Conti of $\cO$ \cite{Aita-Bergmann-Conti97,Bergmann-Conti03};
\item $K$: the {\it generalised permutation} representations of Kawamura of $\cO$ \cite{Kawamura-05};
\item $J$: the {\it Wsyiwyg} representations of Jones of $F$ and $T$ \cite{Jones21};
\item $M$: the monomial representations of $F,T,V$; notably some representations constructed by Golan--Sapir and Aiello--Nagnibeda \cite{Golan-Sapir,Aiello-Nagnibeda22}.
\end{itemize}

Most of the families of representations of above were only defined for the Cuntz algebra. Using the embedding of Birget--Nekrashevych $V\into \cO$ they yield representations of $F,T,V$ that are automatically Pythagorean by \cite[Proposition 7.1]{Brothier-Jones}. 
{\bf Overall, almost all families of irreducible representations of the Thompson groups that we have encountered are Pythagorean.}
Note that it is easier for a representation of $\cO$ to be irreducible than for its restriction to $F$ (resp.~$T,V$) and it is harder for two representations of $\cO$ to be equivalent than for their restriction to $F$ (resp.~$T,V$). Hence, often analysis conducted on representations of $\cO$ cannot always be trivially adapted to $F,T,V$ and vice versa. Nevertheless, our analysis will permit to transfer many results.

For most of the families of above we provide an explicit description of their associated P-modules, determine if they are atomic or diffuse, and compute their P-dimension.
We notice that most examples in the literature are either atomic or are diffuse of dimension one. 
In particular, we obtain the following chain and lattice of inclusions:
$$ABGP\subset BJ\subset DHJ \text{ and } \begin{array}{ccccc}
MSW & \subset & O_B & \subset & \Irr_{\diff}(1)\\
&& \cup&&\\
& & G&&
\end{array}$$
where the family of so-called purely atomic representations $DHJ$ corresponds to what we have called ``atomic representation'' (hence the two terminologies fortunately match) and where $O_B$ is contained in the set of one dimensional diffuse P-representations (which is in bijection with $\Irr_{\diff}(1)$). 
We recover and extend most of the results of these previous articles that concern irreducibility and classification (of course there are other important problems treated in these articles that are not considered by ours). 
Considering all these families in our unified framework permits to answer various open question (like some of the recent survey \cite{Aiello-Conti-Rossi21}).
Very interestingly, some construction and partial classifications of {\it infinite P-dimensional} representations appear in the literature and thus go beyond our results: see \cite{olesen2016thompson, Bergmann-Conti03, Kawamura-05} for the diffuse case and \cite{Bratteli-Jorgensen-19, DHJ-atomic15, Bergmann-Conti03, Kawamura-05} for the atomic cases.

Additionally, we consider the {\it 2-adic ring C*-algebra of the integers} $\cQ$ which is a C*-algebra that contains $\cO$ and whose study was initiated by Larsen and Li \cite{Larsen-Li-12}.
We prove in the diffuse case (in fact, the result holds for an even larger class, see Remark \ref{rem:2-adic}) that the representation theories of $\cO$ and $\cQ$ are identical (see Theorem \ref{theo:O-Q}). 

\begin{letterthm}\label{ltheo:2-adic}
	The categories $\Rep_{\diff}(\cO)$ and $\Rep_{\diff}(\cQ)$ are isomorphic.
\end{letterthm}
This theorem implies partial answers to various open questions regarding $\cQ$.
Finally, we note that none of the representations of $J$ are Pythagorean. 
Moreover, we list explicit examples of irreducible monomial representations that are not built from the parabolic subgroups and thus are not Pythagorean. 

\noindent
\underline{\bf Plan of the article.}
Section \ref{sec:preliminaries} recalls the definitions of the three Thompson groups $F,T,V$, of the Cuntz algebra $\cO$, and the embedding $V\into \cO$ of Birget and Nekrashevych.

Section \ref{sec:pythagorean} introduces the Pythagorean algebra $\cP$ and explains how to promote representations of $\cP$ into representations of $F,T,V$ and additionally to representations of $\cO$. 
We introduce P-modules and various properties of those.
We recall the notion of atomic and diffuse P-modules, and key results from our previous articles.

Section \ref{sec:isomorphism-diffuse} is short but reveals the importance of the notion of {\it diffuseness}. 
We prove that $F$-intertwinners between diffuse representations of $\cO$ are in fact $\cO$-intertwinners providing isomorphisms (not only equivalences) of the full subcategories of diffuse representations. This gives item 4 of Theorem \ref{thm:rigidity-category} which allows us to extend certain subsequent results concerning $\cO$ to also $F,T,V$.

Section \ref{sec:classification} is the most important of the article: we introduce the Pythagorean dimension (Definition \ref{ldef:dimension}) and prove Theorem \ref{thm:smallest-module} (first in the diffuse case).
From there we deduce a practical way for computing the P-dimension and for determining intertwinner spaces of Pythagorean representations.
We complete this result by considering the atomic case and deduce similar statements using our previous classification of \cite{BW23}.
This proves Theorem \ref{thm:rigidity} and items 1,2,3 of Theorem \ref{thm:rigidity-category}.

Section \ref{sec:mani-pyth-rep} interprets our main result for a given P-dimension. Using classical results of differential geometry we deduce Theorems \ref{thm:manifold-diffuse} and \ref{thm:manifold-atomic}. 

Section \ref{sec:examples} confronts our classification and production of irreducible representations with previous work in the literature. Additionally, we briefly present the 2-adic ring C*-algebra of the integers and prove Theorem \ref{ltheo:2-adic}.

\subsection*{Acknowledgments}
The first author warmly thanks Matt Brin, Pierre de la Harpe and Roberto Conti for constructive and pleasant exchanges, and for pointing out relevant references. We also thank Matt Brin for helping to improve a previous version of this manuscript.
Finally, we thank the referee for his comments that have improved the quality of this article.


\section{Preliminaries}\label{sec:preliminaries}
We briefly present the three groups $F,T,V$ of Richard Thompson, the Cuntz algebra $\cO$ and how the firsts embed in the second.
For details we refer to the expository article of Cannon, Floyd and Parry concerning $F,T,V$, the original article of Cuntz and the recent survey of Aiello, Conti and Rossi for $\cO$, and the article of Nekrashevych for the connection between them \cite{Cannon-Floyd-Parry96, Cuntz77, Aiello-Conti-Rossi21, Nekrashevych04}.

\noindent
{\bf Conventions.}
All along the article we only consider Hilbert spaces over the field of complex numbers $\C$. 
A direct sum of Hilbert spaces is always assumed to be orthogonal and denoted $\fH_1\oplus\fH_2$.
All representations of groups are unitary and all representations of C*-algebras are unital and respect the *-structure.

\subsection{The three Thompson groups}
Thompson's group $F$ is the group of piecewise affine homeomorphisms of the unit interval $[0,1]$ with slopes powers of $2$ and having finitely many breakpoints appearing at dyadic rationals $\Z[1/2]$.
Elements of $F$ are completely described by a pair of standard dyadic partitions (in short sdp) that are finite partitions of $[0,1)$ made of standard dyadic intervals (in short sdi) of the form $[\frac{a}{2^n},\frac{a+1}{2^n})$ with $a,n\in \N$ natural numbers (here we follow the convention that $\N$ contains $0$).
Given two sdp's $(I_1,\dots,I_n)$ and $(J_1,\dots,J_n)$ with same number of intervals we consider the unique increasing piecewise affine map sending $I_j$ onto $J_j$ for all $1\leq j\leq n$.
There is a rather obvious bijection between sdp's and rooted finite binary trees (that we simply called trees) where each leaf corresponds to a sdi.
Hence, an element of $F$ is described by a pair of trees $(t,s)$ with same number of leaves. 
Trees are identified with isotopy classes of planar diagrams with roots on top and leaves on bottom.
The correspondence between elements of $F$ and tree-pairs is not one-to-one but can be made so by considering classes $[t,s]$ with respect to the equivalence relation generated by 
$$(t,s)\sim(f\circ t,f\circ s)$$
where $f$ is a forest and $f\circ t$ symbolises the forest $f$ stacked on bottom of $t$ (the $j$-th root of $f$ is lined with the $j$-th leaf of $t$). 

Thompson group $T$ is defined similarly but acting by homeomorphisms on the unit torus $\R/\Z$ or $[0,1]/0\sim 1$.
An element of $T$ is described by a triple $(t,\kappa,s)$ where $(t,s)$ is a pair of trees as above and $\kappa$ is a cyclic permutation of the leaves of $t$ (identified with a bijection from the leaves of $s$ to the leaves of $t$).
The triple $(t,\kappa,s)$ encodes the map that sends in an increasing affine manner $I_j$ onto $J_{\kappa(j)}$ for all $1\leq j\leq n$.
It is in $F$ if and only if $\kappa$ is the identity. 

Thompson group $V$ is defined similarly but allowing any permutation $\kappa$ not just only the cyclic ones. It acts on $[0,1)$ but may have finitely many discontinuities occurring at dyadic rationals.
Although, it acts by homeomorphism on the Cantor set $\cC:=\{0,1\}^{\N^*}$ of infinite binary strings (i.e.~infinite words in the letters $0$ and $1$).
The action consists in changing finite prefixes.
Note that we have the chain of inclusions 
$$F\subset T\subset V.$$
\noindent
{\bf Nota Bene.} We will often identify $\cC$ with the boundary of the rooted infinite complete binary tree $t_\infty$. 
Hence, an infinite binary string is often called a ray.
Further, we identify vertices of $t_\infty$, forming the set $\Ver$, with finite binary string in a similar manner and with leaves of finite rooted binary trees.

Using these identifications it is not hard to express a pair of trees $(t,s)$ as a homeomorphism of the Cantor space.

\subsection{The Cuntz algebra}

\subsubsection{Definition}

The Cuntz algebra is the universal C*-algebra $\cO$ for the presentation:
$$\langle s_0,s_1| s_0s_0^*+s_1s_1^*=s_0^*s_0=s_1^*s_1=1\rangle.$$
It is usually denoted $\cO_2$ (and for any $n\geq 2$ there is a corresponding $\cO_n$ where there are $n$ generators rather than 2) but since we will solely consider $\cO_2$ we will drop the subscript.

Being a C*-algebra means that $\cO$ is an algebra over the field of complex numbers that is equipped with an anti-linear involution $*:\cO\to\cO$ and a sub-multiplicative norm $\|\cdot\|$ satisfying that $\|xx^*\|=\|x\|^2$ for all $x\in \cO$ and for which $*$ is isometric and $\cO$ is complete.

Being a ``universal C*-algebra'' for the presentation of above means that representations of $\cO$ are in one to one correspondence with triple $(S_0,S_1,\fH)$ where $S_0,S_1$ are bounded linear operators acting on a complex Hilbert space $\fH$ and satisfying that 
$$S_0S_0^*+S_1S_1^*=S_0^*S_0=S_1^*S_1=\id_\fH.$$
Equivalently, the representations of $\cO$ are the continuous unital *-algebra morphisms $\cO\to B(\fH)$ where $B(\fH)$ denotes the *-algebra of all bounded linear operators from $\fH$ to $\fH$ equipped with its usual operator norm and the involution $*$ being the adjoint.

If $\mu$ is a finite binary string of $0,1$, then we write $s_\mu$ for the corresponding product respecting the order of the letters (hence if $\mu=011$, then $s_\mu=s_0s_1s_1$).
If $s_i\mapsto S_i,i=0,1$ is a representation, then we write $S_\mu$ for the image of $s_\mu$.

\subsubsection{The standard Cartan subalgebra}
Consider the C*-subalgebra $\cA$ of $\cO$ equal to the completion of the span of the set of projections $s_\mu s_\mu^*$ where $\mu$ is any finite word.
Observe that $\cA$ is abelian and in fact is a {\it maximal} abelian subalgebra (any intermediate algebra between $\cA$ and $\cO$ is either non-abelian or equal to $\cA$).
Moreover, it is a {\it regular} subalgebra of $\cO$, i.e.~the normaliser 
$$N_\cO(\cA):=\{ g\in \cO:\ g\cA g^*=\cA, gg^*=g^*g=1\}$$
generates $\cO$ as a C*-algebra, see for instance \cite{Aiello-Conti-Rossi18} for a proof. 
Hence, by definition, $\cA$ is a {\it Cartan subalgebra} of $\cO$ that we call {\it the standard Cartan subalgebra} of $\cO$.

\subsection{Thompson's group $V$ inside the Cuntz algebra}\label{sec:V-inside-Cuntz}
We consider $V$ as a group of homeomorphisms of the Cantor space $\cC:=\{0,1\}^{\N^*}$.
An element $g\in V$ is defined by two lists $(\nu_1,\cdots,\nu_m)$ and $(\mu_1,\cdots,\mu_m)$ of finite binary strings satisfying that
$$g: \mu_j\cdot w\mapsto \nu_j\cdot w \text{ for all } 1\leq j\leq m \text{ and infinite word } w.$$
Necessarily, each list defines a standard partition of $\cC$.
Birget and Nekrashevych made the fundamental observation that $V$ can be embedded in the unitary group of $\cO$ \cite{Birget04,Nekrashevych04}.
More precisely, the map 
$$\iota:V\to \cO, g\mapsto \sum_{j=1}^m s_{\nu_j} s_{\mu_j}^*$$
realises a group isomorphism from $V$ to $N_{\cO}(\cA).$

\begin{center}{\bf From now on we identify $V$ with its image inside $\cO$.}\end{center}

\section{The Pythagorean algebra and Pythagorean representations}\label{sec:pythagorean}

We recall various definitions and results concerning the Pythagorean algebra and the related Pythagorean  representations. Additionally, we introduce the novel notion of Pythagorean modules. 

\subsection{The Pythagorean algebra}

\subsubsection{A universal C*-algebra}

Define $\cP$ to be the universal C*-algebra generated by two elements $a,b$ satisfying the so-called {\it Pythagorean equality}:
$$a^*a+b^*b=1.$$
We call $\cP$ the {\it Pythagorean algebra} (in short P-algebra).

\subsubsection{Representations of the C*-algebra $\cP$}

As previously explained for $\cO$, the representations of $\cP$ are the unital *-algebra morphisms $\pi:\cP\to B(\fH)$ which corresponds to pairs $(A,B)$ of $B(\fH)$ satisfying $A^*A+B^*B=\id_\fH$.
The class of all representations of $\cP$ forms a category $\Rep(\cP)$ whose morphisms are the {\it intertwinners}, i.e.~if $(\pi,\fH)$ and $(\pi',\fH')$ are representations of $\cP$, then a morphism from $\pi$ to $\pi'$ is a bounded linear map $\theta:\fH\to\fH'$ satisfying that 
\begin{equation}\label{eq:intert}\pi'(x)\circ \theta = \theta\circ \pi(x) \text{ for all } x\in \cP.\end{equation}
Equivalently, a bounded linear map $\theta:\fH\to\fH'$ is an intertwinner when the equation of above is satisfied for $x=a,b,a^*$ and $b^*$.

\subsubsection{Pythagorean module}\label{sec:P-module}

For our purpose we desire having more intertwinners by only checking the intertwinning conditions for $a$ and $b$, and not for their adjoints $a^*,b^*$.
Moreover, we want to consider subspaces of $\fH$ that are closed under only $A,B$ rather than closed under the whole C*-algebra $\cP$ (that is closed under $A,B,A^*$ and $B^*$).
In order to do that we introduce {\it Pythagorean modules}. 

\noindent
{\bf P-module.}
A {\it Pythagorean module} (in short P-module) is a triple $m=(A,B,\fH)$ where $\fH$ is a Hilbert space, $A,B\in B(\fH)$ are bounded linear operators acting on $\fH$ satisfying that $A^*A+B^*B=\id_\fH$. Hence, it is the same data as a representation of $\cP$.

\noindent
{\bf Intertwinner and equivalence.}
An {\it intertwinner} or a {\it morphism} from the P-module $m=(A,B,\fH)$ to the P-module $m'=(A',B',\fH')$ is a bounded linear map $\theta:\fH\to\fH'$ satisfying that 
$$\theta\circ A=A'\circ \theta \text{ and } \theta\circ B=B'\circ \theta.$$
We may write $\theta:(A,B,\fH)\to (A',B',\fH')$ or $\theta:m\to m'$ to express that $\theta$ is an intertwinner of P-modules (and not only a bounded linear map).
Two P-modules $(A,B,\fH)$, $(A',B',\fH')$ are {\it unitarily equivalent} or {\it isomorphic} if there exists a unitary intertwinner $U : (A,B,\fH)\to (A',B',\fH')$. 
We may drop the term ``unitarily''.

\noindent
{\bf Sub-module.}
A {\it sub-module} of $(A,B,\fH)$ is a Hilbert subspace $\fH'\subset\fH$ closed under the action of $A$ and $B$. We equip $\fH'$ with the P-module structure obtained by restricting $A$ and $B$ to $\fH'$.

\noindent
{\bf Full and complete sub-modules and residual subspace.}
Consider a P-module $(A,B,\fH)$ and decompose $\fH$ as the (orthogonal) direct sum $\fH_0\oplus\fZ$ assuming that $\fH_0$ is a sub-module and $\fZ$ is a vector subspace which does not contain any nontrivial sub-module.
We say that 
\begin{itemize}
\item $\fH_0$ is a {\it complete sub-module} of $\fH$ and 
\item $\fZ$ is a {\it residual subspace} of $\fH$.
\end{itemize}
If for all such decompositions we have that $\fZ=\{0\}$, then we say that $\fH$ is {\it full}.
Hence, a P-module is full if and only if it does not contain any proper complete sub-modules (equivalently, it does not contain any nontrivial residual subspace).

\noindent
{\bf Decomposability and reducibility.}
A P-module $(A,B,\fH)$ is {\it decomposable} if it {\it contains} the direct sum of two nontrivial ones.
It is indecomposable otherwise.
A P-module $(A,B,\fH)$ is {\it reducible} if it admits a nontrivial sub-module and irreducible otherwise.

\noindent
{\bf Category of P-modules.}
We write $\Mod(\cP)$ for the category of P-module with intertwinners for morphisms. 

\begin{remark}
\begin{enumerate}
\item We have defined $m=(A,B,\fH)$ to be decomposable when $m_1\oplus m_2\subset m$ with $m_i$ nontrivial sub-modules. This is strictly weaker than asking $m=n_1\oplus n_2$ with $n_i$ nontrivial sub-modules.
Indeed, consider the example with $\fH = \C^3$ so that $Be_1 = e_1$, $Ae_2 = e_2$, $Ae_3 = 1\sqrt{2}e_1$ and $Be_3 = 1/\sqrt{2}e_2$ (and necessarily $Ae_1=0=Be_2$) where $(e_1,e_2,e_3)$ is the standard basis. 
It is decomposable in our sense since $\C e_1$ and $\C e_2$ are sub-modules. Moreover, note that $\C e_3$ is a residual subspace.
However, $\fH$ is {\it not equal} to the direct sum of two nontrivial sub-modules since $\C e_1\oplus \C e_3$ and $\C e_2\oplus \C e_3$ are not sub-modules.
\item Note that $\Rep(\cP)$ and $\Mod(\cP)$ have same class of objects but the morphism spaces of $\Mod(\cP)$ are larger than the one of $\Rep(\cP)$. 
The example of the previous item yields an irreducible representation of $\Rep(\cP)$ but a reducible (and even decomposable) P-module.
Although, two P-modules are unitarily equivalent if and only if their associated representations in $\Rep(\cP)$ are unitarily equivalent.
\item Unlike for representations of $\cP$, a P-module may be indecomposable but reducible.
For example, take $\fH = \C^2$ and let $Ae_1=Ae_2 = e_1$ and (necessarily) $B = 0$. Then $(A,B,\C^2)$ forms a P-module which reduces to $\C e_1 \oplus \C e_2$ where $\C e_1$ is a sub-module and $\C e_2$ a residual subspace.
We will see later that the {\it Pythagorean dimension} of this module (or its associated representation of $F$) is equal to $1$ (and not equal to the usual dimension of $\fH$).
\item If we restrict to {\it separable} P-modules (i.e.~$\fH$ is a separable Hilbert space), then since $\cP$ is separable we still obtain all the {\it irreducible} P-modules. Moreover, the category becomes small, i.e.~its collections of objects and morphisms are sets.
This permits to do all our analysis using only set theory. 
\end{enumerate}
\end{remark}

\subsection{Pythagorean representations}\label{sec:P-rep}
We briefly recall what are Pythagorean representations and refer to \cite{Brothier-Jones, BW22,BW23} for extensive details. 
Additionally, the reader may find it instructive to compare this construction with the one of \cite{Brothier-group1,Brothier-group2} where Hilbert spaces are replaced by discrete groups.

\subsubsection{A larger Hilbert space}
Consider a P-module $(A,B,\fH)$.
For any (finite rooted binary complete) tree $t$ with set of leaves $\Leaf(t)$ we consider the Hilbert space $\fH_t$ of all maps from $\Leaf(t)$ to $\fH$ equipped with the inner product
$$\langle \xi,\eta\rangle:=\sum_{\ell\in\Leaf(t)}\langle \xi_\ell,\eta_{\ell}\rangle,\ \xi,\eta\in\fH_t.$$
We often interpret an element of $\fH_t$ as the tree $t$ whose leaves are decorated by elements of $\fH$.
Moreover, we may write $\xi_j$ instead of $\xi_\ell$ when $\ell$ is the $j$-th leaf of $t$ counting from left to right and starting at $1$.
Consider a larger tree $s=f\circ t$ where $f$ is a forest.
Now, $(A,B)$ defines a linear isometry $\iota_t^s:\fH_t\to\fH_{f\circ t}$ as follows.
If $f$ has a unique caret (we say that $f$ is an elementary forest) at its $j$-th root, then 
$$\iota_t^s(\xi)_j:=A\xi_j \text{ and } \iota_t^s(\xi)_{j+1}=B\xi_j$$ and the other component of the vector are unchanged up to an obvious shift of indices.
For instance:
\begin{center}
	\begin{tikzpicture}[baseline=0cm, scale = 1]
		\draw (0,0)--(-.5, -.5);
		\draw (0,0)--(.5, -.5);		
		\node[label={[yshift=-22pt] \normalsize $\xi_1$}] at (-.5, -.5) {};
		\node[label={[yshift=-22pt] \normalsize $\xi_2$}] at (.5, -.5) {};
		\node[label={[yshift= -3pt] \normalsize $\iota_t^s$}] at (1.2, -.7) {$\mapsto$};
	\end{tikzpicture}%
	\begin{tikzpicture}[baseline=0cm, scale = 1]
		\draw (0,0)--(-.5, -.5);
		\draw (0,0)--(.5, -.5);
		\draw (-.5, -.5)--(-.9, -1);
		\draw (-.5, -.5)--(-.1, -1);
		\node[label={[yshift=-22pt] \normalsize $A\xi_1$}] at (-.9, -1) {};
		\node[label={[yshift=-22pt] \normalsize $B\xi_1$}] at (-.1, -1) {};
		\node[label={[yshift=-22pt] \normalsize $\xi_2$}] at (.5, -.5) {};		
	\end{tikzpicture}%
\end{center}
where $t = \wedge$ and $s = f\circ t$ with $f= \wedge \otimes \vert$ (the symbol $\vert$ denotes the trivial forest, $\wedge$ the unique tree with 2 leaves and $\otimes$ is the horizontal concatenation of forests).
For general forest $f$ (having possibly more than one caret) we decompose $f$ as a finite composition of elementary ones and define $\iota_t^{f\circ t}$ to be the corresponding finite composition of isometries. The decomposition of the forest $f$ may not be unique but each of them yields the same isometry.

Write $\fT$ for the set of all trees and write $\leq$ for the pre-order such that $t\leq s$ if $t$ is a rooted subtree of $s$.
We have defined a direct system of Hilbert spaces 
$$(\fH_t, \iota_t^s:\ t\leq s\text{ trees })$$
where the connecting maps are linear isometries and whose limit $\scrK=\scrK_{A,B}$ is a preHilbert space that we complete into the Hilbert space $$\scrH=\scrH_{A,B}=\varinjlim_{t\in\fT}\fH_t.$$
An element of $\fH_t$ is often written $(t,\xi)$ rather than $\xi$ in order to insist on the tree $t$.
Its class inside $\scrH$ is written as $[t,\xi]$.
Note that 
$$\fH_t\into \scrH, (t,\xi)\mapsto [t,\xi]$$ 
is an isometry and we identify $\fH_t$ with its image inside $\scrH$.
If $I$ is the tree with one leaf (the trivial tree), then note that $\fH_I$ is isomorphic to $\fH$ in an obvious way.
{\bf We will constantly identify $\fH_t$ with its copy inside $\scrH$ and $\fH$ with $\fH_I$.}

\begin{notation}
Summarising notations and choices of mathematical symbols:
we often write $(A,B,\fH)$ for a P-module, $t$ for a tree, $\fH_t$ for the Hilbert space of all maps from $\Leaf(t)$ to $\fH$, $\scrK$ for the preHilbert space equal to the inductive limit of the $\fH_t$ and $\scrH$ for the Hilbert space equal to the completion of $\scrK$.
\end{notation}

\begin{remark}
Note that any forest $f$ defines a map $\iota_t^{f\circ t}$ for $t$ a tree that is composable with $f$ (i.e.~the number of leaves of $t$ is equal to the number of roots of $f$).
We often write $$(t,\xi)\mapsto (f\circ t, \Phi_{A,B}(f)(\xi))$$ for the isometry $\iota_t^{f\circ t}:\fH_t\to\fH_{f\circ t}$.
Hence, $\Phi_{A,B}(f)$ is an isometry from $\fH^n$ to $\fH^m$ where $n$ is the number of roots of $f$ and $m$ its number of leaves.
In fact, $$\Phi=\Phi_{A,B}:\cF\to\Hilb$$ defines a covariant functor from the category of binary forests to the category of Hilbert space where this later has isometries for morphisms.
Better, this functor is {\it monoidal} with respect to the horizontal concatenation of forests and the direct sum of Hilbert spaces.
\end{remark}

\subsubsection{Partial isometries}\label{sec:partial-isometries}
Our analysis of Pythagorean representations exploits the very useful family of partial isometries $\tau_\nu \in B(\scrH)$ where $\nu$ runs over the vertices of the infinite rooted binary regular tree or equivalently the finite binary strings.
We shall be brief in their description and refer the reader to \cite{BW22} (see Subsection 2.1) where these partial isometries were initially defined.	 
Informally, $\tau_\nu$ can be described as follows.

Fix a vertex $\nu$ and consider an element $[t, \xi] \in \scrK$. By taking representatives of $[t, \xi]$ we can ensure that $\nu$ is a vertex of $t$ up to taking a larger tree. Define $t_\nu$ to be the sub-tree of $t$ with root $\nu$ and whose leaves coincide with the leaves of $t$. Then set $\tau_\nu([t, \xi]) := [t_\nu, \eta]$ where $\eta$ is the decoration of the leaves of $t$ in $[t, \xi]$ that coincide with the leaves of $t_\nu$. The map $\tau_\nu$ is well-defined and extends to a surjective partial isometry from $\scrH$ onto itself. 
Hence, its adjoint $\tau_\nu^*$ is an isometry and more generally we shall consider the partial isometries $\tau^*_\nu\tau_\omega$.

These maps can be diagrammatically described. For example, the map $\tau_\nu$ ``grows'' the tree sufficiently large and then ``snips'' the tree at $\nu$.  To illustrate, below demonstrates the computation of $\tau_{01}([\wedge, (\xi_1, \xi_2)])$.

\begin{center}
	\begin{tikzpicture}[baseline=0cm, scale = 1]
		\draw (0,0)--(-.5, -.5);
		\draw (0,0)--(.5, -.5);
		
		\node[label={[yshift=-22pt] \normalsize $\xi_1$}] at (-.5, -.5) {};
		\node[label={[yshift=-22pt] \normalsize $\xi_2$}] at (.5, -.5) {};
		
		\node[label={[yshift= -3pt] \normalsize $\Phi_{A,B}(f_1)$}] at (1.7, -.7) {$\sim$};
	\end{tikzpicture}%
	\begin{tikzpicture}[baseline=0cm, scale = 1]
		\draw (0,0)--(-.5, -.5);
		\draw (0,0)--(.5, -.5);
		\draw (-.5, -.5)--(-.9, -1);
		\draw (-.5, -.5)--(-.1, -1);
		
		\node[label={[yshift=-22pt] \normalsize $A\xi_1$}] at (-.9, -1) {};
		\node[label={[yshift=-22pt] \normalsize $B\xi_1$}] at (-.1, -1) {};
		\node[label={[yshift=-22pt] \normalsize $\xi_2$}] at (.5, -.5) {};
		
		\node[label={[yshift= -3pt] \normalsize $\tau_{01}$}] at (1.4, -.7) {$\longmapsto$};
	\end{tikzpicture}%
	\begin{tikzpicture}[baseline=0cm, scale = 1]
		\node[label={[yshift=-25pt] \normalsize $B\xi_1$}] at (0, -.5) {$\bullet$};
	\end{tikzpicture}.%
\end{center}

More generally, the map $\tau^*_\nu\tau_\omega$ ``snips'' the tree at $\omega$ and then ``attaches'' the resulting sub-tree, along with its components, at the vertex $\nu$ while setting all other components to zero. A computation of $\tau^*_0\tau_1$ is shown below.
	\begin{center}
		\begin{tikzpicture}[baseline=0cm, scale = 1]
			\draw (0,0)--(-.5, -.5);
			\draw (0,0)--(.5, -.5);
			\draw[thick] (.5, -.5)--(.1, -1);
			\draw[thick] (.5, -.5)--(.9, -1);
			
			\node[label={[yshift=-22pt] \normalsize $\xi_1$}] at (-.5, -.5) {};
			\node[label={[yshift=-22pt] \normalsize $\xi_2$}] at (.1, -1) {};
			\node[label={[yshift=-22pt] \normalsize $\xi_3$}] at (.9, -1) {};
			
			\node[label={[yshift= -3pt] \normalsize $\tau_1$}] at (1.7, -.7) {$\longmapsto$};
		\end{tikzpicture}%
		\begin{tikzpicture}[baseline=0cm, scale = 1]
			\draw[thick] (0,-.3)--(-.5, -.8);
			\draw[thick] (0,-.3)--(.5, -.8);
			
			\node[label={[yshift=-22pt] \normalsize $\xi_2$}] at (-.5, -.8) {};
			\node[label={[yshift=-22pt] \normalsize $\xi_3$}] at (.5, -.8) {};
			
			\node[label={[yshift= -3pt] \normalsize $\tau^*_{0}$}] at (1.4, -.7) {$\longmapsto$};
		\end{tikzpicture}%
		\begin{tikzpicture}[baseline=0cm, scale = 1]
			\draw (0,0)--(-.5, -.5);
			\draw (0,0)--(.5, -.5);
			\draw[thick] (-.5, -.5)--(-.9, -1);
			\draw[thick] (-.5, -.5)--(-.1, -1);
			
			\node[label={[yshift=-22pt] \normalsize $\xi_2$}] at (-.9, -1) {};
			\node[label={[yshift=-22pt] \normalsize $\xi_3$}] at (-.1, -1) {};
			\node[label={[yshift=-22pt] \normalsize $0$}] at (.5, -.5) {};		
		\end{tikzpicture}%
	\end{center}
\begin{notation} We define the orthogonal projection
\[\rho_\nu := \tau^*_\nu\tau_\nu\]
which sets all components that are not children of $\nu$ to zero. 
\end{notation}
Hence, if $\{\nu_i\}_{i=1}^n$ is a sdp then $\sum_{i=1}^n\rho_{\nu_i} = \id$. This gives the following easy but important observation.

\begin{observation} \label{obs:scrH-P-module}
	The triplet $(\tau_0, \tau_1, \scrH)$ forms a P-module since we have
	\[\tau^*_0\tau_0 + \tau^*_1\tau_1 = \rho_0 + \rho_1 = \id.\]
	Further, observe that if $\xi \in \fH$ then
	\[\tau_0(\xi) = A\xi \text{ and } \tau_1(\xi) = B\xi.\]
	Hence, $(A,B,\fH)$ is a sub-module of $(\tau_0, \tau_1, \scrH)$. 
\end{observation}

\subsubsection{Actions of the Thompson groups and of the Cuntz algebra} \label{subsubsec:pyth-rep}

Using the P-module $(A,B,\fH)$ and the associated larger Hilbert space $\scrH$ we will define representations of the groups $F,T,V$ and of the Cuntz algebra $\cO$.

\noindent
{\bf Actions of Thompson's group $F$.}
Consider $g\in F$ and $\xi\in\scrK$. There exists a pair of trees $(t,s)$ with same number of leaves such that $g=[s,t]$.
Recall that $\scrK$ is the dense subspace of $\scrH$ equal to the union of all the $\fH_z$ with $z$ a tree.
Hence, there exists a tree $z$ so that $\xi\in\fH_z$. We write $\xi=[z,\eta]$ to emphasize the tree $z$ where $\eta:\Leaf(z)\to \fH$ is a map.
There exists some forests $p,q$ satisfying that $p\circ z=q\circ t$ (this is the so-called left-Ore condition).
In particular, $$\xi=[z,\eta]=[p\circ z, \Phi(p)(\eta)] \text{ and } g=[s,t] = [q\circ s,q\circ t].$$
We set 
$$g\cdot \xi:= [q\circ s, \Phi(p)(\eta)].$$
Using that the composition of forests is a cancellative operation (in fact it suffices to be weakly right-cancellative) we easily deduce that this operation is well-defined (it does not depend on the choice of $(s,t)$ nor of $z$ nor of $(p,q)$) and defines an action of $F$ on $\scrK$.
This actions uniquely extends into a unitary representation
$$\sigma=\sigma_{A,B}:F\to \U(\scrH)$$
on the completion $\scrH$ of $\scrK$.

\begin{definition}
The map $\sigma$ is the \emph{Jones representation} associated to the functor $\Phi$ also called the \emph{Pythagorean} representation (in short P-representation) associated to $(A,B,\fH)$.
\end{definition}

\noindent
{\bf Extension to Thompson's groups $T$ and $V$.}
An element of $F$ is encoded by a (class of) pair of trees $[t,s]$ with same number of leaves.
To define an element of $T$ (resp.~$V$), add in this data a cyclic permutation (resp.~a permutation) of the leaves such as:
$$[\kappa\circ t, s]$$
with $t,s$ trees and $\kappa$ a permutation.
To canonically extend the action $\sigma$ of $F$ we encode $\kappa$ as a permutation of the vector's coordinates.
If we consider $g=[\kappa\circ t, s]\in V$ and $[s,\eta]$ is a vector of $\fH_s$, then the formula for the action becomes:
$$g\cdot [s,\eta] = [\kappa\circ t, s]\cdot [s, \eta] = [t,\zeta]$$
where the $j$-th coordinate of $\zeta$ is equal to the $\kappa(j)$-th coordinate of $\eta$.
We refer the reader to \cite[Section 2.2]{Brothier23-Haagerup} for details (note that in the first author's previous article the usual tensor product of Hilbert spaces is used rather than the direct sum).
We write $\sigma^T$ and $\sigma^V$ for the canonical extensions of $\sigma$ to the larger groups $T$ and $V$, respectively.

\noindent
{\bf Action of the Cuntz algebra $\cO$.}
The extension of $\sigma$ to the Cuntz algebra $\cO$ is even more obvious to define from our point of view.
Recall that $s_0,s_1$ are the usual generators of $\cO$ and $\tau_0,\tau_1$ are the surjective partial isometries of $\scrH$ defines in Section \ref{sec:partial-isometries}.
We have that $$\sigma^\cO:\cO\to B(\scrH), s_0\mapsto \tau_0^*, s_1\mapsto \tau_1^*$$
defines a representation of $\cO$.

\noindent
{\bf Compatibility of the actions.}
From the definition of the embedding of Birget and Nekrashevych $\iota:V\to\cO$ that we recall in Section \ref{sec:V-inside-Cuntz} we immediately deduce that the restriction of $\sigma^\cO$ to $V$ is equal to $\sigma^V$.
Hence, the three Thompson groups $F,T,V$ and the Cuntz algebra $\cO$ all act on the {\it same} Hilbert space $\scrH$ and these actions are built canonically from the P-module $(A,B,\fH)$.
When the context is clear we may suppress the superscript of $\sigma$.

\noindent
{\bf Subrepresentations of Pythagorean representations.}
Recall from Observation \ref{obs:scrH-P-module} that $(\tau_0, \tau_1, \scrH)$ is a P-module. Any sub-module of $\scrH$ can be canonically associated with a subrepresentation of $\sigma^X$ in the following manner with $X=F,T,V,\cO$.

\begin{definition} \label{def:mod-sub-rep}
For a sub-module $\scrX \subset \scrH$ define the closed subspace
	\[\sub{\scrX} := \overline{\textrm{span}}\{\cup_{\nu \in \Ver} \tau^*_\nu(\scrX)\} \subset \scrH.\]
This is the smallest topologically closed subspace of $\scrH$ that contains $\scrX$ and is closed under the operations $\tau_0^*$ and $\tau_1^*$.
Intuitively, $\sub{\scrX}$ can be considered to be the subspace formed by decorating the leaves of all trees with elements in $\scrX$ and taking the topological closure.
The space $\sub{\scrX}$ is closed under $\tau_0^*$, $\tau_1^*$ by definition and continuity, and is also closed under $\tau_0$ and $\tau_1$.
Therefore, $\sub{\scrX}$ defines a representation of $\cO$ (given by $s_i\mapsto \tau_i^*\restriction_{\sub{X}}, i=0,1$).
This induces a subrepresentation $\sigma^X_\scrX \subset \sigma^X$ for $X=F,T,V,\cO$. 
We may say that $\sub{\scrX}$ is {\it generated} by $\scrX$.
\end{definition}

\begin{observation} \label{obs:tau-mod-sub-rep}
	Here is an obvious but crucial observation: if $\scrX\subset\scrH$ is a sub-module, then $(\tau_0\restriction_\scrX, \tau_1\restriction_\scrX,\scrX)$ forms a P-module whose associated P-representation is unitarily equivalent to $\sigma_\scrX$. 
	We will freely identify these two representations.
\end{observation}

\begin{lemma} \label{lem:complete-mod}
Let $(A,B,\fH)$ be a P-module with associated P-representation $\sigma:\cO\act\scrH$ and partial isometries $\tau_i:=\sigma(s_i^*), i=0,1$.
Equip $\scrH$ with the P-module structure $(\tau_0,\tau_1,\scrH)$.
If $\scrX\subset\scrH$ is a sub-module, then it is complete if and only if $\sub{\scrX} = \scrH$ (equivalently, $\sigma_\scrX = \sigma$).
\end{lemma}

\begin{proof}
Consider a complete sub-module $\scrX\subset\scrH$.
Observe that $\sub{\scrX}\oplus \sub{\scrX}^\perp$ defines a decomposition of the representation of $\sigma$ of $\cO$.
In particular, $\sub{\scrX}^\perp$ is a sub-module of $(\tau_0,\tau_1,\scrH)$ and thus must be trivial since $\scrX$ is complete. This proves the obvious direction.
	
	To prove the converse we shall instead prove the contrapositive. Let $\scrX$ be a sub-module of $\scrH$ that is not complete and thus there exists a nontrivial sub-module $\scrY \subset \scrX^\perp$. We shall claim that $\sub{\scrX}$ and $\sub{\scrY}$ are orthogonal subspaces. Indeed, it is sufficient to show that 
	\[\langle \tau^*_\nu(x), \tau^*_\omega(y) \rangle = 0\]
	for any two elements $x \in \scrX, y \in \scrY$ and any two vertices $\nu, \omega$. If $\nu, \omega$ are disjoint vertices (that is, their associated sdi's are disjoint) then the above immediately follows. Hence, without loss of generality, assume that $\nu$ is either a child of $\omega$ or $\nu = \omega$. Let $\mu$ be the vertex such that $\omega \cdot \mu = \nu$ where $\cdot$ is the concatenation of sequences. We have:
	\[\langle \tau^*_\nu(x), \tau^*_\omega(y) \rangle = \langle \tau^*_\omega\tau^*_\mu(x), \tau^*_\omega(y) \rangle = \langle x, \tau_\mu(y) \rangle = 0\]
	where the last equality follows because $\scrY$ is a sub-module of $\scrH$ and thus $\tau_\mu(y) \in \scrY$. Hence, we have shown that $\sub{\scrX}$ and $\sub{\scrY}$ are orthogonal nontrivial subspaces and thus $\sub{\scrX}$ is a proper subspace of $\scrH$. This completes the proof.
\end{proof}

\begin{remark} \label{rem:complete-sub-rep}
	This shows that if $(A,B,\fH)$ is a P-module with associated representation $\cO\act\scrH$, then the subspace $\fH\subset\scrH$ is always a complete sub-module of $\scrH$. 
	Furthermore, suppose $\fK$ is a complete sub-module of $\fH$ and $\fH$ is finite dimensional. Then from \cite{BW23} we have $\sub{\fK} = \scrH$ and thus $\fK$ is also a complete sub-module of $\scrH$. 
\end{remark}

\subsubsection{Direct sum}
Observe that if $(A,B,\fH)$ and $(A',B',\fH')$ are P-modules with associated representations $\sigma,\sigma'$ acting on $\scrH,\scrH'$, respectively, then $(A\oplus A',B\oplus B',\fH\oplus\fH')$ is a P-module with associated representation $\sigma\oplus\sigma'$ acting on $\scrH\oplus \scrH'$.
Hence, the Jones construction is compatible with direct sums.
This extends to infinite direct sums and even direct integrals in the obvious manner.

\subsubsection{Functoriality} \label{subsec:p-rep-functors}

Fix two P-modules $m=(A,B,\fH)$ and $m'=(A',B',\fH')$ and consider their associated representations $\sigma:\cO\act \scrH, \sigma':\cO\act\scrH'$.
Let $\theta:\fH\to\fH'$ be an intertwinner.
If $t$ is a tree with say $n$ leaves we define the map 
$$\theta_t:\fH_t\to \fH_t', [t,\eta_1,\cdots,\eta_n]\mapsto [t,\theta(\eta_1),\cdots,\theta(\eta_n)].$$
This is a bounded linear operator of same norm as $\theta$.
Since $\theta$ is an intertwinner the collection $(\theta_t,\ t\in\fT)$ indexed by all the tree is compatible with the directed system structures and defines a bounded linear map from $\scrK$ to $\scrK'$ which uniquely extends into
$$\Theta:\scrH\to\scrH'.$$
Diagrammatically, this map works as follows.
\begin{center}
	\begin{tikzpicture}[baseline=0cm, scale = 1]
		\draw (0,0)--(-.5, -.5);
		\draw (0,0)--(.5, -.5);
		
		\node[label={[yshift=-22pt] \normalsize $\xi_1$}] at (-.5, -.5) {};
		\node[label={[yshift=-22pt] \normalsize $\xi_2$}] at (.5, -.5) {};
		
		\node[label={[yshift= -3pt] \normalsize $\Theta$}] at (1.4, -.7) {$\longmapsto$};
	\end{tikzpicture}%
	\begin{tikzpicture}[baseline=0cm, scale = 1]
		\draw (0,0)--(-.5, -.5);
		\draw (0,0)--(.5, -.5);
		
		\node[label={[yshift=-22pt] \normalsize $\theta(\xi_1)$}] at (-.55, -.5) {};
		\node[label={[yshift=-22pt] \normalsize $\theta(\xi_2)$}] at (.55, -.5) {};		
	\end{tikzpicture}.
\end{center}
In symbols, following the notations of Section \ref{sec:partial-isometries}, we have that 
$$\Theta\circ \tau_\nu = \tau_\nu'\circ \Theta \textrm{ and } \Theta\circ \tau^*_\nu = \tau^{'*}_\nu\circ \Theta \text{ for all vertex } \nu.$$
It is not hard to see that $\Theta$ is a morphism from $\sigma$ to $\sigma'$ in the category of representations of $\cO$ (and thus of $F,T,V$), i.e.
$$\Theta\circ\sigma(x) = \sigma'(x)\circ \Theta \text{ for all } x\in \cO.$$
We deduce that 
$$\Pi^\cO:\Mod(\cP)\to \Rep(\cO),\ (A,B,\fH)\mapsto (\sigma_{A,B}:\cO\act \scrH_{A,B}),\ \theta\mapsto \Theta$$
defines a covariant functor.
By restricting the action $\sigma_{A,B}$ to $X=F,T,V$ we deduce three other functors 
$$\Pi^X:\Mod(\cP)\to\Rep(X).$$ 
Moreover, if we equip these categories with the {\it monoidal} structure given by direct sums we observe that we have {\it monoidal} functors.

\begin{remark}
A priori there are much more morphisms in $\Rep(\cO)$ (and even more in $\Rep(F)$) than in $\Mod(\cP).$
Although, our main theorem will prove that up to taking certain full subcategories our functors $\Pi^X$ become fully faithful (i.e.~is bijective on morphisms) for $X=F,T,V,\cO$.
\end{remark}

\subsection{Diffuse and atomic Pythagorean representations} \label{subsec:diff-atom-rep}

We recall some notions and results on diffuse and atomic representations from \cite{BW22,BW23}.
To ease the reading we first introduce definitions that are (nontrivially) equivalent with the original ones and then explain why they coincide.
For this subsection, unless stated otherwise, $X$ and $Y$ shall refer to an object in the lists $\reps$ and $F,T,V$, respectively.

A P-module $(A,B,\fH)$ uniquely defines a representation $\cP \mapsto B(\fH)$ so that $a \mapsto A, b \mapsto B, a^*\mapsto A^*,b\mapsto B^*$.
When $\xi$ is in $\fH$ we often write $a\xi$ for $A\xi$ and similarly use the notation $w\xi$ when $w$ is a word in $a,b$. 
If $p$ is an infinite word in $a,b$ going from right to left, then we denote $p_n$ to be the first $n$ letters.
\begin{definition}
	A P-module $\fH$ is called 
	\begin{itemize}
		\item \textit{diffuse} if $\lim_{n\to\infty}p_n\xi = 0$ for all $\xi \in \fH$ and all infinite words $p$;
		\item \textit{atomic} if it does not contain any nontrivial diffuse sub-modules.
			\end{itemize}
\end{definition}

\noindent
{\it Diffuse and atomic parts of a P-module.}
From the above definitions we deduce the following decomposition of a general P-module:
\[\fH = \fH_{\diff} \oplus \fH_{\atom} \oplus \fH_{\res}^{AD}\]
where $\fH_{\diff},\fH_{\atom}$ are the largest diffuse and atomic sub-modules, and $\fH_{\res}^{AD}$ is defined to be the orthogonal complement of $\fH_{\diff} \oplus \fH_{\atom}$. 

\begin{definition}
The vector subspace $\fH_{\res}^{AD}$ is necessarily {\it residual} in the sense given in Subsection \ref{sec:P-module}: $\fH_{\res}^{AD}$ does not contain any nontrivial sub-modules.
We shall refer to $\fH_{\res}^{AD}$ as being \textit{the AD}-residual subspace of $\fH$ where $AD$ stands for atomic and diffuse.
\end{definition}

\begin{remark}
For any P-module $\fH$ we have the canonical decomposition $\fH = \fH_{\diff} \oplus \fH_{\atom} \oplus \fH_{\res}^{AD}$ of above which works for any dimension of $\fH$ (including infinite dimension).
If $\dim(\fH)<\infty$, then we have better: $\fH=\oplus_{i=1}^n \fH_i\oplus \fH_{\res}$ where $n$ is finite and the $\fH_i$ are irreducible (either atomic or diffuse) and $\fH_{\res}$ is a residual subspace that is {\it the largest} residual subspace of $\fH$.
We always have that $\fH_{\res}^{AD}\subset \fH_{\res}$ but the inclusion may be proper even in the finite dimensional case.
Indeed, $\fH_{\diff}$ and $\fH_{\atom}$ may contain some nontrivial residual subspaces.
On the other hand, it is not always possible to express $\fH$ as a direct sum of two sub-modules, hence $\fH_{\res}^{AD}$ may be non-trivial.
Moreover, if $\dim(\fH)=\infty$, then they may not exists a largest residual subspace $\fH_{\res}$.
These three phenomena occurs in the following examples.

Indeed, consider $(A,B,\C^2)$ where $Ae_1=e_2= Be_2$ and necessarily $Be_1=0=Ae_2$. 
We have $\C^2=:\fH = \fH_{\atom}$ and thus $\fH_{\diff}=\fH_{\res}^{AD}=\{0\}$ while $\fH_{\res}=\C e_1$.
Next, consider $(A,B,\C^3) :=(1/\sqrt{2}, 1/\sqrt{2}, \C) \oplus m_{01,1}$ (see Subsection \ref{subsec:atomic-case} for definition of $m_{01,1}$). Then define $(\ti A ,\ti B,\C^4)$ so that $\ti A, \ti B$ restrict to $A,B$, respectively, on $\textrm{span}\{e_1, e_2, e_3\}$, and $\ti Ae_4 = 1/2(e_1+e_2),\ \ti Be_4 = 1/2(-e_1+e_3)$. Then $(\C^4)_{\res}^{AD} = \C e_4$ and $\C^4$ cannot be expressed as a direct sum of two sub-modules.
Finally, consider $\fH=\ell^2(\N)$ with $A=B$ equal to the unilateral shift operator times $1/\sqrt 2$ and observe that $\ell^2(\{0,1,\cdots,j\})$ is a residual subspace for any $j\geq 0$. Hence, there is no ``largest" residual subspace for $\fH$.

We warn the reader that the sub-modules $\fH_{\diff}, \fH_{\atom}$ defined above do not necessarily coincide with the similarly termed sub-modules defined in our previous article \cite{BW23}. This is because in \cite{BW23} we specialised to the finite-dimensional case (where the above stronger statement involving $\fH_{\res}$ can be exploited) where else the above applies to the more general infinite-dimensional case. However, if $\dim(\fH) <\infty$ and $\fH_{\res}^{AD} = \fH_{\res}$ then the definitions for $\fH_{\diff}, \fH_{\atom}$ in the present article and in \cite{BW23} do coincide.
\end{remark}

\noindent
{\it Diffuse and atomic representations of the Thompson groups and of the Cuntz algebra.}
Recall, from every P-module $m$ we can obtain a representation of $X$ via the functor 
$$\Pi^X:\Mod(\cP)\to\Rep(X).$$ 
We shall say that a representation $\sigma$ of $X$ is \textit{diffuse} if there exists a diffuse P-module $m$ such that $\Pi^X(m) \cong \sigma$. Furthermore, we say that $\sigma$ is \textit{atomic} if it does not contain any diffuse subrepresentations. From here on, we may refer to diffuse (resp.~atomic) representations, thus suppressing the redundant word ``Pythagorean''.

\noindent
{\it Diffuse P-modules and Ind-mixing representations.}
In \cite{BW22} the authors introduced the property {\it\NInd} for a representation of a group $G$ which is when the representation does not contain the induction of a finite dimensional representation of any subgroup of $G$. One of the main results from \cite{BW22} and \cite{BW23} is that a P-module $m$ is diffuse if and only if its associated representation $\sigma$ of $F$ is \NInd{} (the same also holds for $T$ and $V$). 
This permits to show that if two P-modules induce equivalent P-representations and one of them is diffuse, then both of them must be diffuse.
This last statement also holds after replacing $F$ with $T,V,\cO$ and replacing diffuse with atomic. 

\noindent
{\it Decomposition.}
If $\fH$ is a finite dimensional P-module, then one can decompose it into a finite direct sum of irreducible ones plus a residual subspace.
The atomic irreducible finite dimensional P-module and their associated representations are now totally understood, see \cite{BW23} and also Subsection \ref{subsec:atomic-case}.
We shall show in this present article that a diffuse irreducible finite dimensional P-module yields an irreducible representation of $F$ (and thus of $T,V,\cO$) and moreover classify those using only the underlying P-modules.

\section{Isomorphisms of the categories of diffuse P-representations}\label{sec:isomorphism-diffuse}

\textbf{From now on all P-modules considered are assumed to be diffuse.}

\subsection{The von Neumann algebras of intertwinners} \label{subsec:vNA-interwinners}
For the general theory of von Neumann algebras we refer the reader to \cite{Dixmier81}.

The \textit{strong operator topology} (SOT) on $B(\scrH)$ is the locally convex topology generated by the family of seminorms $B(\scrH) \ni T \mapsto \norm{T(z)}$ indexed by vectors $z \in \scrH$. Hence, a net of (bounded linear) operators $(T_\lambda)_\lambda$ converges to an operator $T$ in the SOT if and only if
\[\lim_\lambda \norm{T(z) - T_\lambda(z)} = 0 \textrm{ for all } z\in \scrH.\]
This will be denoted by $$T_\lambda \xrightarrow{s} T.$$
This is the so-called {\it pointwise convergences} topology or the {\it product topology} where $\scrH$ is equipped with its norm topology.
The SOT is strictly weaker than the norm topology on $B(\scrH)$ when $\dim(\scrH)=\infty$ (of course if $\dim(\scrH)<\infty$ then the SOT and the norm topology are the same).

If $S\subset B(\scrH)$ is a multiplicative subset containing the identity and closed under taking adjoint, then the celebrated theorem of von Neumann shows that the SOT closure of the linear span of $S$ is equal to its bicommutant $S''$ (i.e.~the relative commutant of the relative commutant $S'$ of $S$ inside $B(\scrH)$). 
We say that $S''$ is the von Neumann algebra generated by $S$ and denote it by $\VN(S)$.
When $S=\sigma(G)$ for a unitary representation $\sigma:G\to\U(\scrH)$ of a group $G$, then note that $\sigma(G)'$ is the space of $G$-intertwinners (which is a von Neumann algebra since $S'=S'''$). Hence, $\VN_\sigma(G):=\VN(\sigma(G))$ is the relative commutant of the space of $G$-intertwinners. 

If $\sigma$ is a diffuse representation, then we show that the space of intertwinners for $F,T,V,\cO$ are all equal.
We prove it by showing that $\sigma(F)$ and $\sigma(\cO)$ generate the same von Neumann subalgebra of $B(\scrH)$ and thus have same relative commutant which are the $F$-intertwinners and $\cO$-intertwinners, respectively, (using again the identity $S'=S'''$).

\begin{proposition} \label{prop:pyth-vNA}
If $\sigma:\cO\act\scrH$ is a diffuse representation then the von Neumann algebras generated by $F,T,V,\cO$ are all equal, i.e.~$\sigma(F)''=\sigma(\cO)''.$
\end{proposition} 

\begin{proof}
	Since we have the chain of inclusion $F\subset T\subset V\subset \cO$ it is sufficient to only show that $\VN_\sigma(\cO)\subset\VN_\sigma(F).$
	The operators $\tau_0,\tau_1$ generate $\VN_\sigma(\cO)$ as a von Neumann algebra. Hence, it is sufficient to show that $\tau_i \in \VN_\sigma(F)$ for $i=0,1$.
	
	Since $\sigma$ is diffuse we can use Proposition 2.8 of \cite{BW22} which asserts that
	\[\frac{1}{n}\sum_{k=0}^{n-1} \sigma(g^k) \xrightarrow{s} \id - \rho_{\supp(g)} \text{ for all } g\in F\]
	where $\supp(g) := \overline{\{p \in \cC : g(p) \neq p\}}$ is the closure of the support of $g$ and $\rho_{\supp(g)}$ is the projection defined in Subsection \ref{sec:partial-isometries}.
	We deduce that $\rho_\nu \in \VN_\sigma(F)$ for all vertex $\nu \in \Ver$. 
	
	Fix $\omega,\nu$ some finitely supported binary strings containing at least one $0$ and one $1$ (i.e.~none of them are contained in an endpoint $0^\infty$ or $1^\infty$ of the Cantor space). There exists $g\in F$ so that $g(I_\omega)=I_\nu$ where $I_\omega$ is the sdi associated to $\omega$. This implies that $\sigma(g) \rho_\omega=\tau_\nu^*\tau_\omega\in \VN_{\sigma(F)}$.
	 A similar argument shows $\tau_\nu^*\tau_\omega\in \VN_{\sigma}(F)$ when both $\nu, \omega$ are finitely supported sequences of only $0$'s (resp.~only $1$'s).

Now consider the map $\tau_0$. By using the identity $\id = \rho_0 + \rho_1$ we can rewrite $\tau_0$ as
	\[\tau_0 = \rho_0\tau_0 + \rho_1\tau_0 = \tau^*_0\tau_{00} + \tau^*_1\tau_{01}.\]
	Repeating this procedure, we obtain more generally that
	\begin{equation} \label{eqn:tau_0}
		\tau_0 = \sum_{\substack{\nu \in \Ver:\\ \vert \nu \vert = n}} \tau^*_\nu\tau_{0\nu} \textrm{ for } n \geq 1.
	\end{equation} 
	Diagrammatically, this can be visualised by observing that ``lifting'' the vertex $0$ to the root $\emptyset$ is equivalent to ``lifting'' the two immediate children, $00$ and $01$, of $0$ to the two immediate children, $0$ and $1$, of the root $\emptyset$. 
	Then, reapplying this idea to all children of $0$ at distance $n$.
	
	From the above reformulation, we have expressed $\tau_0$ as a sum of terms $\tau^*_\nu\tau_{0\omega}$ that belong to $\VN_\sigma(F)$ except for the last term which will be in the form $\tau^*_{1^n}\tau_{01^n} = \rho_{1^n}\tau_0$. However, by \cite{BW23}, the assumption that $\sigma$ is diffuse gives
	\[\rho_{1^n}\tau_0 \xrightarrow{s}  0 \textrm{ as } n \to \infty.\]
	Hence, by emitting the last term in Equation \ref{eqn:tau_0}, this proves that $\tau_0$ can be written as the SOT limit of operators in $\VN_\sigma(F)$. Therefore, $\tau_0 \in \VN_\sigma(F)$ and a similar reasoning shows that $\tau_1 \in \VN_\sigma(F)$. This completes the proof.
\end{proof}

Using a 2 by 2 matrix trick we deduce equalities of intertwinner spaces.

\begin{corollary} \label{coro:P-rep-intertwinner}
If $\sigma^\cO:\cO\act\scrH$ and $\ti\sigma^\cO:\cO\act\ti\scrH$ are two diffuse representations, then $\theta:\scrH\to\ti\scrH$ intertwinnes the $F$-actions if and only if it intertwinnes the $\cO$-actions.
In particular, $\sigma^F:F\act\scrH$ is irreducible if and only if the extension $\sigma^\cO:\cO\act\scrH$ is irreducible and $\sigma^F \cong \ti\sigma^{F}$ if and only if $\sigma^\cO \cong \ti\sigma^{\cO}$.
\end{corollary}

\begin{proof}
Consider two diffuse representations $\sigma:\cO\act\scrH$ and $\ti\sigma:\cO\act\ti\scrH$ and let $\theta:\scrH\to\ti\scrH$ be a bounded linear map that intertwinnes the $F$-actions.
Let $\tau:=\sigma\oplus \ti\sigma:\cO\act \scrH\oplus\ti\scrH$ be the direct sum of these representations which is still diffuse.
The map $$\widehat \theta:\scrH\oplus\ti\scrH \to \scrH\oplus\ti\scrH,\ (\xi,\ti\xi)\mapsto (0,\theta(\xi))$$
is bounded linear and intertwinnes the $F$-action (i.e.~the restriction of $\tau$ to $F$), that is, $\widehat \theta$ belongs to the relative commutant of $\tau(F)$.
Proposition \ref{prop:pyth-vNA} implies that $\widehat \theta$ is in the relative commutant of $\tau(\cO)$ and thus intertwinnes the $\cO$-action. 
Therefore, $\theta$ also intertwinnes the $\cO$-action.
The corollary easily follows.
\end{proof}

We deduce the following corollary concerning all (including infinite P-dimensional) diffuse representations of $F$. 

\begin{corollary} \label{coro:diffuse-sub-rep}
	Any subrepresentation of a diffuse representation $\sigma$ of $F$ is again a diffuse representation of the form $(\sigma_\scrX, \sub{\scrX})$ where $\scrX$ is a sub-module of $(\tau_0, \tau_1, \scrH)$.  
\end{corollary}

\begin{proof}
	Let $(A,B,\fH)$ be a diffuse P-module with associated representation $\sigma:F\act\scrH$ and suppose $\scrX \subset \scrH$ is a subrepresentation. By Proposition \ref{prop:pyth-vNA} the subspace $\scrX$ is closed under $\tau_0, \tau_1$. Hence, $\scrX$ is a sub-module of $\scrH$ and we can conclude by Observation \ref{obs:tau-mod-sub-rep} and noting that $\sub{\scrX} = \scrX$.  
\end{proof}

\subsection{Isomorphisms of the full subcategories of diffuse representations}
Recall $\Rep_{\diff}(X)$ is the full subcategory of $\Rep(X)$ whose objects are the {\it diffuse} (Pythagorean) representations of $X$ for $X = \reps$. We now provide a functorial interpretation of Corollary \ref{coro:P-rep-intertwinner}.
Recall that two categories $\cC$ and $\cD$ are {\it isomorphic} when there exists a functor $\Phi:\cC\to\cD$ that is bijective on objects and additionally for each $c,c'$ in $\cC$ the functor $\Phi$ realises a bijection from $\cC(c,c')$ to $\cD(\Phi(c),\Phi(c')).$ This is a stronger notion than being {\it equivalent} categories where we only ask that for any object $d$ in $\cD$ there is an object $c$ in $\cC$ so that $\Phi(c)$ is isomorphic to $d$.

\begin{theorem} \label{theo:isom-diff-functor}
	The four categories $\Rep_{\diff}(X)$ with $X=F,T,V,\cO$ are \emph{isomorphic}.
\end{theorem}

\begin{proof}
Fix $Y$ being either $F,T$ or $V$.
Let $(\sigma^\cO, \scrH)$ be a diffuse representation of $\cO$. 
Restricting the action of $\cO$ to $Y$ yields the forgetful functor $\Rep_{\diff}(\cO)\to\Rep_{\diff}(Y)$.
Conversely, let $\sigma^Y:Y\act\scrH$ be a diffuse representation of $Y$. 
The proof of Proposition \ref{prop:pyth-vNA} permits to reconstruct $\tau_0$ and $\tau_1$ canonically from the action of $Y$ on $\scrH$. 
This provides the action $\sigma^\cO$ of the Cuntz algebra and defines a functor $:\Psi:\sigma^Y\mapsto \sigma^\cO$ satisfying $\Pi^\cO=\Psi\circ\Pi^Y$ when $\Pi^Y$ and $\Pi^\cO$ are restricted to diffuse representations (see Subsection \ref{subsec:p-rep-functors} for the definition of $\Pi^Y$).
Corollary \ref{coro:P-rep-intertwinner} shows that $\Psi$ is bijective on the spaces of morphisms.
Moreover, the composition of $\Psi$ with the restriction functor of above is the identity on objects and thus $\Psi$ must be bijective on objects.
We deduce that the two categories are isomorphic.
\end{proof}

\section{Classification of finite dimensional P-representations}\label{sec:classification}

This section contains the main results of the paper. 
We introduce the Pythagorean dimension, classify all finite dimensional diffuse Pythagorean representations, and describe the resulting moduli spaces of irreducible classes as smooth manifolds indexed by dimension. For brevity, several of the results concerning diffuse representations in this section are only stated for $F$; however, they can be readily extended to the $T,V,\cO$ case by applying Theorem \ref{theo:isom-diff-functor}. 

\subsection{Pythagorean dimension} \label{subsec-pyth-dim}

\begin{definition} \label{def:pythag-dim}
	Let $(\sigma^X, \scrH)$ be \emph{any} Pythagorean representation of $X$ for $X = \reps$. 
	Define the \emph{Pythagorean dimension} (in short P-dimension) $\dim_P(\sigma^X)$ of $\sigma^X$ to be infinite if there are no finite dimensional P-module $m$ satisfying $\Pi^X(m)\simeq \sigma^X$.
	Otherwise, the P-dimension of $\sigma^X$ is the minimal dimension $\dim(\fH)$ for all P-modules $m=(A,B,\fH)$ satisfying that $\Pi^X(m)$ is equivalent to $\sigma^X$; in symbols:
	$$\dim_P(\sigma^X) = \min(\dim(\fH):\ \Pi^X(A,B,\fH)\simeq \sigma^X).$$

	Similarly, define the P-dimension $\dim_P(\fH)$ of a P-module $\fH$ to be the P-dimension of the associated P-representation of $F$.
	
	If the P-dimension is $d\in\N$ (resp.~finite), then we may describe $\sigma^X$ as being $d$-dimensional (resp.~finite dimensional). Note that all Pythagorean representations are necessarily infinite dimensional for the usual dimension (see Remark 1.4 in \cite{BW22}), thus this should not cause any confusion.
\end{definition}

By definition $\dim_P(\fH)\leq \dim(\fH)$ for a P-module $\fH$ where the first refers to the P-dimension and the second to the usual dimension of $\fH$ as a complex vector space.
Moreover, observe $\dim_P(\sigma^X)$ does not depend on $X=F,T,V,\cO$.

\begin{example}
\begin{enumerate}
\item Using the classification of \cite{BW23} we may deduce the following (see also Subsection \ref{subsec:atomic-case}).
Let $p$ be an infinite binary string and let $\lambda_{Y/Y_p}:Y\act\ell^2(Y/Y_p)$ be the associated quasi-regular representation of $Y$ for $Y$ in $F,T,V$. Remove the case when $Y=F$ and $p$ a dyadic rational (i.e. $p$ is eventually constant). 
If $p$ is eventually periodic then $\dim_P(\lambda_{Y/Y_p})=d$ where $d$ is the length of its period. Otherwise, in a future work it will be shown that $\lambda_{Y/Y_p}$ is indeed an atomic representations coming from an infinite dimensional P-module and thus $\dim_P(\lambda_{Y/Y_p})=\infty$.
\item
Similarly, if $w$ is a prime binary word of length $d\geq 2$, $z\in \bS_1$, and $m_{w,z}$ is the associated P-module of Subsection \ref{subsec:atomic-case}, then $\dim_P(m_{w,z})=d$. When $z=1$, then the associated $Y$-representation is $\lambda_{Y/Y_p}$ (where $p=w^\infty$ and $Y$ in $F,T,V$) and for nontrivial $z$ is monomial.
\item
If $(A,B,\C)$ is a P-module, then necessarily it is of P-dimension 1.
\end{enumerate}
\end{example}

The next lemma shows that given a P-representation $F\act \scrH$ we may only consider sub-modules of $\scrH$ for computing $\dim_P(\scrH)$.

\begin{lemma} \label{lem:pyth-dim-tau-mod}
	If $\sigma:F\act \scrH$ is a diffuse representation then its P-dimension is equal to the minimum of the dimension of all \emph{complete} sub-modules $\scrX$ in $\scrH$.	
\end{lemma}

\begin{proof}
Consider a diffuse representation $\sigma:F\act \scrH$.
	By Remark \ref{obs:tau-mod-sub-rep} and Lemma \ref{lem:complete-mod} it is obvious that the Pythagorean dimension must be equal or smaller than the dimension of any complete sub-module of $\scrH$. Hence, it is sufficient to show there exists a complete sub-module of $\scrH$ with usual dimension equal to $\dim_P(\sigma)$.
	
Let $\ti m:=(\ti A, \ti B, \ti\fH)$ be a P-module such that $\dim(\ti \fH) = \dim_P(\sigma)$ and $(\ti \sigma,\ti\scrH):=\Pi^F(\ti m)\simeq (\sigma,\scrH)$.
Then there is a unitary intertwinner $\Theta : \ti \scrH \rightarrow \scrH$. 
Set $\fX := \Theta(\ti \fH)$ and $\xi := \Theta(\eta) \in \fX$ for some $\eta \in \ti \fH$. Then we have:
	\[\tau_0(\xi) = \tau_0(\Theta(\eta)) = \Theta(\ti \tau_0(\eta)) = \Theta(\ti A \eta) \in \fX\]
	where we use the fact that $\ti A = \ti\tau_0\restriction_{\ti \fH}$ and $\Theta$ intertwinnes $\tau_i$ with $ \ti\tau_i$ by Corollary \ref{coro:P-rep-intertwinner}. A similar reasoning yields that $\tau_1(\xi) \in \fX$, thus proving that $\fX$ is a sub-module of $\scrH$ with $\dim(\fX) = \dim(\ti \fH) = \dim_P(\sigma)$. 
	
To finish the proof it is sufficient to show that $\fX$ is a complete sub-module of $\scrH$, i.e.~$\langle \fX\rangle=\scrH$.
	By another application of Corollary \ref{coro:P-rep-intertwinner} we have:
	\[\tau^*_i(\xi) = \tau^*_i(\Theta(\eta)) = \Theta(\ti\tau^*_i(\eta)) \textrm{ for } i = 0,1.\]
	Further, by definition $\sub{\ti\fH} = \ti\scrH$. Hence, with the above equation we can conclude that $\sub{\fX} = \Theta(\sub{\ti\fH}) = \scrH$. 
\end{proof}

\subsection{Existence of a smallest nontrivial module}
\label{subsec:smallest-P-mod}

For an atomic representation, its P-dimension can be easily determined (see Subsection \ref{subsec:atomic-case}) by using the classification of such representations in \cite{BW23}.
However, for diffuse representations, a priori, it is not clear what is its P-dimension for it requires knowledge of the complete sub-modules of $\scrH$ as demonstrated by Lemma \ref{lem:pyth-dim-tau-mod}.

It is easy to see that the set of sub-modules in $\scrH$ forms a directed set. We shall prove that for finite P-dimensional irreducible P-representations this directed set contains a \emph{unique smallest nontrivial sub-module} which is necessarily contained in the P-module $\fH$. {\bf This is a very surprising result and the most important of our article}. 

\begin{theorem} \label{theo:min-invariant-subspace}
Let $(A,B, \fH)$ be a diffuse finite dimensional P-module whose associated P-representation $\sigma:F\act\scrH$ is irreducible. Then there exists a nontrivial sub-module $\fK \subset \fH$ such that any sub-module $\scrX \subset \scrH$ must contain $\fK$.
	
In particular, $\fK \subset \fH$ is the unique smallest nontrivial sub-module in $\scrH$ and satisfies $\dim(\fK) = \dim_P(\sigma)$. Furthermore, if $\fH$ is an irreducible P-module then $\fK = \fH$.
\end{theorem}

\begin{proof}
	First, observe the assumption that $\sigma$ is irreducible assures that all nontrivial sub-modules of $\scrH$ are complete by Corollary \ref{coro:diffuse-sub-rep}.
	Let $\scrX$ be any nontrivial sub-module of $\scrH$ and let $\fK$ be a (nontrivial) sub-module of $\scrH$ which satisfies $\dim(\fK) = \dim_P(\sigma)$ (existence of such a sub-module is guaranteed by Lemma \ref{lem:pyth-dim-tau-mod}). In particular, $\fK$ is a finite dimensional subspace and is a minimal sub-module in $\scrH$ with respect to set inclusion.
We are going to show that $\scrX\cap \fK$ is nontrivial (by showing that their unit sphere intersect) which would imply that $\fK$ is contained in $\scrX$.
	
Consider an arbitrary unit element $z \in \scrX$ and fix $0 < \varepsilon < 1$. Since $\sub{\fK} = \scrH$ there exists $[t, \xi] \in \scrH$ such that $\norm{z - [t, \xi]} < \varepsilon$ and $\xi_\nu \in \fK$ for $\nu \in \Leaf(t)$ (recall $\xi_\nu \in \fK$ is the component of $\xi$ under the leaf $\nu$). This can be re-formulated as
	\[\norm{z - [t, \xi]}^2 = \norm{z - \sum_{\nu \in \Leaf(t)} \tau^*_\nu(\xi_\nu)}^2 = \sum_{\nu\in\Leaf(t)} \norm{\tau_\nu(z) - \xi_\nu}^2 < \varepsilon^2.\]
	Further, we can assume that $\xi_\nu = 0$ whenever $\tau_\nu(z) = 0$ as this will continue to satisfy the inequality $\norm{z - [t, \xi]} < \varepsilon$.
	Then there exists a vertex $\omega \in \Leaf(t)$ such that
	\[\tau_\omega(z) \neq 0 \textrm{ and } \norm{\frac{\tau_\omega(z)}{\norm{\tau_\omega(z)}} - \frac{\xi_\omega}{\norm{\tau_\omega(z)}}} < \varepsilon.\]
	Indeed, if such a vertex did not exist, then that will imply 
	\[\sum_{\nu\in\Leaf(t)} \norm{\tau_\nu(z) - \xi_\nu}^2 \geq \varepsilon^2\sum_{\nu\in\Leaf(t)} \norm{\tau_\nu(z)}^2 = \varepsilon^2 \]
	which is a contradiction.
Since $\varepsilon<1$ we deduce that $\xi_\omega\neq 0$.
	Moreover,
	\begin{align*}
		\norm{\frac{\tau_\omega(z)}{\norm{\tau_\omega(z)}} - \frac{\xi_\omega}{\norm{\xi_\omega}}} 
		&\leq \norm{\frac{\tau_\omega(z)}{\norm{\tau_\omega(z)}} - \frac{\xi_\omega}{\norm{\tau_\omega(z)}}} +
		\norm{\frac{\xi_\omega}{\norm{\tau_\omega(z)}} - \frac{\xi_\omega}{\norm{\xi_\omega}}} \\
		&< \varepsilon + \norm{\xi_\omega} \cdot \left| \frac{1}{\norm{\tau_\omega(z)}} - \frac{1}{\norm{\xi_\omega}}\right|
		 =\varepsilon + \left| \frac{\Vert \xi_\omega\Vert}{\Vert\tau_\omega(z)\Vert} - 1 \right|\\
		 & = \varepsilon +  \frac{1}{\| \tau_\omega(z) \|} \cdot  \left| \|\xi_\omega \| -  \| \tau_\omega(z) \| \right|< 2 \varepsilon.
	\end{align*}
	
We have shown that for any $\varepsilon>0$ there exist two unit vectors $y\in\fK$ and $z\in\scrX$ so that $\norm{y-z}<\varepsilon.$
	Therefore, from the above procedure, we can construct a sequence of unit vectors $(y_n : n \geq 1)$ of $\fK$ and another sequence of unit vectors $(z_n : n \geq 1)$ in $\scrX$ such that $\norm{y_n - z_n} \rightarrow 0$ as $n \rightarrow \infty$. By compactness of the unit sphere in the \emph{finite dimensional} space $\fK$, we can pass to a subsequence and assume that $(y_n)$ converges to a unit vector $y \in \fK$. Hence, it follows that $(z_n)$ also converges to $y$. 
Since $\scrX$ is a sub-module it is in particular a Hilbert space and thus topologically closed inside $\scrH$. This implies that the limit point $y$ must be in $\scrX$.
	
We have proven that the intersection $\fK\cap \scrX$ is nontrivial since it contains a vector of norm one. 
Moreover, since $\fK$ and $\scrX$ are P-modules (they are sub-modules of $(\tau_0^*,\tau_1^*,\scrH$)) we deduce that $\fK\cap\scrX$ is a nontrivial sub-module of $\scrH$.
By minimality of the P-module $\fK$ we deduce that $\fK=\fK\cap \scrX$ and thus $\fK$ is contained in $\scrX$.
Hence, $\fK$ is contained in {\it any} nontrivial sub-module of $\scrH$. 
The theorem easily follows.
\end{proof}

\begin{remark}\label{rem:THB}
Theorem \ref{theo:min-invariant-subspace} will later be extended (in the finite P-dimensional situation) to the atomic case and moreover to the reducible case where $\fK$ will become the smallest {\it complete} sub-module rather than the smallest {\it nontrivial} sub-module.
However, the proof of the existence of $\fK$ uses the compactness of the unit sphere of the P-module $\fH$ and thus cannot be adapted to {\it infinite P-dimensional} P-modules. 
An explicit counterexample in the diffuse case is given by $\fH=\ell^2(\Z)$ and $A=B=S/\sqrt 2$ which is the unilateral shift operator divided by $\sqrt 2$. Each subspace $\fK_j:=\ell^2(\{j,j+1,j+2,\cdots\})$ with $j\in\Z$ defines a complete sub-module but clearly their intersection is trivial.
By a similar reasoning, instead letting $A=S$ and $B=0$ provides a counterexample in the atomic case.
\end{remark}

\subsection{From intertwinners of diffuse P-representations to intertwinners of P-modules}

\subsubsection{Irreducibility and equivalence.} \label{subsec:pyth-irrep-equiv}

We can now provide a complete classification of finite dimensional diffuse P-representations. 
Surprisingly, while these results are significant and very nontrivial, the techniques used in the proofs are elementary in nature.

\begin{theorem} \label{theo:pythag-irrep}
	Let $m=(A,B, \fH)$ be a diffuse finite dimensional P-module. Then the associated P-representation $\Pi^F(m):=(\sigma,\scrH)$ of $F$ is irreducible if and only if $\fH$ is indecomposable.
\end{theorem}

\begin{proof}
Consider a P-module $m=(A,B,\fH)$ and the associated P-representation $\Pi^F(m)$ denoted $(\sigma,\scrH)$.

Suppose that $\fH$ is decomposable and $\fX_1, \fX_2$ are two orthogonal sub-modules of $\fH$. Then it is clear that $\sigma_{\fX_1}, \sigma_{\fX_2}$ are two orthogonal nontrivial subrepresentations of $\sigma$, and thus $\sigma$ is reducible. This proves the easy direction.
	
Suppose now that $\fH$ is indecomposable. 
Further, let $\scrX \subset \scrH$ be any nontrivial closed subspace which is invariant under $\sigma$ (i.e.~a subrepresentation). 
Denote $\Theta$ to be the orthogonal projection of $\scrH$.
We are going to show that $\Theta$ is injective and thus $\Theta=\id$.

\noindent	\textbf{Claim 1.} We have that $\ker(\Theta)\cap \fH=\{0\}.$

Consider $\fK$ the set of $\xi\in\fH$ such that $\Theta(\xi)=0$.
Since $\Theta$ intertwinnes the $F$-action it commutes with $\tau_0,\tau_1$ by Proposition \ref{prop:pyth-vNA}.
We deduce that $\fK$ is a sub-module of $\fH$.
If $\fK$ is nontrivial, then $\fK$ must be complete inside $\fH$ since this latter is indecomposable.
By Remark \ref{rem:complete-sub-rep} $\fK$ is complete inside $\scrH$ and thus $\langle \fK\rangle=\scrH$.
Since $\Theta$ commutes with the adjoints $\tau_0^*,\tau_1^*$ we deduce that $\langle \fK\rangle$ is contained inside $\ker(\Theta)$ and thus $\Theta=0$, a contradiction.
This proves the claim. 
	
\noindent\textbf{Claim 2.} The map $\Theta$ is injective.

We have proven that the restriction $\Theta\restriction_{\fH}$ is injective. Since $\dim(\fH)<\infty$ we deduce that $\Theta\restriction_{\fH}$ is bounded below.
Hence, there exists $c>0$ satisfying that $\|\Theta(\xi)\|\geq c\|\xi\|$ for all $\xi\in\fH$.
Let us show that $\Theta$ is bounded below on the whole space $\scrH$.
Consider any $z := [t, \xi]$ in the dense subspace $\scrK$ of $\scrH$ which we can express as 
	\[z = \sum_{\nu \in \Leaf(t)} \tau_\nu^*(\xi_\nu)\]
	where $\xi_\nu \in \fH$ is the component of $\xi$ under the leaf $\nu$. In addition, we have the following relations:
	\begin{align*}
		\norm{\Theta(z)}^2 &= \norm{\Theta(\sum_{\nu \in \Leaf(t)} \tau_\nu^*(\xi_\nu))}^2 
		= \norm{\sum_{\nu \in \Leaf(t)} \tau_\nu^*(\Theta(\xi_\nu))}^2 \\
		&= \sum_{\nu \in \Leaf(t)} \norm{\Theta(\xi_\nu)}^2
		\geq \sum_{\nu \in \Leaf(t)} c^2\norm{\xi_\nu}^2 = c^2\norm{\xi}^2.
	\end{align*}
	The above shows that $\Theta$ is bounded below on a dense subspace $\scrK \subset \scrH$. 
	By density and continuity of $\Theta$ (it is a projection) we deduce that $\Theta$ is injective. Hence, $\Theta=\id$ and $\sigma$ is irreducible.
\end{proof}

\begin{corollary} \label{coro:pythag-irrep}
	Every diffuse finite dimensional P-representation of $F$ is a finite direct sum of irreducible diffuse P-representations. 	
\end{corollary}

\begin{proof}
	This follows from the previous theorem and by the fact, from \cite{BW23}, that all diffuse finite dimensional P-module decomposes in a finite direct sum of diffuse irreducible P-modules (plus possibly a residual subspace).
	\end{proof}

We now provide necessary and sufficient conditions for when irreducible Pythagorean representations are equivalent to each other.

\begin{theorem} \label{theo:pythag-equiv}
Consider two P-modules $(A, B, \fH)$ and $(\ti A, \ti B, \ti\fH)$ that are indecomposable, diffuse and finite dimensional.
Let $(\sigma,\scrH)$ and $(\ti\sigma,\ti\scrH)$ be their associated {\it irreducible} representations of $F$ and denote by $\fX$ and $\ti\fX$ the unique smallest nontrivial sub-modules of $\fH, \ti\fH$, respectively.
	
	The representations $\sigma$ and $\ti\sigma$ of $F$ are unitarily equivalent if and only if there exists a unitary operator $u : \fX \rightarrow \ti \fX$ which intertwinnes the sub-modules $\fX, \ti\fX$.
	In particular, if $\fH, \ti\fH$ are irreducible P-modules, then $\sigma$ is equivalent to $\ti\sigma$ if and only if there exists a unitary operator $u:\fH\to\ti\fH$ satisfying that 
	$$u\circ A= \ti A\circ u \text{ and } u\circ B = \ti B\circ u.$$
	The same statement holds if we replace $F$ by $T,V$ or $\cO$.
\end{theorem}

\begin{proof}
	Consider two P-modules $(A, B, \fH)$ and $(\ti A, \ti B, \ti\fH)$ that are indecomposable, diffuse and finite dimensional. 
	By Theorem \ref{theo:pythag-irrep}, their associated representations $(\sigma,\scrH)$ and $(\ti\sigma,\ti\scrH)$ of $F$ are irreducible.
	Consider $\fX,\ti\fX$ as above which exist by Theorem \ref{theo:min-invariant-subspace}.
		
	Suppose that there exists a unitary operator $u:\fX\to\ti\fX$ which intertwinnes the P-modules structures. The functor $\Pi^F:\Mod(\cP)\to\Rep(F)$ maps $u$ to a unitary operator $\Theta : \sub{\fX} \rightarrow \sub{\ti\fX}$ which intertwinnes the two actions of $F$. Since $\sigma$ and $\ti\sigma$ are irreducible, $\fX$ and $\ti\fX$ are complete sub-modules. Subsequently, we have that $\sub{\fX}=\scrH$, $\sub{\ti\fX}=\ti\scrH$ and thus $\Theta$ is an intertwinner from $\scrH$ to $\ti\scrH$. This proves the easy direction of the theorem.
	
	Now we shall prove the more surprising direction and assume that there exists $\Theta : \scrH \rightarrow \ti \scrH$ a unitary operator intertwining the actions $\sigma, \ti \sigma$ of $F$.
	Up to equivalence and taking restrictions of the initial P-module,  we can assume that $\fH = \fX$, $\ti\fH = \ti\fX$. By Theorem \ref{theo:min-invariant-subspace} we have that $\fH, \ti\fH$ are in fact the unique smallest nontrivial sub-modules of the larger Hilbert spaces $\scrH,\ti\scrH$, respectively, and $\dim(\fH) = \dim_P(\sigma)$, $\dim(\ti \fH) = \dim_P(\ti\sigma)$. Further, since the Pythagorean dimension is an invariant for P-representations this implies that $\dim(\fH) =\dim(\ti\fH)$.
	Observe that the image $\Theta(\fH)$ is a sub-module of $\ti\scrH$ with $\dim(\Theta(\fH)) = \dim(\fH) = \dim(\ti\fH)$. 
	Hence, by uniqueness of the smallest nontrivial sub-module, we can conclude that $\ti \fH = \Theta(\fH)$. Therefore, the unitary intertwinner $\Theta$ restricts to a unitary operators from $\fH$ to $\ti \fH$ which intertwinnes the P-module structures.
	
	By Corollary \ref{coro:P-rep-intertwinner} the case $X=F$ implies all the other cases $X=F,T,V,\cO$.
\end{proof}

\subsubsection{Extension to reducible representations} \label{subsec:ext-reducible-rep}

Appealingly, the P-dimension is compatible with the usual direct sum of representations when considering diffuse P-representations permitting us to easily compute P-dimensions. 

\begin{proposition} \label{prop:p-dim-sum}
	If $\sigma_1, \sigma_2$ are two diffuse Pythagorean representations, then
	\[\dim_P(\sigma_1 \oplus \sigma_2) = \dim_P(\sigma_1) + \dim_P(\sigma_2).\]
	Here we take the convention that any addition involving $\infty$ is equal to $\infty$.
\end{proposition}

\begin{proof}
	Consider two P-representations $(\sigma_1,\scrH_1)$ and $(\sigma_2,\scrH_2)$ of $F$.
	It is obvious that $\dim_P(\sigma_1 \oplus \sigma_2) \leq \dim_P(\sigma_1) + \dim_P(\sigma_2)$. Hence, we shall prove the reverse inequality. Since it is clear that $\dim_P(\sigma) \leq \dim_P(\pi)$ if $\sigma \subset \pi$ we can assume that both $\sigma_1, \sigma_2$ are finite dimensional Pythagorean representations. Further, by Corollary \ref{coro:pythag-irrep} it is sufficient to prove the lemma when $\sigma_1, \sigma_2$ are both irreducible.  
	
	By Lemma \ref{lem:pyth-dim-tau-mod}, let $\scrX \subset \scrH_1 \oplus \scrH_2$ be a complete sub-module such that $\dim(\scrX) = \dim_P(\sigma_1 \oplus \sigma_2)$. 
If $\scrX$ is not full, then there is a proper complete sub-module $\scrX_0\subset\scrX$. By completeness $\langle \scrX_0\rangle$ is the carrier Hilbert space of $\sigma_1\oplus\sigma_2$ and thus $\dim_P(\sigma_1\oplus\sigma_2)\leq \dim(\scrX_0)\leq \dim(\scrX)-1$, a contradiction. 
Therefore, $\scrX$ is full.
By Theorem \ref{theo:pythag-irrep} and the property of being full, $\scrX$ can be decomposed as $\scrX_1 \oplus \scrX_2$ where $\scrX_i$ is an irreducible sub-module which induces the subrepresentation $\sigma_i$ for $i = 1,2$.
	Then by invoking Lemma \ref{lem:pyth-dim-tau-mod} again we can conclude since
	\[\dim_P(\sigma_1) + \dim_P(\sigma_2) \leq \dim(\scrX_1) + \dim(\scrX_2) = \dim(\scrX) = \dim_P(\sigma_1 \oplus \sigma_2). \qedhere\]
\end{proof}

We now extend Theorem \ref{theo:min-invariant-subspace} to reducible Pythagorean representations.

\begin{corollary} \label{cor:min-invariant-subspace}
	Let $(A,B,\fH)$ be a diffuse P-module of finite Pythagorean dimension with associated $F$-representation $(\sigma,\scrH)$. Then there exists a unique smallest complete sub-module $\fK$ of $\scrH$ which is contained in $\fH$. More precisely, $\fK$ is the (unique) complete and full sub-module of $\fH$ or equivalently the smallest complete sub-module of $\fH$.
\end{corollary}

\begin{proof}
First suppose that $\fH$ is finite dimensional (as a vector space).
Let $(\sigma, \scrH)$ be the associated P-representation. 
From \cite{BW23} the P-module $\fH$ can be decomposed as 
\[\fH = \oplus_{i=1}^n \fH_{i}^{\oplus m_i} \oplus \fH_{\res} = \oplus_{i=1}^n \ti\fH_i \oplus \fH_{\res}\]
where $\{\fH_i\}_i$ are irreducible, pairwise inequivalent diffuse sub-modules, $\ti\fH_i = \fH_{i}^{\oplus m_i}$, and $\fH_{\res}$ is the residual subspace (note we have used the usual decomposition formula for $\fH$ and have grouped together equivalent sub-modules). 
By Theorem \ref{theo:pythag-irrep} this induces the decomposition
\[\sigma = \oplus_{i=1}^n \sigma_{i}^{\oplus m_i} = \oplus_{i=1}^n \ti\sigma_i\]
where the $\sigma_{i}$ are irreducible and pairwise inequivalent.

Let us show that $\fK := \oplus_{i=1}^n \ti\fH_i$ is the smallest complete sub-module of $\scrH$. Fix a complete sub-module $\scrX$. Up to removing a residual subspace we can assume that $\scrX$ is full. 
If $\scrX$ is irreducible, then $\langle \scrX\rangle$ is irreducible and thus $\sigma$ is irreducible implying that $n=1$ and $m_1=1$.
We are done by Theorem 	\ref{theo:min-invariant-subspace}.
Otherwise, $\scrX$ decomposes into a nontrivial direct sum of two which yields a decomposition of $\sigma$. Since this latter is a finite direct sum of irreducible representations the process of decomposing $\scrX$ must eventually terminate. 
Hence, $\scrX=\oplus_{\alpha} \scrX_\alpha$ with $\alpha$ in a finite index set and $\scrX_\alpha$ irreducible P-module.
By Theorem \ref{theo:pythag-irrep} we have that $\langle \scrX_\alpha\rangle$ must be irreducible.
By Schur Lemma we have that each $\langle \scrX_\alpha\rangle$ is unitarily equivalent to a certain $\sigma_i$.
Up to re-decomposing $\scrX$ we may now assume that
	\[\scrX = \oplus_{i=1}^n \scrX_{i}^{\oplus m_i} = \oplus_{i=1}^n \ti\scrX_i \]
	where necessarily $\langle \scrX_i\rangle \simeq \sigma_i.$
Theorem \ref{theo:min-invariant-subspace} implies that $\sub{\fH_i} = \sub{\scrX_i}$ for each $i$ and each copy of $\fH_i$ and $\scrX_i$. 
Therefore, $\fK$ is contained in $\scrX$ and thus is the smallest complete sub-module of $\scrH$. 
This proves the result for when $\fH$ is finite-dimensional.   

Now we remove the restriction that $\fH$ is finite dimensional as a vector space (but note that $\fH$ is still of finite P-dimension). 
By Lemma \ref{lem:pyth-dim-tau-mod} there exists a finite dimensional sub-module $\fX\subset\scrH$ which is complete.
Applying the proof of above to $\fX$ yields the existence of a unique smallest complete sub-module $\fK$ in $\scrH$. Since $\fH$ is itself a complete sub-module it must contain $\fK$.
\end{proof}

The next corollary extends Theorem \ref{theo:pythag-equiv} to the reducible case.

\begin{corollary} \label{cor:intertwinner-rep-to-mod}
	Every intertwinner between diffuse representations of finite P-dimension restricts to an intertwinner between their respective smallest complete sub-module.
	In particular, two diffuse representations of finite P-dimension are unitarily equivalent if and only if their respective smallest complete sub-modules are unitarily equivalent.
\end{corollary}

\begin{proof}
	Let $(\sigma,\scrH)$ and $(\ti\sigma,\ti\scrH)$ be two such representations with intertwinner $\Theta : \scrH \rightarrow \ti\scrH$. By Corollary \ref{coro:pythag-irrep} we have a decomposition into irreducible diffuse subrepresentations:
	\[(\sigma,\scrH) = \oplus_{i=1}^m (\sigma_i,\scrH_i) \text{ and } (\ti\sigma,\ti\scrH) = \oplus_{j=1}^n (\ti\sigma_j,\ti\scrH_j).\]
	Denote by $\fH_i\subset\scrH_i$ and $\ti\fH_j\subset \ti\scrH_j$ the unique smallest complete sub-modules for all $i,j$.
	By Schur's lemma, for each $i,j$ the morphism $\Theta$ restricts into a morphism $\Theta_{j,i}:\sigma_i\to \ti\sigma_j$ which is either zero or an isomorphism.
	Given $i,j$ and assuming that $\Theta_{j,i}\neq 0$, this morphism restricts into an isomorphism between the smallest complete sub-modules $\fH_i$ and $\ti\fH_j$. The result then follows from Corollary \ref{cor:min-invariant-subspace}.
\end{proof}

\begin{remark}\label{rem:THC}
We have proven that if a P-module is indecomposable and is additionally {\it diffuse and finite P-dimensional}, then its associated representations of $F,T,V,\cO$ are all irreducible.
This is no longer true in full generality if we do not restrict to {\it diffuse and finite P-dimensional} P-modules.
Indeed, indecomposable {\it atomic} P-modules of finite dimension yields the following reducible representation of $F$: $1_F\oplus\lambda_{F/F_{1/2}}$. 
Although, this is rather an exception for the finite dimensional atomic case which only appears for the one P-dimensional representations of $F$ (the extensions to $T,V,\cO$ are irreducible and for higher dimension this never happens even for $F$).
In the infinite P-dimensional case we have that an indecomposable atomic P-module yields the quasi-regular representation 
$\lambda_{ F/\widehat F_{1/2} } :F\act \ell^2( F/\widehat F_{1/2} )$ 
where  $\widehat F_{1/2}=\{g\in F:\ g(1/2)=1/2, g'(1/2)=1\}$ is a subgroup of $F$ that is {\it not} self-commensurated and thus $\lambda_{F/\widehat F_{1/2}}$ is {\it reducible} (and so are the extensions to $T,V$ and even the extension to $\cO$).
Moreover, the example outlined in Remark \ref{rem:THB} provides a reducible representation of $F$ (worst it is a direct integral over a non-atomic measure of representations just like the previous example $\lambda_{ F/\widehat F_{1/2} }$) built from an indecomposable infinite dimensional diffuse P-module which stays reducible when extended to $T,V$ and $\cO$.\\
Second, we have proven that we may classify diffuse irreducible finite P-dimensional representations of $F,T,V,\cO$ by checking conjugacy of their underlying smallest nontrivial P-modules.
We will explain in Subsection \ref{subsec:atomic-case} that this extends to atomic finite P-dimensional representations of $F,T,V,\cO$ (after discarding the case of $F$ in P-dimension one). 
The classification is even less tractable in the infinite P-dimensional case since indecomposable P-modules may yield reducible representations both in the atomic and diffuse cases.
\end{remark}

\subsection{The atomic case} \label{subsec:atomic-case}
We now extend certain main results to the atomic case.
We start by giving the generic form of all finite dimensional irreducible atomic P-modules and their associated P-representations.

\noindent
{\it P-module.}
Fix $d\geq 1$ and define $W_d$ to be a set of representatives of prime binary strings of length $d$, i.e.~finite word $w$ in $0,1$ that cannot be written as $v\cdot v\cdots v$ with $v$ a smaller word.
Define the following partial shift operators:
$$\begin{cases}
A_w(e_k) = e_{k+1} \text{ and }B_we_k = 0 & \text{ if the $k$th digit of $w$ is equal to $0$}\\
A_w(e_k) = 0 \text{ and }B_we_k = e_{k+1} & \text{ otherwise},
\end{cases}$$
where $(e_1,\cdots,e_d)$ is the usual basis of $\C^d$ and for convenience of notation, we write $e_{d+1}$ for $e_1$. 
Given $z\in \bS_1$ in the circle we define
$$m_{w,z}:=(A_w D_z, B_w D_z,\C^d)$$
where $D_z$ is the diagonal matrix with coefficients equal to $1$ except the last one equal to $z$.
By \cite{BW23}, we have that {\it every} irreducible atomic P-module is unitarily equivalent to a P-module of the form $m_{w,z}$.

\noindent 
{\it P-representations of $F,T,V$.}
Consider now $Y=F,T,V$, $d\geq 1$, and $(w,z)\in W_d\times \bS_1$.
The P-module $m_{w,z}$ yields the following P-representation of $Y$.
Let $p:=w^\infty\in\cC:=\{0,1\}^{\N^*}$ be the infinite binary string obtained by concatenating $w$ with itself indefinitely.
Define the parabolic subgroup
$$Y_p:=\{g\in Y:\ g(p)=p\}$$
for the usual action $Y\act \cC$.
Define now the one dimensional representation: 
$$\chi_{z}^p: Y_w\to \C,\ g\mapsto z^{\log(2^d)(g'(p))}$$
where $\log(2^d)$ is the logarithm in base $2^d$.
Finally, consider the induced representation $\Ind_{Y_p}^Y \chi_z^p$ of the Thompson group $Y$ which is by definition {\it monomial}.
From \cite{BW23} we deduce that if $(Y,d)\neq (F,1)$ then 
$$\Pi^Y(m_{w,z})\simeq \Ind_{Y_p}^Y \chi_z^p.$$
Moreover, the representation $\Ind_{Y_p}^Y \chi_z^p$ is irreducible (since $Y_p\subset Y$ is a self-commensurated subgroup) and by inspection we have that $m_{w,z}$ is irreducible as well.
In particular, $$\dim_P(m_{w,z})=|w|=d.$$
When $(Y,|w|)=(F,1)$ then $w$ is either $0$ or $1$ and thus $p:=w^\infty$ is one of the endpoint of $\cC$ which is already fixed by $F$.
Hence, $m_{w,z}$ is an irreducible P-module of dimension one.
However, the associated representation of $F$ is reducible and unitarily equivalent to $\chi_z^p\oplus \Ind_{F_{r}}^F \chi_z^r$ or $\chi_z^q \oplus \Ind_{F_{s}}^F \chi_z^s$, respectively, where $p=0^\infty, q=1^\infty, r=1\cdot 0^\infty, s=0\cdot 1^\infty$.
Note that even in the $(F,1)$-case two pairwise non-isomorphic P-modules yield two pairwise non-isomorphic representations except when $z=1$: $\Pi^F(m_{0,1})\simeq \Pi^F(m_{1,1})$ even though $m_{0,1}\not\simeq m_{1,1}$.

Using the Mackey--Schoda criterion we easily deduce that $\Pi^Y(m_{w,z})\simeq \Pi^Y(m_{v,y})$ if and only if $(w,z)=(v,y)$ for $Y=F,T,V$ (and $(Y,d)\neq (F,1)$) \cite{Mack51}, see also \cite{Burger-Harpe97} or the recent monograph \cite{Bekka-Harpe20}.
Now, we can observe that at the P-module level we have the same equivalence classes: $m_{w,z}\simeq m_{v,y}$ if and only if $(w,z)=(v,y)$.

{\it P-representations of $\cO$.}
Consider $(w,z)$ and $(v,y)$ as above and the representations $\Pi^\cO(m_{w,z}), \Pi^\cO(m_{v,y}).$
Note that these representations are irreducible since their restriction to $V$ are.
If $(w,z)=(v,y)$, then $m_{w,z}\simeq m_{v,y}$ implying that $\Pi^\cO(m_{w,z})\simeq \Pi^\cO(m_{v,y}).$
Otherwise, $\Pi^\cO(m_{w,z}), \Pi^\cO(m_{v,y})$ cannot be unitary equivalent because their restriction to $V$ are not.
We deduce that the same classification holds for $\cO$.

In summary, consider two irreducible atomic P-modules $m,\ti m$, $X=F,T,V,\cO$, and $\dim_P(m)=d$ is finite.
If $(X,d)\neq (F,1)$, then $m\simeq \ti m$ if and only if $\Pi^X(m)\simeq \Pi^X(\ti m)$ for $X=F,T,V,\cO$.
Moreover, $\Pi^X(m)$ is irreducible.

\noindent
{\it P-dimension of atomic P-representations.}
If $m_{w,z}$ is an atomic P-module as above with $(w,z)\in W_d\times \bS_1$, then $\dim_P(m_{w,d})=|w|=d$. 
Now, a finite dimensional atomic representation of $X$ is of the form $\Pi^X(\oplus_{i=1}^k m_{w_i,z_i})$ with $k$ finite, $w_i \in W_{d_i}$, and $z_i\in \bS_1$.
In particular, 
$$\dim_P(\Pi^X(\oplus_{i=1}^k m_{w_i,z_i}))=\dim_P(\oplus_{i=1}^k m_{w_i,z_i}) = \sum_{i=1}^k d_i$$ and thus is independent of $X=F,T,V,\cO$ and is additive.

Finally, let us show that the P-dimension is additive for \textit{all} P-representations which follows by a similar argument to Proposition \ref{prop:p-dim-sum}. Let $\sigma_{\atom}$ be an atomic P-representation and $\sigma_{\diff}$ a diffuse P-representation. We can assume that both representations have finite P-dimension and are irreducible. Let $\scrX$ be a complete full sub-module satisfying $\dim(\scrX) = \dim_P(\sigma_{\atom} \oplus \sigma_{\diff})$. By the general decomposition theorem for P-modules, we have $\scrX = \scrX_{\atom} \oplus \scrX_{\diff}$ where $\scrX_{\atom}$ (resp.~$\scrX_{\diff}$) is a full atomic (resp.~diffuse) P-module which induces $\sigma_{\atom}$ (resp. $\sigma_{\diff}$). From there, we can conclude from the previous results that $\dim_P(\sigma_{\atom} \oplus \sigma_{\diff}) = \dim_P(\sigma_{\atom}) + \dim_P(\sigma_{\diff})$.

Therefore, from the classification results of \cite{BW23} and the above elementary observations we can extend the main results of Section \ref{sec:classification}, thus completing the proof of Theorem \ref{thm:smallest-module} for the atomic case.

\subsection{Equivalence of categories for finite P-dimensions}

We now interpret the results developed so far from this section using the language of categories and complete the proof for Theorem \ref{thm:rigidity-category}. 
Recall from the introduction that $\Mod_{\full}^{\FD}(\cP)$ denotes the full subcategory of P-modules whose objects are {\it finite dimensional full} P-modules. 
Finally, write $\Mod^{\FD+}_{\full}(\cP)$ for the full sub-category of $\Mod_{\full}^{\FD}(\cP)$ obtained by removing P-modules that contains atomic P-modules of dimension one.
The full subcategories $\Rep^{\FD}(X), \Rep^{\FD+}(F)$ for $X=F,T,V,\cO$ are similarly defined.

\begin{theorem}\label{theo:equiv-category}
	The following assertions are true.
	\begin{enumerate}
		\item The functor $\Pi^Z:\Mod_{\full}^{\FD}(\cP)\to \Rep^{\FD}(Z)$ 
		is an \emph{equivalence of categories} for $Z=T,V,\cO$.
		\item The functor $\Pi^F:\Mod_{\full}^{\FD+}(\cP)\to \Rep^{\FD+}(F)$ is an \emph{equivalence of categories}.
		\item The categories $\Rep^{\FD}(Z)$ with $Z=T,V,\cO$ and $\Rep^{\FD+}(F)$ are all \emph{semisimple}.
	\end{enumerate}
\end{theorem}

\begin{proof}
Fix $Z=T,V,\cO$ and consider a full P-module $m$ of finite dimension.
There exists $k\geq 1$ and some irreducible P-modules $m_1,\dots,m_k$ such that $m= \oplus_{i=1}^k m_i$.
Now, from the definition of the Pythagorean functor we get that $\Pi^Z(m)= \oplus_{i=1}^k \Pi^Z(m_i)$.
If $m_i$ is diffuse, then by Corollary \ref{cor:min-invariant-subspace} we have that $\Pi^Z(m_i)$ is irreducible.
If $m_i$ is atomic, then by the classification of Section \ref{subsec:atomic-case} of all irreducible atomic P-modules of finite dimension and their associated representation  we deduce that $\Pi^Z(m_i)$ is irreducible.
Hence, in any case $\Pi^Z(m_i)$ is irreducible.
Moreover, again by applying Corollary \ref{cor:min-invariant-subspace} in the diffuse case and the classification of atomic P-modules and representations of Section \ref{subsec:atomic-case} we deduce that $\Pi^Z(m_i)\simeq \Pi^Z(m_j)$ if and only if $m_i\simeq m_j$.
From there (using Schur lemma) we easily deduce that $\Pi^Z:\Mod_{\full}^{\FD}(\cP)\to \Rep^{\FD}(Z)$ realises an equivalence of categories. This proves the first statement.
A similar argument gives the second statement.
We have recalled in the preliminaries that a full finite dimensional P-module decomposes into a finite direct sum of irreducible ones.
This means that $\Mod_{\full}^{\FD}(\cP)$ is semisimple.
Hence, the third statement follows from the first two since semisimplicity is preserved by equivalences of categories.
\end{proof}

\begin{remark}
Note that the existence of nontrivial residual subspaces in the decomposition of non-full P-modules prevents $\Mod^{\FD}(\cP)$ to be semisimple. In fact, this later category is not even {\it abelian} \cite[Chapter VII.3]{MacLane71}.
	Indeed, if $\fH:=\fH_0\oplus\fZ$ is a P-module with $\fZ$ a nontrivial residual subspace and $\fH_0$ a sub-module, we have that the obvious embedding $j:\fH_0\into\fH$ is {\it monic} and moreover it can be proved that its cokernel exists and is the zero map. This implies that $j$ is not a kernel; if it were it would be the kernel of a zero map implying that $\fH$ and $\fH_0$ are isomorphic which is absurd.
This can easily be checked on the example $\fH:=\C e_1\oplus\C e_2$ so that $A(e_1)=e_1=B(e_2), A(e_2)=0=B(e_1)$ and $\fH_0:=\C e_1, \fZ:=\C e_2$.
\end{remark}

\section{Classification of diffuse and atomic P-modules per dimension}
\label{sec:mani-pyth-rep}
In this section we unravel a beautiful geometric structure for irreducible classes of P-representations hence providing {\it moduli spaces}.
We will freely use elementary facts on differential manifolds and refer the reader, for details and proofs, to the book of Lee \cite{Lee12}.
All manifolds considered here are {\it real} and {\it smooth}, and all submanifolds are {\it embedded} (or {\it regular}): a subset $S$ of a manifold $M$ that is itself a manifold whose topology is the subspace topology and whose smooth structure yields a smooth embedding inside $M$. Note that an open subset of a manifold is a submanifold (and all submanifold $S\subset M$ with same dimension than $M$ arise that way, see \cite[Proposition 5.1]{Lee12}).
We may drop the terms ``smooth'' and ``embedded''.

\subsection{Manifolds of P-modules}
In \cite{BW22} (Section 4) it was shown that the set of diffuse one dimensional P-representations forms a manifold $\bS_3\backslash(\cC_1 \cup \cC_2)$: the real 3-sphere $$\bS_3=\{(a,b)\in \C^2:\ |a|^2+|b|^2=1\}$$
to which we remove the two circles $\cC_1, \cC_2$ given by $a=0$ or $b=0$, respectively.
Note that the two circles correspond to the atomic reducible representations of $F$ (which becomes irreducible when extended to $T,V,\cO$).
By Theorems \ref{theo:pythag-irrep} and \ref{theo:pythag-equiv}, every diffuse one dimensional P-representation is irreducible and moreover two of them are equivalent if and only if they correspond to the same point of the 3-sphere. 
Hence, $\bS_3\backslash(\cC_1 \cup \cC_2)$ is in bijection with all irreducible classes of diffuse representations of $F$ (resp.~$T,V,\cO$) of P-dimension one.
Remarkably, we shall show that this result can be extended for any finite P-dimension. This will not only provide a beautiful classification result but will also demonstrate how the number of equivalence classes of irreducible diffuse representations grows rapidly with P-dimension.

Rather than working with an abstract Hilbert space $\fH$, we now consider $\C^d$ for a certain dimension $d\geq 1$ equipped with the coordinate basis and the standard inner product. 
Recall that by construction of $\cP$ we have:

\begin{observation} \label{obs:pyth-pair-matrix}
Given $d\geq 1$, the set of P-modules $(A,B,\C^d)$ is in bijection with the set of isometric matrices in $M_{2d, d}(\C)$ (formed by stacking $A$ on top of $B$).
\end{observation}

We shall frequently identify P-modules with their associated rectangular matrices. 
We now define various classes of P-modules.

\begin{definition} \label{def:set-pyth-pair}
For any $d\geq 1$ we consider the following lattice of sets:
$$\Irr_{\diff}\sqcup \Irr_{\atom}=\Irr(d) \subset P_d\subset P_{\leq d}$$
where $P_{\leq d}$ is the set of P-module of the form $(A,B,\C^d)$ (and thus of P-dimension smaller or equal to $d$), $P_d$ the subset of full ones (or equivalently of P-dimension exactly $d$), $\Irr(d)$ the subset of irreducible ones, and finally $\Irr_{\diff}(d)$ and $\Irr_{\atom}(d)$ the irreducible diffuse and atomic ones, respectively.
\end{definition}

\begin{lemma}
	The set $P_{\leq d}$ forms a $3d^2$-dimensional submanifold of $M_{2d,d}(\C)$.
\end{lemma}

\begin{proof}
Define the map $f:M_{2d,d}(\C)\to H_d,\ R\mapsto R^*R$ where $H_d$ is the set of hermitian matrices of size $d$.
We consider $M_{2d,d}(\C)$ and $H_d$ as {\it real} vector spaces.
Observe that $f$ is smooth with derivative at $R\in M_{2d,d}(\C)$ being the real linear map:
$$df_R:M_{2d,d}(\C)\to H_d,\ X\mapsto R^*X+ X^* R.$$
Note that $f^{-1}(\{\id\})=P_{\leq d}$. Hence, to prove that $P_{\leq d}$ is a manifold of dimension $3d^2$ it is sufficient to prove that $df_R$ is surjective for all $R\in f^{-1}(\{\id\})$ by \cite[Corollary 5.14]{Lee12}.
If $K=K^*$ and $R^*R=\id$, then $df_R(RK/2)=K$ proving the surjectivity.
\end{proof}

\subsection{Manifolds of diffuse representations}
\subsubsection{Action of $\PSU(d)$ on P-modules.}
Let $\U(d)$ be the unitary group of size $d$ and consider the action by conjugation:
\[\alpha : \U(d) \act P_{\leq d},\ u \cdot (A,B,\C^d) := (uAu^*,uBu^*,\C^d).\]
This action factorises into an action of $\PSU(d):=\U(d)/\bS_1$ (the projective special unitary group) that we continue to denote by $\alpha$.
It is rather obvious that conjugating a P-module by a unitary preserves various properties such as being: atomic, diffuse, full, irreducible, and in particular preserves the P-dimension. We deduce the following fact.

\begin{lemma} \label{lem:alpha-act-restrict}
The subsets $P_d, \Irr(d), \Irr_{\diff}(d)$ and $\Irr_{\atom}(d)$ are all stabilised by $\alpha$ (and we write $\alpha$ for the corresponding restrictions).
\end{lemma}

\subsubsection{Diffuse submanifolds of $P_{\leq d}$.}

\begin{lemma}
	The set $\Irr(d)$ is an open subset of $P_{\leq d}$ and in particular is a submanifold of same dimension $3d^2$.
\end{lemma}

\begin{proof}
Fix a natural number $1\leq k\leq d-1$ and define $q_k\in M_{d,d}(\C)$ to be the orthogonal projection onto $V_k\subset \C^d$: the vector subspace spanned by the first $k$ vectors of the standard basis $(e_1,\cdots,e_d)$ of $\C^d$.
Observe that if $(A,B,\C^d)$ is a P-module, then $V_k$ is a sub-module of it if and only if we have
$$Aq_k=q_kAq_k \text{ and } Bq_k=q_kBq_k.$$
Subsequently, define the map
\[\phi_k : P_{\leq d} \rightarrow M_{d,d}(\C) \times M_{d,d}(\C),\ (A,B,\C^d) \mapsto (q_k^\perp A q_k, q_k^\perp Bq_k)\]
where $q_k^\perp:=I_d-p_k$ and set $P_{\leq d}^k$ to be the pre-image of $0$. 
By construction, $P_{\leq d}^k$ is the set of P-modules admitting $V_k$ as a sub-module.
Since $\phi_k$ is a continuous function and $P_{\leq d}$ is compact (inside $M_{2d,d}(\C)$) we deduce that $P_{\leq d}^k$ is compact. 

Let $\PSU(d)\cdot P_{\leq d}^k$ be the range of the map
$$\PSU(d)\times P_{\leq d}^k,\ (u,(A,B,\C^d))\mapsto (uAu^*,uBu^*,\C^d)$$
which is compact since $\PSU(d)$ and $P_{\leq d}^k$ are.
Note that $\PSU(d)\cdot P_{\leq d}^k$ is the set of all P-modules $(A,B,\C^d)$ admitting a sub-module of dimension $k$.
We deduce that $\Irr(d)$ is the complement of a finite union of compact subspaces $\cup_{k=1}^{d-1}\PSU(d)\cdot P_{\leq d}^k$ and thus must be open.
\end{proof}

\begin{lemma}\label{lem:atomic}
We have the partition 
$$\Irr(d)=\Irr_{\diff}(d)\sqcup \Irr_{\atom}(d)$$
where $\Irr_{\atom}(d)$ is a compact submanifold of dimension $d^2$ and $\Irr_{\diff}(d)$ is an open submanifold of dimension $3d^2$.
\end{lemma}

\begin{proof}
An irreducible P-module is either atomic or diffuse. Therefore, we have a partition as above.
Recall that $W_d$ is a set of representatives of prime words of length $d$ modulo cyclic permutations.
From \cite{BW23} we have that
$$\PSU(d)\times W_d\times \bS_1\to \Irr_{\atom}(d), (u,w,z)\mapsto u\cdot m_{w,z}$$ 
is a bijection, see notations from Section \ref{subsec:atomic-case}.
We immediately deduce that $\Irr_{\atom}(d)$ is a compact submanifold of $\Irr(d)$ of dimension $d^2$.
Hence, its complement $\Irr_{\diff}(d)$ must be open inside $\Irr(d)$.
Thus, $\Irr_{\diff}(d)$ is an open submanifold of $\Irr(d)$ of dimension $3d^2.$ 
\end{proof}

\subsubsection{Manifold of equivalence classes of diffuse irreducible representations.}

We prove the main theorem of this section using the following standard theorem on Lie group actions.

\begin{theorem}\label{theo:Lee}\cite[Theorem 21.10]{Lee12}
Let $\beta:G\act M$ be a smooth free proper action of a Lie group $G$ on a manifold $M$. 
Then the orbit space $G\backslash M$ is a topological manifold that admits a unique smooth structure for which the canonical map $M\onto G\backslash M$ is a submersion. Moreover, its dimension is $\dim(G\backslash M)=\dim(M)-\dim(G)$.
\end{theorem}

By combining Theorems \ref{theo:pythag-equiv} and \ref{theo:Lee}
we deduce the following.

\begin{theorem}\label{theo:diffuse-d}
	Fix a natural number $d\geq 1$. 
	The action $\alpha:\PSU(d)\act\Irr_{\diff}(d)$ is smooth, free and proper. 
The orbit space $\PSU(d)\backslash \Irr_{\diff}(d)$ has a natural structure of a manifold of dimension $2d^2+1$. 
	Moreover, it is in bijection with the irreducible classes of diffuse representations of $F$ with P-dimension $d$.
	
	The same statements hold if we replace $F$ by $T,V$ or $\mathcal O$. 
\end{theorem}

\begin{proof}
The action $\al:\PSU(d)\act \Irr_{\diff}(d)$ is a composition of matrix multiplications and taking adjoint of matrices. 
Hence, it defines a smooth action. 
Since $\PSU(d)$ is compact and $\alpha$ is continuous (because it is smooth) we deduce that $\alpha$ is proper.
Consider $u \in \PSU(d)$, $m \in \Irr_{\diff}(d)$ and suppose $u \cdot m = m$. 
If $u$ is nontrivial, then necessarily $m$ decomposes as a direct sum of two nontrivial sub-modules, a contradiction.
Therefore, $\alpha$ is free and satisfies all the assumptions of Theorem \ref{theo:Lee}.
Hence, the orbit space $\PSU(d)\backslash \Irr_{\diff}(d)$ is a smooth manifold with dimension equal to 
$$\dim(\Irr_{\diff}(d)) - \dim(\PSU(d)) = 3d^2 - (d^2-1) = 2d^2+1.$$
Thereafter the rest of the theorem follows from Theorem \ref{theo:pythag-equiv}.
\end{proof}

We now briefly prove Theorem \ref{thm:manifold-atomic} for the atomic case.

\begin{proof}[Proof of Theorem \ref{thm:manifold-atomic}]
The first item comes from \cite{BW23} and the observation that the map $(u,w,z)\mapsto u\cdot m_{w,z}$ is smooth.
By definition, $(u,w,z)\mapsto u\cdot m_{w,z}$ is also $\PSU(d)$-equivariant.
We obtain that $\{m_{w,z}:\ (w,z)\in W_d\times \bS_1\}$ is a set of representatives of the orbit space $\PSU(d)\backslash \Irr_{\atom}(d)$ which is isomorphic to the manifold $W_d\times \bS_1$.
The last statement follows from the main theorem of \cite{BW23}.
\end{proof}

\section{Representations of the Thompson groups and of the Cuntz alegbra in the literature}\label{sec:examples}
We present various families of representations of $F,T,V,\cO$ from the literature and interpret them using our language. 
We mainly enunciate facts without proofs but in the end answer various open questions. 
This last part is largely independent of the rest and provides a survey on representations of the Thompson groups and the Cuntz algebra. 

\subsection{Bratteli and Jorgensen's permutation representations of $\cO$}
Bratteli and Jorgensen introduced the widely celebrated family of \textit{permutation} representations which have had numerous important applications in wavelet theory and the study of fractals as iterated function systems \cite{Bratteli-Jorgensen-19}. These representations are defined generally for $\cO_N$; however, we shall only focus on the case when $N=2$ which can then be extended to the general case in the obvious way.
A representation $\pi$ of $\cO$ on $\scrH$ is said to be \textit{permutative} if there exists an orthonormal basis $\{\varphi_n\}_{n \in \N}$ for $\scrH$ that is stabilised by the generators of $\cO$; that is,
\[\pi(s_i)\varphi_k = \varphi_{m}\]
for all $i=0,1$, natural number $k \in \N$, and where $m=m(i,k)\in\N$ depends on $i,k$. 
The authors define the ``coding'' map 
$$\kappa: \N \rightarrow \cC:=\{0,1\}^{\N^*},\ n\mapsto x_1x_2\dots$$
from $\N$ to the Cantor space $\cC$ where
\[\pi(s_{x_k}^*\dots s_{x_2}^* s_{x_1}^*)\varphi_n \neq 0\]
for all $k \in \N$.

If $(\pi, \scrH)$ is a permutation representation, then the coding map simply identifies which ray each vector $\varphi_n$ is contained in. Furthermore, using the notation from Subsection \ref{subsec:diff-atom-rep}, we have $w\varphi_n = 0$ for all finite word $w$ that is not a finite prefix of $\kappa(n)$. This shows that $\pi$ is necessarily an atomic representation.

From this description it is not immediate but possible to find the P-dimension of $(\pi,\scrH)$ and the smallest complete sub-module of $(\pi(s_0)^*,\pi(s_1)^*,\scrH)$ (when it exists). 
Indeed, the family of vectors $(\varphi_n:\ n\in\N)$ permits to easily construct sub-modules of $\scrH$.
To any $S\subset\N$ we consider the closed linear subspace $\fH_S\subset\scrH$ spanned by the $\varphi_n$ with $n\in S$.
If for all $i=0,1$ and $n\in S$ there exists $m\in S$ satisfying that $\pi(s_i)^*\varphi_n = \varphi_m$, then $\fH_S$ is a sub-module of $\scrH$.
When all such $S$ are infinite we deduce that $\dim_P(\scrH)=\infty$. 
Here, we have that $p = \kappa(n)$ for some $n \in S$ and $p$ is necessarily aperiodic. 

Assume now that $(\pi,\scrH)$ is irreducible and there exists a finite $S$ so that $\fH_S$ is a sub-module of $\scrH$.
In that case there is a smallest such $S$, denoted $S_{\smal}$, whose associated span 
$$\fH_{\smal} = \textrm{span}\{\varphi_n:\ n\in S_{\smal}\}$$ 
is the desired smallest nontrivial (complete) sub-module of $\scrH$. 
Hence, the P-dimension of $\scrH$ is the cardinal of $S_{\smal}$.
Let $n \in S$ and $w$ be a period of $\kappa(n)$ (since $\pi$ is finite P-dimension, $\kappa(n)$ must be eventually periodic). Then $\fH$ is unitarily equivalent to the P-module $m_{w,1}$ with $(w,1)\in W_d\times \bS_1$ defined in Subsection \ref{subsec:atomic-case}. Hence, the restriction of $\pi$ to $F$ yields the quasi-regular representation $\lambda_{F/F_{p}}$ where $p=w^\infty$. 

It can be deduced that the permutation representations of Bratteli and Jorgensen are made of extensions to $\cO$ of {\it quasi-regular} representations of the Thompson groups (i.e.~the atomic representations without phases).

\subsection{Atomic representations of Dutkay, Haussermann and Jorgensen} 
Dutkay, Haussermann and Jorgensen expanded upon the seminal work of Bratteli and Jorgensen on permutation representations by introducing the family of purely atomic representations of $\cO$ in \cite{DHJ-atomic15}. The extension of the atomic representations to $\cO$ considered in the current authors' previous article \cite{BW23} are precisely the purely atomic representations from \cite{DHJ-atomic15}. 
The two studies share common features but interestingly some differences as well which we will briefly expose below.
Perhaps, the key differences between their studies and ours is that we have been mainly studying the small Hilbert space $\fH$ that carries the P-module structure in order to derive results on representations of the Thompson groups and the Cuntz algebra.
In contrast, the work of Dutkay--Haussermann--Jorgensen is more focused on the large Hilbert space $\scrH$ via the study of certain projectors in the atomic case (rather than the diffuse case).
Moreover, they only consider representations of the Cuntz algebra while we have been rather focusing in describing precisely the restrictions to $F,T,V$ of these representations.

Let $(\pi, \scrH)$ be a representation of $\cO$ with $S_i := \pi(s_i)$ for $i=0,1$. 
In \cite{DHJ-atomic15} the authors define a projection-valued measure $P$ on the Cantor space $\cC$. 
Then the following are defined:
\[S(\scrH) := \{p \in \cC : P(p) \neq 0\},\ \scrH_{\atom} := \oplus_{p \in S(\scrH)} P(p)\scrH,\ \scrH_{\textrm{cont}} := \scrH \ominus \scrH_{\atom}.\]
The authors then say $\pi$ is \textit{purely atomic} (or atomic for short) if $\scrH = \scrH_{\atom}$. Equivalently, $\pi$ is atomic if there exists a subset $\Omega \subset \cC$ such that $\oplus_{p \in \Omega}P(p) = \id$. In which case, $\Omega$ is called the support of $P$, $\supp(P)$.
It can be seen that $\scrH_{\atom}$ coincides with the definition of $\scrH_{\atom}$ given in \cite{BW23} while $\scrH_{\textrm{cont}}$ is the subspace $\scrH_{\diff}$. The authors then study atomic representations for the remainder of their article. 

Now suppose $(\pi, \scrH)$ is an atomic representation.
The authors define $\textrm{Orbit}(p)$ to be the set of all rays which are in the same tail equivalence class as $p$. From which, they decompose $\scrH_{\atom}$ into subrepresentations $P(\textrm{Orbit}(p))\scrH$ for rays $p$ in the support of $P$. 
The authors then prove that the subrepresentation $(\pi, P(\textrm{Orbit}(p))\scrH)$ is irreducible if and only if $\dim(P(p)) = 1$. Further, they show that two irreducible atomic representations, $\pi$ and $\ti\pi$, are equivalent if and only if:
\begin{enumerate}
	\item $\supp(P) = \supp(\ti P)$; and
	\item if $p \in \supp(P)$ is a periodic ray then there exists a unitary map $U : P(p) \rightarrow \ti P(p)$ such that $US_w = \ti S_w U$ on $P(p)$ where $w$ is the period of $p$.
\end{enumerate} 
The authors then conclude the main section of their paper by proving that when $\pi$ has {\it finite P-dimension} then
\[\oplus \{P(p)\scrH : p \textrm{ is a periodic ray in } \supp(P)\}\]
is the unique smallest nontrivial P-module in $\scrH$ which recovers the result of Theorem \ref{thm:smallest-module} in the present paper for the atomic case.

\begin{remark}Note that the permutation representations of Bratteli--Jorgensen are the extensions to $\cO$ of {\it quasi-regular} representations of the Thompson groups while the purely atomic representations of Dutkay--Haussermann--Jorgensen are the the extension to $\cO$ of {\it monomial} representations of the Thompson groups.
\end{remark}

\subsection{Barata--Pinto, Ara\'ujo--Pinto and Guimar\~aes--Pinto's representations of the Thompson groups}
Barata and Pinto in \cite{barata2019representations} studied the restriction of a specific class of Bratteli--Jorgensen's permutation representations of $\cO$ to obtain a family of irreducible pairwise inequivalent representations of $V$ indexed by the interval $[0,1]$. 
These representations were later proved to be irreducible and pairwise inequivalent when restricted to the smaller groups $F$ and $T$ by Guimar\~aes and Pinto \cite{Guimaraes-Pinto22}.

The class of permutation representations arise from the interval map $f(x) = 2x\ (\textrm{mod } 1)$ and acts on the Hilbert space $\scrH_x := \ell^2(\textrm{orb}(x))$ where $x \in [0,1]$ and $\textrm{orb}(x) = \{f^m(x) : m \in \Z\}$. The authors showed that the restriction to Thompson's group $V$ is given by the formula
\[\pi_x : V \rightarrow B(\scrH_x) : \pi_x(g)\delta_y = \delta_{g(y)} \textrm{ for all } g \in V, y \in \textrm{orb}(x)\]
where $\{\delta_y\}_{y \in \textrm{orb}(x)}$ is the canonical basis of $\scrH_x$. 

Observe that 
\[\langle \pi_x(g)\delta_y, \delta_y \rangle = 1\]
if $g(y) = y$ and otherwise equals $0$. Hence, let $p$ be the ray in $\cC$ associated to the real number $x$ obtained by its binary expansion (if $x$ is a dyadic rational the choice of binary representation is not important). Then it can be deduced that $\pi_x$ is equivalent to the quasi-regular representation of $V$ associated to $V/V_p$. This corroborates the above discussion on the general class of permutations representations from Bratteli and Jorgensen.

Note, the representations in \cite{barata2019representations} were later extended to the Higman--Thompson groups by Ara\'ujo and Pinto in \cite{Araujo-Pinto22} for which the above continues to hold after making suitable adjustments.

\subsection{Garncarek and Olesen's generalisation of the Koopman representation and Olesen's Bernoulli representations}
The classical Koopman representation of $F$ is the unitary representation $\kappa:F\act L^2([0,1])$ deduced from the usual action $F\act [0,1]$ that is defined by the formula:
$$\kappa(g)f= \left( \frac{ dg_*L}{dL} \right)^{1/2} \cdot f\circ g^{-1} \text{ for } g\in F, f\in L^2([0,1])$$
where $L$ is the Lebesgues measure.
Note that the action $F\act ([0,1],L)$ is {\it not} measure-preserving and thus we must introduce the Radon-Nikodym derivative term $\left( \frac{ dg_*L}{dL} \right)^{1/2}$ to ensure $\kappa$ is unitary.
This makes the formula more complicated but in fact is an advantage if one wants to deform $\kappa$.
Indeed, Garncarek constructed in \cite{Gar12} a one parameter family of irreducible representations of $F$ by {\it twisting} the Radon-Nikodym derivative term by a parameter $s\in\R$ as follows:
$$\kappa_s(g)f := \left( \frac{ dg_*L}{dL} \right)^{is}\cdot\kappa(g)f \text{ where } g\in F, f\in L^2([0,1]).$$

The set of one dimensional Pythagorean pairs can be naturally identified with the 3-sphere $\bS_3$.
Via this identification, it was observed by the current authors in \cite{BW22} that the family of representations $(\kappa_s:\ s\in [0,2\pi))$ forms the following circle $C$ in $\bS_3$:
$$C:=\{ (\frac{\omega}{\sqrt 2}, \frac{\omega}{\sqrt 2}):\ \omega\in \bS_1\}.$$
Olesen in his PhD thesis \cite{olesen2016thompson} (under the supervision of Haagerup and Musat, and following an advice of Monod) expanded on the pioneering work of Garncarek by generalising the representations above. 
Let $AC$ denote the set of increasing homeomorphisms $\psi$ of $[0,1]$ such that for all $g\in F$ we have that 
\[g^\psi := \psi^{-1} \circ g \circ \psi\]
is {\it absolutely continuous}. Then the family of representations $$(\kappa_s^\psi:\ s\in\R, \psi \in AC)$$ is given by 
$$\kappa_s^\psi(g) := \kappa_s(g^\psi) \text{ for } g\in F$$ 
where $\kappa_s$ is extended in the obvious way to be defined on absolutely continuous functions.

The author proved a nice characterisation of irreducibility and equivalence. The representations $\kappa_s^\psi$ are irreducible if and only if the action of $F^\psi$ on $([0,1], L)$ is ergodic. 
Further, if $t=s$, then $\kappa_s^\psi \cong \kappa_t^\varphi$ if and only if the measures $\psi_*L, \varphi_*L$ are equivalent. In the case when $t\neq s$, then the forward direction continues to hold; however, sufficient conditions for equivalence are unknown. 
While the above provide a nice interpretation of irreducibility and equivalence for $(\kappa_s^\psi)$ in terms of measure theoretic language, Olesen does note that it is not always obvious to distinguish between different classes of measures and thus can be difficult to detect equivalence classes of representations.

As one would expect, the above generalised family of representations are Pythagorean. To show this, we only need to provide an extension of $\kappa_s^\psi$ to a representation of $\cO$. 
We start by extending the Koopman representation of $F$ to the Cuntz algebra $\cO$.
Define the functions $g_i, g_i^*$ on $[0,1]$ for $i=0,1$:
\[g_0(x) = \chi_{[0, 1/2]}(x)2x,\ g_1(x) = \chi_{[1/2, 1]}(x)(2x-1),\ g_0^*(x) = \frac{x}{2},\ g_1^*(x) = \frac{x+1}{2}\]
where $x\in [0,1]$ and $\chi_A$ is the characteristic function for a subset $A \subset [0,1]$. Then define the isometries $S_i \in B(L^2([0,1]))$ by
\[S_i(f) = (g_i')^{1/2} \cdot f \circ g_i \textrm{ for all } f \in L^2([0,1]) \text{ and } i=0,1\]
whose adjoints are defined in the same way but replacing $g_i$ with $g_i^*$.
It is easy to verify that the above isometries defines a representation $\sigma:s_i\mapsto S_i,i=1,0$ of $\cO$ and is an extension of the Koopman representation of $F$. Note that the Koopman representation is equivalent to $\kappa_1^{\id}$ (where $\id$ is the identity function on $[0,1]$) and is also equivalent to the Pythagorean representation induced by $(1/\sqrt{2}, 1/\sqrt{2}, \C)$, see Proposition 6.2 in \cite{Brothier-Jones}. 

Now define $g_i^\psi := \psi^{-1} \circ g_i \circ \psi$ as a function on $[0,1]$ and define $(g_i^\psi)^*$ analogously. Finally, define the isometries $S_{i,s}^\psi \in B(L^2([0,1]))$ in the obvious way:
\[S_{i,s}^\psi(f) = ((g_i^\psi)')^{1/2+is} \cdot f \circ g_i^\psi \textrm{ for all } f \in L^2([0,1]).\]
It is not difficult to deduce that the above isometries define a representation $\sigma_s^\psi:\cO\act L^2([0,1])$ and its restriction to $F$ is equivalent to $\kappa_s^\psi$, thus proving that $\kappa_s^\psi$ is a Pythagorean representation.

While the above provides a very explicit description of how to extend $\kappa_s^\psi$ to $\cO$, it is not immediately clear nor trivial to find the P-dimension of $\kappa_s^\psi$ for a general $\psi \in AC$. Though, there are (nontrivial) constructible examples of $\psi$ such that $\kappa_s^\psi$ is diffuse and has infinite P-dimension. 
In these cases the work of Olesen goes beyond ours and provide answers on when {\it infinite diffuse} Pythagorean representations are irreducible and equivalent.

\noindent
{\bf Olesen's Bernoulli representations.}
We now discuss a subclass of the above representations which Olesen focuses on which are the so-called \textit{Bernoulli} representations. These are amenable to performing explicit calculations and are constructed by considering the family of Bernoulli measures $\{\mu_p\}_{p \in (0,1)}$ on the Cantor space $\{0,1\}^\N$: the infinite product measure $\beta_p^{\ot \N}$ on $\{0,1\}^{\N}$ so that $\beta_p(\{0\})=p$ for a fixed $0<p<1.$ 
The Bernoulli measure $\mu_p$ on $\{0,1\}^\N$ can be pulled back to a measure on $[0,1]$ from which an absolutely continuous function $\psi_p$ on $[0,1]$ can be obtained (see Remarks 5.1.27 and 5.2.10 in \cite{olesen2016thompson}). Then $\kappa_s^p := \kappa_s^{\psi_p}$ is termed a Bernoulli representation.
Using the result from Olesen stated in the previous subsection, it can be shown that the representations $$(\kappa_s^p: s\in [0, 2\pi), p \in (0,1))$$ are irreducible and pairwise unitarily inequivalent. 

We shall now specifically identify the P-modules that induce the Bernoulli representations. Continuing the notation from the previous subsection, we have that $(g_0^{\psi_p})^*(x) = px$ and $(g_1^{\psi_p})^*(x) = (1-p)x + p$ (see the discussion after Proposition 5.5.1 in \cite{olesen2016thompson}). Hence, the operators $(S_{i,s}^{\psi_p})^*$ are given by
\[(S_{0,s}^{\psi_p})^*(f)(x) = p^{1/2+is}f(px),\ (S_{1,s}^{\psi_p})^*(f)(x) = (1-p)^{1/2+is}f((1-p)x+p)\]
for $f \in L^2([0,1])$ and $x \in [0,1]$. Observe that
\[S_{0,s}^{\psi_p}\id = p^{1/2+is}\id \textrm{ and } S_{1,s}^{\psi_p}\id = (1-p)^{1/2 + is}\id.\]
Thus, $(p^{1/2+is}, (1-p)^{1/2+is}, \C\id)$ is a (complete) sub-module of $(S_{0,s}^{\psi_p}, S_{1,s}^{\psi_p}, L^2([0,1]))$. Then by Observation \ref{obs:tau-mod-sub-rep} it follows that the Pythagorean representation induced by the one dimensional P-module $(p^{1/2+is}, (1-p)^{1/2+is}, \C)$ is equivalent to $\kappa_s^p$. In fact, it can be seen that the Bernoulli representations are precisely all the diffuse Pythagorean representations arising from the one dimensional P-modules $(a,b,\C)$ such that $a$ and $b$ have same argument. 

\subsection{The Hausdorff representations of Mori, Suzuki and Watatani}
Mori, Suzuki and Watatani studied in \cite{MSW-07} representations on fractal sets from which they defined the family of \textit{Hausdorff} representations of $\cO_N$. We shall again only focus on the case $N=2$ as the general case can be easily deduced from it.

Fix $n\geq 1$ and consider {\it similitudes} $\sigma_i : \R^n \rightarrow \R^n$ for $i=0,1$ such that
\[\norm{\sigma_i(x) - \sigma_i(y)} = \lambda_i \norm{x - y}\]
for all $x,y \in \R^n$ where $0 < \lambda_i < 1$ are fixed (Lipschitz) constant. Then it is a standard result that there exists a unique compact set $K \subset \R^n$ such that 
\[K = \sigma_0(K) \cup \sigma_1(K)\]
which is called a {\it self-similar fractal set}.
We further assume that there exists $V\subset \R^n$ that is open and non-empty, so that $\sigma_i(V),i=0,1$ are disjoint and contained in $V$.
Now let $d := \dim_H(K)$ be the {\it Hausdorff dimension} of the fractal set $K$ and let $\mu$ be the Borel measure on $K$ given by the restriction of the normalised $d$-dimensional Hausdorff measure on $K$. It is then known that the Hausdorff dimension $d$ is the unique constant such that
\[\lambda_0^d + \lambda_1^d = 1\]
and it can be shown that $\mu(\sigma_0(K) \cap \sigma_1(K)) = 0$. Therefore, a representation $\pi_\mu : \cO \act L^2(K, \mu)$ can be defined given by the formula
\[\pi_\mu(s_i)f = g_i \cdot f \circ \kappa_i^{-1} \textrm{ for all } f \in L^2(L, \mu)\] 
where $g_i \in L^2(K, \mu)$ such that it is valued $\lambda_i^d$ on $\sigma_i(K)$ and $0$ elsewhere. The representation $\pi_\mu$ is called the {\it Hausdorff representation} for the system $\{\sigma_0, \sigma_1\}$.

The main result in \cite{MSW-07} is to precisely classify the equivalence classes of all Hausdorff representations (Theorem 3.1). This is achieved by proving that the Hausdorff representations are equivalent to a class of representations defined earlier in the paper, called the Bernoulli representations (see Example 2). These Bernoulli representations are constructed differently from the ones of Olesen.  
Indeed, the ones from Mori, Suzuki and Watatani are defined on $L^2(X,\nu)$ where $X = \{0,1\}^\N$ is the Cantor space and $\nu$ is the Bernoulli measure for some $p \in (0,1)$. The adjoints of the two isometries that induce the representation are given by
\[S_0^*(f)(x_1x_2\dots) = p^{1/2}f(0x_1x_2\dots),\ S_1^*(f)(x_1x_2\dots) = (1-p)^{1/2}f(1x_1x_2\dots)\]
for infinite binary strings $x_1x_2\dots\in X$ and $f\in L^2(X,\nu)$.
In Theorem 3.1, the authors show that a Hausdorff representation for a system $\{\sigma_0, \sigma_1\}$ with associated Lipschitz constants $\lambda_i$ and Hausdorff dimension $d$ is equivalent to the Bernoulli representation associated to $(\lambda_0^d, \lambda_1^d)$. Subsequently, the list $(\lambda_0^d, \lambda_1^d)$ forms a complete invariant for the Hausdorff representation. 
Conversely, in Proposition 3.2 the authors show that for any probability vector $(p, 1-p)$ with $0<p<1$ and any $0 < d \leq 1$, there exists a system of proper contractions $\{\sigma_0, \sigma_1\}$ such that 
\[(p, 1-p) = (\lambda_0^d, \lambda_1^d).\]
Therefore, the equivalence classes of Hausdorff representations are indexed by real numbers in $(0,1)$. However, it should be noted that the equivalence classes does not preserve the Hausdorff dimension of the underlying fractal set (see Example 3).

\noindent
{\bf Comparison with Olesen's representations and our classification.}
As one may expect, the Hausdorff representations of Mori--Suzuki--Watatani are an extension of (a subclass of) Olesen's Bernoulli representations of $F$ to representations of $\cO$. This can be easily seen by noting that if $1 \in L^2(X, \nu)$ is the function of constant value one on $X$ then $S_0^*(f)1 = p^{1/2}\cdot 1$ and $S_1^*(f)1 = (1-p)^{1/2}\cdot 1$. Hence, by a similar argument as from the previous subsection, the above representation is equivalent to the Pythagorean representation associated to $(p^{1/2}, (1-p)^{1/2}, \C)$. 
In particular, they are all diffuse representations of Pythagorean dimension one.
Interestingly, the Bernoulli representations from \cite{olesen2016thompson} and the Hausdorff representations from \cite{MSW-07} coincide despite arising from different motivations: Mori--Suzuki--Watatani were motivated by the study of representations on fractal sets while Olesen was inspired by the representations of Garncarek which are an analogue of the principal series representations of the Lie group $\text{PSL}(2,\R)$. Though, it is noted that the Bernoulli representations of \cite{olesen2016thompson} are slightly more generalised for it allows for the two Pythagorean operators to be twisted by multiplying by a common scalar in $\bS_1$.

\noindent{\bf Answer to a problem of Aiello--Conti--Rossi.}
In the recent survey \cite{Aiello-Conti-Rossi21} on the Cuntz algebra, the authors ask whether the Bernoulli representations of \cite{MSW-07} (and thus of \cite{olesen2016thompson}) are permutation representations (Question 9.41).
Our results show that the answer is no. Moreover, our analysis permits to deduce the stronger result: the restriction to $F$ of any Bernoulli representation is not unitarily equivalent to any permutation representations.
Indeed all Bernoulli representations (in the sense of \cite{MSW-07} or \cite{olesen2016thompson}) are {\it diffuse} representations of {\it P-dimension one}.
In contrast, the permutation representations are {\it atomic} and have {\it various Pythagorean dimensions}. 

\subsection{Aita--Bergmann--Conti's induced product representations and Kawamura's generalised permutation representations.}
Kawamura constructed a rich family of \textit{generalised permutations} (GP for short) of $\cO_N$ in \cite{Kawamura-05} which are inspired by the influential work of Bratteli and Jorgensen. 
To avoid introducing further notation, rather than providing the original description of these GP representations, we shall instead immediately state the P-modules from which they arise from. Further, for simplicity we shall only consider the case when $N=2$; however, it can be easily extended to the general case in the obvious way.

Consider $d\in\N^*\cup\{\infty\}$, the Hilbert space $\fH_d:=\ell^2(\{1,2,\cdots,d\})$ with orthonormal basis $\{e_1,e_2,\cdots\}$, and a $d$-tuple $z=(z_1,z_2,\cdots)$ where $z_k=(a_k,b_k)\in\bS_3$ (i.e.~$|a_k|^2+|b_k|^2=1)$ for all $k$.
Define now the weighted shift operators $A_z,B_z\in B(\fH_d)$ defined as
$$A_ze_k=a_{k+1}e_{k+1} \text{ and } B_ze_k=b_{k+1} e_{k+1}$$
taking indices modulo $d$ if $d$ is finite.
This yields a P-module $m_z:=(A_z,B_z,\fH_d)$.
The associated P-representation $\kappa_z$ is equivalent to a GP representation and $z$ is referred to as the GP vector. Conversely, all GP representations arise in this way.
Observe that when $d=1$ then the GP representations are precisely all the one dimensional P-representations.

\noindent
{\bf Aita--Bergmann--Conti's induced product representations.}
Aita, Bergmann, and Conti introduced in \cite{Aita-Bergmann-Conti97} the family of {\it induced product} representations (IP for short) of the \textit{extended Cuntz algebra} (also referred to as the generalised Cuntz algebra). The extended Cuntz algebra $\cO_\fH$ involves considering an initial Hilbert space $\fH$ and taking an inductive limit
\[\cdots \rightarrow B(\fH^{\oplus r}, \fH^{\oplus r+k}) \rightarrow B(\fH^{\oplus r+1}, \fH^{\oplus r+k+1}) \rightarrow \cdots\]
relative to some isometric embedding (for more details see \cite{CDPR94} where these $C^*$-algebras were first constructed). When the dimension of $\fH$ is finite then $\cO_\fH$ is isomorphic to the usual Cuntz algebra $\cO_{\dim(\fH)}$; however, when the dimension is infinite then $\cO_\fH$ is an extension of $\cO_\infty$. The IP representations are constructed by taking any infinite tuple of unit vectors $z = (z_1, z_2, \dots)$ in $\fH$ which yields an IP representation $\pi_z$ of $\cO_\fH$. 
In \cite{Aita-Bergmann-Conti97} the authors took $\fH = L^2(G)$ for some locally compact group $G$ and were able to derive an equivalent condition for $G$ to be amenable based on the existence of an IP representation satisfying certain properties. Bergmann and Conti then further investigated in a subsequent article \cite{Bergmann-Conti03} the intertwinner spaces of the IP representations.

Kawamura showed that when $\fH = \C^d$ then the IP representation $\pi_z$ is equivalent to the GP representation $\kappa_z$ of $\cO_d$, see \cite[Theorem 3.3]{Kawamura-05}. Hence, all IP representations are equivalent to a GP representation associated to an infinite GP vector. Hereby, for the remainder of this subsection, when $z$ is an infinite vector we shall use $\kappa_z$ to denote both the GP and IP representation associated to $z$. We shall continue to refer to $z$ as being the GP vector and when $z$ is finite then $\kappa_z$ will solely refer to the associated GP representation.

\noindent
{\bf Atomic and diffuse IP and GP representations.}
From an {\bf irreducible} representation $\kappa$ we determine which are atomic and diffuse. First suppose $\kappa$ has a {\it finite} GP vector $z = (z_1, z_2, \dots, z_d)$ with $z_k = (a_k, b_k)$. Then $\kappa$ is atomic if and only if for each $k = 1, \dots, d$ either $\vert a_k \vert = 1$ or $\vert b_k \vert = 1$. In this case, set 
$$
w_k =\begin{cases} 0 \text{ if } |a_k|=1 \\
1 \text{ if } |b_k|=1
\end{cases} 
\text{and } 
\varphi_k =\begin{cases} a_k \text{ if } |a_k|=1 \\
b_k \text{ if } |b_k|=1
\end{cases} 
$$
and form:
$$w = w_1\dots w_d\in\{0,1\}^d \text{ and } \varphi:=\prod_{k=1}^d \varphi_k\in\bS_1.$$
Then the P-module $m_z$ is unitarily equivalent to $Z \cdot m_{w}$ where $Z=\Diag(\varphi_1,\cdots,\varphi_d)$ is the unitary diagonal matrix with entries $\varphi_k$ (see the notation from Subsection \ref{subsec:atomic-case}).
Furthermore, the restriction of $\kappa$ to $V$ is equivalent to the irreducible monomial representation $\Ind_{V_w}^V \chi_\varphi^w$. In the special case that $\varphi= 1$ then $\pi$ is a permutation representation of Bratteli and Jorgensen. 
The converse situation to the above is if at least one of the pairs $(a_k, b_k)$ have nonzero values, for both $a_k$ and $b_k$, in which case $\kappa$ is a diffuse representation. In particular, if $d=1$ and $\Arg(a_1) = \Arg(b_1)$ then we recover the Bernoulli representations of Olesen and Mori--Suzuki--Watatani.

Now suppose the GP vector of $\kappa$ is infinite (here $\kappa$ refers to both the IP and GP representation). 
Then $\kappa$ is atomic if and only if there exists a sequence $\varphi_k \in \{a_k, b_k\}$ such that $\prod_{k=1}^\infty \varphi_k$ is nonzero. If so, set $p_k = 0$ if $\varphi_k = a_k$ otherwise set $p_k = 1$ and form the ray $p = p_1p_2\dots$. Then the restriction of $\kappa$ to $V$ is equivalent to the irreducible quasi-regular representation $\lambda_{V/V_p}$ (hence the phase disappears). 
In particular, this shows that $\kappa$ is a permutation representation in the sense of Bratteli--Jorgensen.

\noindent
{\bf Classification of IP and GP representations.}
The GP vector $z$ is called periodic if there exists a vector $y$ and $n > 1$ such that $y^n = z$ and is called asymptotically periodic if there exists a positive integer $p$ such that 
$$\sum_{n=1}^\infty \left( 1- \vert \langle z_n, z_{n+p} \rangle \vert  \right)< \infty.$$
Kawamura and Bergmann--Conti derived very powerful results regarding the irreducibility and equivalence of GP and IP representations. 

In the {\it finite} case, Kawamura proved that $\kappa_z$ is irreducible if and only if $z$ is not periodic. If $z$ is periodic with $z = y^n$ and $y$ is non-periodic then the GP representation can be decomposed as a direct sum of irreducible GP subrepresentations:
\[\kappa_z \cong \oplus_{k=0}^{n-1} \kappa_{y_k} \text{ where } y_k:=e^{2\pi ik/n}y.\]
Furthermore, Kawamura showed that two GP representations with finite GP vectors are equivalent if and only if their GP vectors are equal up to cyclic permutation (and are necessarily of equal length). 

In the {\it infinite} case, Bergmann--Conti in \cite{Bergmann-Conti03} proved that $\kappa_z$ is irreducible if and only if $z$ is not asymptotically periodic. 
Additionally, they show that $\kappa_y$ and $\kappa_{z}$ are equivalent if and only if there exist two non-negative integers $p,q$ such that 
$$\sum_{n=1}^\infty \left( 1-\vert \langle y_{n+p}, z_{n+q} \rangle \vert \right) < \infty.$$
Kawamura extended the above results by proving that $\kappa_z$ is not equivalent to any representation with finite GP vector.

\noindent
{\it P-dimension.}
For a finite non-periodic GP vector, it is not difficult to verify that the associated P-module is irreducible. Hence, from the above decomposition result and from the results of the present paper, it follows that for any GP vector of finite length $d$, the associated representation  has P-dimension $d$.  
Now, consider a GP vector $z$ of infinite length with associated P-module $(A_z, B_z, \fH_z)$. Let $\fK \subset \fH_z$ be a finite dimensional subspace and denote $P_{\fK}$ to be the projection from $\fH$ to $\fK$. Since $A_ze_k, e_k$ are orthogonal vectors for all $k \in \N^*$ it can be observed that $\lim_{n\to\infty} P_{\fK}(A^nz) = 0$ for all $z \in \fK$. Hence, it can be deduced that $\fK$ cannot be a sub-module of $\fH_z$ and thus the representation $\kappa_z$ has infinite P-dimension. In particular, all IP representations have infinite P-dimension.

The results of Bergmann--Conti and Kawamura concerning finite GP vectors giving atomic representations corroborate ours.
If $z$ is an infinite GP vector that is not asymptotically periodic and so that all the products $\prod_k \varphi_k$ are equal to zero (for all possible choices of $\varphi_k\in\{a_k,b_k\}, k\geq 1$), then by Bergmann--Conti its associated representation is diffuse and irreducible (and necessarily of infinite P-dimension). 
Therefore, Bergmann--Conti's results go beyond ours and provide interesting insights for classifying infinite P-dimensional representations.

{\bf Restriction to the Thompson groups.}
Kawamura and Bergmann--Conti exclusively studied the representations of $\cO$ and did not consider their restriction to the Thompson groups. Note that proving irreducibility and inequivalence is harder for restrictions of representations to smaller groups. Hence, the classification results of Bergmann--Conti on (infinite P-dimensional) diffuse GP representations do not immediately extend to the Thompson groups. However, we are able to easily show that they do extend using the results in the present paper. Indeed, we are able to apply Theorem \ref{theo:isom-diff-functor} which asserts that the representation theory for $F,T,V,\cO$ are identical for \textit{all} diffuse representations. Hence, we obtain that the infinite P-dimensional diffuse GP representations of the Thompson groups are irreducible and equivalent if and only if their extension to $\cO$ are irreducible and equivalent, respectively. In particular, to the best of our knowledge, this provides the first \emph{explicit} family of irreducible and pairwise inequivalent diffuse representations of $F$ that have infinite P-dimension. 

\subsection{Extension of Pythagorean representations to the 2-adic ring C*-algebra of the integers.}
The {\it 2-adic ring $C^*$-algebra of the integers} $\cQ$ is the universal $C^*$-algebra generated by a unitary $u$ and an isometry $s_1$ satisfying the two relations
\[s_1u = u^2s_1 \text{ and } s_1s_1^* + us_1s_1^*u^* = 1.\]
This algebra was closely studied by Larsen and Xi in \cite{Larsen-Li-12} (but note that it previously appeared among different classes of C*-algebras, see Remark 3.2 of \cite{Larsen-Li-12}).
The C*-algebra $\cQ$ enjoys similar properties to $\cO$ such as being simple, purely infinite and nuclear. Interestingly, there is a natural embedding of $\cO$ inside $\cQ$ via
\[s_0 \mapsto us_1,\ s_1 \mapsto s_1.\]
Henceforth, we shall identify $\cO$ with its copy inside $\cQ$. Surprisingly, while the embedding is straightforward, the properties of the inclusion is not well understood. For a discussion on the known properties and open questions, we recommend the reader the recent survey by Aiello, Conti and Rossi \cite{Aiello-Conti-Rossi21} (see Section 9). 

One area of particular interest is the relationship between the categories $\Rep(\cO)$ and $\Rep(\cQ)$. It is known that every permutation representation of $\cO$ can be (but not necessarily uniquely) extended to a representation of $\cQ$, which shall also be called a permutation representation. Further, the restriction of irreducible permutation representations of $\cQ$ to $\cO$ preserves equivalence classes. The restriction also preserves irreducibility with the only exception being the so-called \textit{canonical} representation.

For diffuse representations we shall show there exists an even stronger rigidity statement which extends Theorem \ref{theo:isom-diff-functor}. This result is essentially an application of Proposition 4.1 in \cite{Larsen-Li-12}. However, we provide a brief proof (which is in the same vein as the original proof by Larsen and Li) to demonstrate the power of the diagrammatic calculus we have developed. 

Note, we say a representation of $\cQ$ is {\it diffuse} if its restriction to $\cO$ is diffuse and we define $\Rep_{\diff}(\cQ)$ as the corresponding full subcategory of $\Rep(\cQ)$.

\begin{theorem}\label{theo:O-Q}
	The categories $\Rep_{\diff}(\cO)$ and $\Rep_{\diff}(\cQ)$ are isomorphic.
\end{theorem}

\begin{proof}
Restricting representations of $\cQ$ to its subalgebra $\cO$ provides the obvious forgetful functor $\Rep_{\diff}(\cQ)\to\Rep_{\diff}(\cO)$.
Let $(\sigma^\cO, \scrH)$ be a diffuse representation of $\cO$. 
Recall our notations: $\sigma^\cO(s_i) = \tau^*_i$ for $i=0,1$ for the partial isometries of Subsection \ref{sec:partial-isometries}. We wish to construct a unitary operator $U$ corresponding to the generator of $\cQ$ using the diagrammatic partial isometries $\tau_i$. 
Define $U \in B(\scrH)$ as
	\[U = \tau_0^*\tau_1 + \lim_{n\to\infty} \sum_{k=1}^n \tau^*_{1^k0}\tau_{0^k1}\]
	where the limit is taken in the strong operator topology. The mapping of the different vertices in the expression for $U$ are shown below where red dots are moved to blue dots.
	
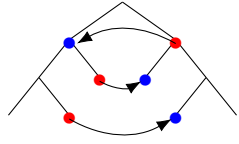
\begin{center}
		\begin{tikzpicture}
			\draw (0,0)--(-.7, -.5);
			\draw (0,0)--(.7, -.5);
			\draw (-.7, -.5)--(-1.1, -1);
			\draw (-.7, -.5)--(-.3, -1);
			\draw (-1.1, -1)--(-1.5, -1.5);
			\draw (-1.1, -1)--(-.7, -1.5);
			\draw (.7, -.5)--(1.1, -1);
			\draw (.7, -.5)--(.3, -1);
			\draw (1.1, -1)--(1.5, -1.5);
			\draw (1.1, -1)--(.7, -1.5);
						
			\node[] at (-.7, -.55) {$\textcolor{blue}\bullet$};
			\node[] at (-.3, -1.05) {$\textcolor{red}\bullet$};
			\node[] at (-.7, -1.55) {$\textcolor{red}\bullet$};
			\node[] at (.7, -.55) {$\textcolor{red}\bullet$};
			\node[] at (.3, -1.05) {$\textcolor{blue}\bullet$};
			\node[] at (.7, -1.55) {$\textcolor{blue}\bullet$};
			
			\draw[-{Latex[length=2mm]}] (.7, -.55) to[bend right] (-.6, -.55);
			\draw[-{Latex[length=2mm]}] (-.3, -1.05) to[bend right] (.25, -1.05);
			\draw[-{Latex[length=2mm]}] (-.7, -1.55) to[bend right] (.65, -1.55);
			
		\end{tikzpicture}%
	\end{center}
	We claim that $U$ is the unitary operator we are after. Indeed, the assumption of diffuseness ensures that the limit in the strong operator topology exists and thus $U$ is well-defined. Further, from the above diagram it is evident that each of the partial isometries in the expression for $U$ have pairwise orthogonal initial (resp. final) spaces. Again, the diagram and diffuseness assures that the direct sum of the initial (resp. final) spaces is equal to the entire space. Therefore, it can be concluded that $U$ is a surjective isometry and thus is unitary.
	
	From there, it is straightforward to verify that $\tau^*_1U = U^2\tau^*_0$, $\tau_1UU^*\tau_1^* = \id$. Hence, from $\sigma^\cO$ we obtain a representation $\sigma^\cQ$ acting on the \textit{same} Hilbert space.
	Additionally, $\tau_0^* = U\tau_1^*$ which shows that $\sigma^\cQ$ is the (unique) extension of $\sigma^\cO$ to $\cQ$.
	Furthermore, by the canonical construction of $U$ (it does not depend on the specific representation of $\cO$), any intertwinner between diffuse representations of $\cO$ extends to an intertwinner between the associated representations of $\cQ$. Hence, the above mapping defines a functor $\Rep_{\diff}^{\cO} \rightarrow \Rep_{\diff}^\cQ$ which is bijective on the morphism spaces and on objects. This completes the proof.
\end{proof}

This rigidity statement combined with our analysis on Pythagorean representations yields powerful results for classifying representations of the 2-adic ring C*-algebra.

\begin{remark}\label{rem:2-adic}
	From the above proof, it is apparent that the requirement of the representations being diffuse can be relaxed to only require that the representations arise from P-modules such that $A^n, B^n \xrightarrow{s} 0$ (equivalently, $\sigma^\cO(s_i^*)^n \xrightarrow{s} 0$). By \cite{BW23}, this is equivalent to the restriction of the representations to $F$ being weakly mixing (i.e.~does not contain any finite dimensional subrepresentations).
\end{remark}

\noindent
{\bf Partial answers to questions of Aiello--Conti--Rossi.}
In \cite{Aiello-Conti-Rossi21} the authors ask whether the results regarding the extension of permutation representations to $\cQ$ in the discussion preceding the proposition can be extended more generally. In particular, they ask the following (Questions 9.37, 9.38).

Let $\sigma_1, \sigma_2$ be irreducible representations of $\cQ$.
	\begin{enumerate}
		\item If their restrictions to $\cO$ are equivalent, does this imply that $\sigma_1, \sigma_2$ are equivalent?
		\item If $\sigma_1$ is not equivalent to the canonical representation then is its restriction to $\cO$ also irreducible?
	\end{enumerate}

The above theorem shows that this is indeed the case for the large class of diffuse representations (and can be further extended as explained in the above remark). Moreover, every representation of $\cO$ is a direct sum of a diffuse and atomic representation by \cite{BW22}. Therefore, it only remains to check the above questions for atomic representations. However, atomic representations can be considered as generalisations of permutation representations, where additional twists by scalars in $\bS_1$ are introduced. Hence, for the first question we anticipate an affirmative answer. For the second question we also anticipate an affirmative answer after additionally excluding the atomic representations that are derived by twisting the {\it canonical} representation.

\subsection{Jones' Wsyiwyg representations of $F$}

We now provide a family of representations constructed by Jones that are {\it not} Pythagorean \cite{Jones21}.
The celebrated theorem of Jones shows that for any $\delta$ in the set $$J:=\{2\cos(\pi/n):\ n\geq 3\}\cup [2,\infty)$$
there exists a subfactor with index $\delta^2$ and no other indices are allowed \cite{Jones83}.
In particular, at each $\delta\in J$ we may form the so-called Temperley--Lieb--Jones planar algebra $\TLJ_\delta$ made of string diagrams (from Kauffman diagrams \cite{Kauffman90}) and having a {\it positive definite} trace.
In each of $\TLJ_\delta$ there is a category $\cC_\delta$ generated (diagrammatically) by a single trivalent element as defined in \cite[Proof of Proposition 4.3]{Morrison-Peters-Snyder17}.
Using Jones' technology, by sending the tree with two leaves $Y$ to this trivalent vertex we obtain a unitary representation $\pi_\delta:F\act \scrH_\delta$ of Thompson's group $F$ that may be extended to $T$ using the planar algebra rotation tangle.
Visually, this procedure consists in taking a tree diagram describing an element of $F$ and to interpret it as a string diagram in the category $\cC_\delta$. 
This produces a scalar which corresponds to a matrix coefficient of the representation $\pi_\delta$ of $F$.
For this reason, Jones called them {\it Wysiwyg representations}: ``what you see is what you get''.

Jones found that there exists a unit vector $\Psi_\delta\in\scrH_\delta$ which satisfies that 
\begin{equation}\label{eq:vector}\pi_\delta(x_0)\Psi_\delta = \Psi_\delta \textrm{ and } \ker(\pi_\delta(x_0) - \id) = \C\Psi_\delta\end{equation}
where $x_0$ is the classical first generator of Thompson's group $F$.
He restricted the carrier Hilbert space $\scrH_\delta$ of the representation to the cyclic component associated to $\Psi_\delta$; we continue to denote by $\pi_\delta:F\act\scrH_\delta$ this subrepresentation.
Equation \ref{eq:vector} implies that $\pi_\delta$ is irreducible and moreover $\pi_\delta$ remembers all numerical invariant obtained from the vector $\Psi_\delta$.
This implies that the family $(\pi_\delta:\ \delta\in J)$ are pairwise unitary inequivalent.
This is such a beautiful result: planar algebras permit to construct a one parameter family of pairwise inequivalent irreducible unitary representations of Thompson's group $F$ that are indexed by the celebrated set $J$!

Note that these representations are constructed using Jones' technology (from a functor starting from the category of binary forests and ending in $\cC_\delta$). Although, this functor is very different from the ones considered in this article and indeed the representations $\pi_\delta$ obtained are not Pythagorean as explained below.

\begin{lemma}
	For any $\delta\in J$, the Wsyiwyg representation $\pi_\delta$ is not equivalent to any Pythagorean representation.
\end{lemma}

\begin{proof}
Fix $\delta\in J$ and consider the associated Wysiwyg representation $\pi_\delta:F\act\scrH_\delta$ with cyclic vector $\Psi_\delta$.	 
Suppose that $\pi_\delta$ was equivalent to some Pythagorean representation $\sigma:F\act\scrH$. Then there exists a nonzero vector $\Omega \in \scrH$ such that $\sigma(x_0)\Omega = \Omega$ and $\ker(\sigma(x_0) - \id) = \C\Omega$ where $x_0$ is defined as above and its realisation as a tree diagram is shown below. 
	\begin{center}
		\begin{tikzpicture}[baseline=0cm, scale = 1]
			\draw (0,0)--(-.5, -.5);
			\draw (0,0)--(.5, -.5);
			\draw (-.5, -.5)--(-.9, -1);
			\draw (-.5, -.5)--(-.1, -1);
			
			\node[label={\normalsize $x_0 = ($}] at (-1.5, -1) {};		
			\node[label={\normalsize $,$}] at (.65, -1) {};		
		\end{tikzpicture}%
		\begin{tikzpicture}[baseline=0cm, scale = 1]
			\draw (0,0)--(-.5, -.5);
			\draw (0,0)--(.5, -.5);
			\draw (.5,-.5)--(.1,-1);
			\draw (.5,-.5)--(.9,-1);
			
			\node[label={\normalsize $)$}] at (1.1, -1) {};
		\end{tikzpicture}%
	\end{center}
	Since $\sigma$ is irreducible and Pythagorean it must be either diffuse or atomic.
	If $\sigma$ is diffuse then from Proposition 2.8 in \cite{BW22}, the space $\ker(\sigma(x_0) - \id)$ must be trivial, a contradiction. 
	Hence, $\sigma$ must be atomic if $\sigma$ is Pythagorean. 
	As a simple observation, either $\tau_{0}(\Omega)$ or $\tau_{1}(\Omega)$ must be nonzero. By the symmetry of the tree diagram for $x_0$, without loss of generality  suppose we have the former. Then since $\sigma(x_0)\Omega = \Omega$ this implies $\tau_0(z) = z$ where $z = \tau_0(\Omega)$. However, this also implies that $\sigma(g)z = z$ for all $g\in F$. Hence, the subspace $\C z \subset \scrH$ forms a one dimensional subrepresentation of $\sigma$ which is equivalent to the trivial representation. This contradicts the fact that $\pi_\delta$ is irreducible. Therefore, we conclude that $\pi_\delta$ cannot be equivalent to a Pythagorean representation.
\end{proof}

\subsection{The regular and induced representations of Thompson's groups.}

To conclude we mention some classical representations of Thompson's groups that are not Pythagorean. As noted in \cite{Brothier-Jones} and \cite{BW22}, all Pythagorean representations are not mixing. Indeed, any matrix coefficient of any P-representation of $F$ does not vanish at infinity. 
In particular, the left-regular representation $\lambda_F:F\act\ell^2(F)$ of $F$ is not Pythagorean. 
Moreover, all the deformations of $\lambda_F$ arising from Jones' technology in \cite{Brothier-Jones19-HnonT} are not Pythagorean.
Here, by Jones' technology we mean the functorial method pioneered by Jones which permits to lift an action of a category to its fraction groups (Pythagorean representation arise this way but there are many other constructions).

We have seen that the Wysiwyg representations of Jones of $F$ are all irreducible and not Pythagorean. 
To find other such examples we may consider quasi-regular representations or more generally monomial representations arising from {\it self-commensurating} subgroups. 
Indeed, if $H\subset F$ is a self-commensurating subgroup and $\chi:H\to\bS_1$ a representation, then its associated induced representation $\Ind_H^F\chi$ is irreducible (and conversely $\Ind_H^F\chi$ is irreducible only if $H\subset F$ is self-commensurating, see the original theorem of Mackey and also the article of Burger and de la Harpe \cite{Mack51, Burger-Harpe97}).
Our classification result of \cite{BW23} (and results on infinite dimensional atomic representations that will be covered in a future article) asserts that the only irreducible monomial Pythagorean representations of $F$ arise from the parabolic subgroups $F_p:=\{g\in F:\ g(p)=p\}$ of $F$ where $p$ is a point of the Cantor space.
Hence, other self-commensurating subgroups yields irreducible representations of $F$ that are not Pythagorean such as the oriented Jones subgroup $\Vec F$ of $F$ (as proved by Golan and Sapir in \cite{Golan-Sapir}).
Note that the quasi-regular representation $\lambda_{F/\Vec F}$ associated to $\Vec F$ is part of an infinite family of Jones' representations built from the Temperley--Lieb--Jones planar algebras \cite{Aiello-Brothier-Conti}. 
Although, it is still unknown if these later representations are generically irreducible.
There are other examples of self-commensurating subgroups that can be constructed diagrammatically, see specific examples given by Aiello--Nagnibeda \cite{Aiello-Nagnibeda22}.

\noindent
{\bf Nota Bene.} As a final note we remark that all the examples cited of representations of the Thompson groups that are not Pythagorean are all constructed using Jones' technology.


\newcommand{\etalchar}[1]{$^{#1}$}

\end{document}